\documentclass[reqno,11pt]{amsart}

\numberwithin{equation}{section}
\usepackage{amsmath}
\usepackage{esint} 
\usepackage{cancel}
\usepackage{amsthm}
\usepackage{epsfig}
\usepackage{psfrag}
\usepackage{graphicx}
\usepackage{graphpap,latexsym,epsf}
\usepackage{color}
\usepackage{amssymb,mathrsfs,enumerate}
\usepackage{mathtools}
\usepackage{microtype}
\usepackage{subfigure}
\usepackage{setspace}
\definecolor{citegreen}{rgb}{0,0.6,0}
\definecolor{refred}{rgb}{0.8,0,0}
\usepackage[colorlinks, citecolor=citegreen, linkcolor=refred]{hyperref}
\newcommand{\R}{\mathbb{R}}
\newcommand{\N}{\mathbb{N}}
\newcommand{\Sph}{\mathbb{S}}

\def\HHH{{\rm H}}
\def\RRR{{\mathrm R}}

\def\a{\alpha}
\def\b{\beta}

\newcommand{\pa}{\partial}
\newcommand{\Om}{\Omega}
\newcommand{\ffi}{\varphi}

\newcommand{\ep}{\varepsilon}
\newcommand{\rmd}{{\rm d}}

\newcommand{\umax}{u_{{\rm max}}}
\newcommand{\mmax}{m_{{\rm max}}}

\newcommand{\go}{g_0}
\newcommand{\cgo}{g^{(0)}}

\newcommand{\Ric}{{\rm Ric}}

\newcommand{\D}{{\rm D}}
\newcommand{\DD}{{\rm D}^2}
\newcommand{\De}{\Delta}

\newcommand{\cho}{{\rm h}^{(0)}}
\newcommand{\Ho}{{\rm H}}
\newcommand{\Cr}{{\rm G}}

\newcommand{\g}{g}

\newcommand{\Ricg}{{\rm Ric}_g}

\newcommand{\Rg}{{\rm R}_g}
\newcommand{\na}{\nabla}
\newcommand{\nana}{\nabla^2}
\newcommand{\Deg}{\Delta_g}
\newcommand{\hg}{{\rm h}_g}
\newcommand{\chg}{{\rm h}^{(g)}}
\newcommand{\Hg}{{\rm H}_g}

\newcommand{\hhh}{{\rm h}}

\mathchardef\emptyset="001F

\usepackage{enumitem}
\setlist[description]{%
  topsep=1pt,               
  itemsep=1pt,               
  font={\normalfont\itshape\underline}, 
}

\definecolor{vgreen}{rgb}{0.1,0.5,0.2}
\definecolor{viola}{RGB}{85,26,139}

\newtheorem{theorem}{Theorem}[section]
\newtheorem{remark}{Remark}
\newtheorem{corollary}[theorem]{Corollary}
\newtheorem{definition}{Definition}
\newtheorem{proposition}[theorem]{Proposition}

\newtheorem{notation}{Notation}
\newtheorem{normalization}{Normalization}
\newtheorem{lemma}[theorem]{Lemma}

\newtheorem*{conjecture}{Conjecture}

\begin{document}

\title[On the mass of static metrics with positive cosmological constant -- II
]{On the mass of static metrics \\ with positive cosmological constant -- II}

\author[S.~Borghini]{Stefano Borghini}
\address{S.~Borghini, Uppsala Universitet, L{\"a}gerhyddsv{\"a}gen 1, 752 37 Uppsala, Sweden and Universit\`a degli Studi di Trento,
via Sommarive 14, 38123 Povo (TN), Italy}
\email{stefano.borghini@math.uu.se}

\author[L.~Mazzieri]{Lorenzo Mazzieri}
\address{L.~Mazzieri, Universit\`a degli Studi di Trento,
via Sommarive 14, 38123 Povo (TN), Italy}
\email{lorenzo.mazzieri@unitn.it}


\begin{abstract} 
This is the second of two works, in which we discuss the definition of an appropriate notion of mass for static metrics, in the case where the cosmological constant is positive and the model solutions are compact. In the first part, we have established a positive mass statement,
characterising the de Sitter solution as the only static vacuum metric with zero mass. 
In this second part, we prove optimal area bounds for horizons of black hole type and of cosmological type, corresponding to Riemannian Penrose inequalities and to cosmological area bounds \`a la Boucher-Gibbons-Horowitz, respectively. Building on the related rigidity statements, we also deduce a uniqueness result for the Schwarzschild--de Sitter spacetime. 
\end{abstract}

\maketitle


\smallskip
\noindent 

\date{\today}

\maketitle


\section{Introduction and statement of the main results}

In this paper we continue the study started in~\cite{Bor_Maz_2-I} about the notion of {\em virtual mass} of a static metric with positive cosmological constant. To make the exposition as much self-contained as possible, we briefly recall the basic notions and definitions.

\subsection{Setting of the problem and preliminaries.} 
\label{sub:prelim}

In this paper we consider {\em static vacuum metrics} in presence of a positive cosmological constant. These are given by triples $(M,\go, u)$ where $(M,\go)$ is an $n$-dimensional compact Riemannian manifold, $n \geq 3$, with nonempty smooth boundary $\pa M$, and $u \in {\mathscr C}^\infty  (M)$ is a smooth nonnegative function obeying to the following system
\begin{equation}
\label{eq:SES}
\begin{dcases}
u\,\Ric=\DD u+ \frac{2\Lambda}{n-1}\,u\,\go, & \mbox{in } M  ,\\
\ \;\, \De u=-\frac{2\Lambda}{n-1}\, u, & \mbox{in } M ,
\end{dcases}
\end{equation}
where $\Ric$, $\D$, and $\De$ represent the Ricci tensor, the Levi-Civita connection, and the Laplace-Beltrami operator of the metric $g_0$, respectively, and $\Lambda >0$ is a positive real number called {\em cosmological constant}.
We will always assume that the boundary $\pa M$ coincides with the zero level set of $u$, so that, in particular, $u$ is strictly positive in the interior of $M$. For more detailed discussions on the legitimacy of these assumptions, we refer the reader to~\cite{Ambrozio,Hij_Mon_Rau}. 
In the rest of the paper the metric $g_0$ and the function $u$ will be referred to as {\em static metric} and {\em static (or gravitational) potential}, respectively, whereas the triple $(M,g_0, u)$ will be called a {\em static solution}. For a more complete justification of this terminology as well as for some comments about the physical nature of the problem, we refer the reader to the introduction of~\cite{Bor_Maz_2-I} and the references therein. Here, we only recall that, having at hand a solution $(M, g_0, u)$ to~\eqref{eq:SES}, it is possible to recover a static solution $(X, \gamma)$ to the 
{\em vacuum Einstein field equations}
\smallskip
\begin{equation}
\label{eq:EFE}
\Ric_\gamma \,- \frac{\RRR_\gamma}{2} \, \gamma \,+ \Lambda \, \gamma\, = \, 0 \, , \quad \hbox{ in \,\, $\R \times M $} \, ,
\end{equation}
just by setting $X = \R \times M$ and letting $\gamma$ be the Lorentzian metric defined on $X$ by
\begin{equation*}
\gamma \, = \, - \,  u^2  dt \otimes dt \, + \, g_0 \, .
\end{equation*}

To complete the setup of our problem, we now list some of the basic properties of static solutions to system~\eqref{eq:SES}, whose proof can be found in~\cite[Lemma~3]{Ambrozio} as well as in the indicated references.

\begin{itemize}
\item Concerning the regularity of the function $u$, we know from~\cite{Chrusciel_1,ZumHagen} that $u$ is analytic. In particular, by the results in~\cite{Sou_Sou_1}, we have that its critical level sets are discrete. 
\smallskip
\item Since the manifold $M$ is compact, $\pa M = \{ u=0\}$ and $u>0$ in $M \setminus \pa M$, the static potential $u$ achieves its maximum in the interior of $M$. To fix the notation, we set
\begin{equation*}
 \umax \,\, = \,\, \max_M u  \qquad \hbox{and} \qquad {\rm MAX}(u)=\{p\in M \, : \, u(p)=\umax\} \, .
\end{equation*}
Since $u$ is analytic, one has that, according to~\cite{Lojasiewicz_2} (see also~\cite[Theorem~6.3.3]{Kra_Par}), the locus ${\rm MAX}(u)$ is a (possibly disconnected) stratified analytic subvariety whose strata have dimensions between $0$ and $n-1$. More precisely, it holds
$$
{\rm MAX}(u)\,=\,\Sigma^0\sqcup \Sigma^1\sqcup\dots\sqcup \Sigma^{n-1}\,,
$$
where $\Sigma^i$ is a finite union of $i$-dimensional analytic submanifolds, for every $i=0,\dots,n-1$. This means that, given a point $p\in\Sigma^i$, there exists a neighborhood $p\in\Omega\subset M$ and an analytic diffeomorphism $f:\Omega\to\R^n$ such that 
$$
f(\Omega\cap\Sigma^i)\,=\,L\cap f(\Omega)\,,
$$
for some $i$-dimensional linear space $L\subset\R^n$. In particular, the set $\Sigma^{n-1}$ is a smooth analytic hypersurface and it will play an important role in what follows. We will refer to the hypersurface $\Sigma^{n-1}$ as the {\em top stratum} of ${\rm MAX}(u)$.

\smallskip
\item 
Taking the trace of the first equation in~\eqref{eq:SES} and substituting the result into the second one, 
it is immediate to deduce that the scalar curvature of the metric $g_0$ is constant, and more precisely it holds
\begin{equation}
\label{eq:CSC}
\RRR=2\Lambda \, .
\end{equation}
In particular, we observe that choosing a normalization for the cosmological constant corresponds to fixing a scale for the metric $g_0$. Throughout the paper we will choose the following normalization
\begin{equation}
\label{eq:norma}
\Lambda \, = \, \frac{n(n-1)}{2} \, .
\end{equation}
So that in particular the manifold $(M, \go)$ will have constant scalar curvature $\RRR \equiv n(n-1)$.

\smallskip
\item 
The boundary $\pa M = \{u=0\}$, which is assumed to be a smooth submanifold of $M$, is also a regular level set of $u$. 
In particular it follows from the equations that it is a (possibly disconnected) totally geodesic hypersurface in $(M,\go)$. The connected components of $\pa M$ will be referred to as {\em horizons}. In Definition~\ref{def:horiz} below, we will distinguish between {\em horizons of black hole type}, {\em horizons of cosmological type} and {\em horizon of cylindrical type}. In order to simplify the exposition of some of the results in the paper, it is convenient to suppose that the manifold $M$ is orientable. This of course is not restrictive. In fact, if the manifold is not orientable, we can consider its orientable double covering, and transfer
the results obtained on this latter to the original manifold by means of the projection. 
We recall that an orientation of $M$ induces an orientation on the boundary $\pa M$, therefore, in particular, if $M$ is orientable so are the horizons.
\smallskip
\item Finally, one has that the quantity $|\D u|$ is locally constant and positive on $\pa M$.  
Notice that the value of $|\D u|$ at a horizon depends on the choice of the normalization of $u$. A more invariant quantity is the so called {\em surface gravity} of an horizon $S$, which can be defined as the constant 
\begin{equation}
\label{eq:surf_grav_normalization}
\kappa(S)\,\,=\,\,\frac{|\D u|_{|_S}}{\umax} \,,
\end{equation} 
where we recall that $\umax$ is the maximum of $u$ in $M$. For a more precise explaination of the physical motivations behind this definition, we refer the reader to~\cite{Bor_Maz_2-I}.
\end{itemize}

\noindent Recasting all the normalizations that we have introduced so far, we are led to study the following system
\begin{equation}
\label{eq:prob_SD}
\begin{dcases}
u\,\Ric=\DD u+n\,u\,g_0, & \mbox{in } M\\
\ \;\,\De u=-n\, u, & \mbox{in } M\\
\ \ \ \ \; u>0, & \mbox{in }  M\setminus\pa M \\
\ \ \ \ \; u=0, & \mbox{on } \pa M 
\end{dcases}
 \qquad  \hbox{with} \  M \ \hbox{compact orientable} \ \hbox{and} \ \RRR\equiv n(n-1)\, .
\end{equation}
This system is of course equivalent to~\eqref{eq:SES}, with some of the assumptions made more explicit. 
In this work, we are interested in the classification of static triples up to {\em isometry}, or at least up to a finite {\em covering}. Even though these notions are quite natural, we recall their precise definitions in the setting of static triples.
\begin{definition}
We say that two triples $(M,\go,u)$ and $(M',\go',u')$ are {\em isometric} if there exists a Riemannian isometry $F:(M,\go)\to (M',\go')$ such that, up to a normalization of $u$, it holds $u=u'\circ F$. 
We say that $(M,\go,u)$ is a {\em covering} of $(M',\go',u')$ if there exists a Riemannian covering $F:(M,\go)\to (M',\go')$ such that, up to a normalization of $u$, it holds $u=u'\circ F$. 
\end{definition} 
We conclude this subsection introducing some more terminology, whose meaning will be clarified in the next subsection by the detailed description of the rotationally symmetric solutions to~\eqref{eq:prob_SD}.
\begin{definition}
\label{def:horiz}
Let $(M, \go, u)$ be a solution to problem~\eqref{eq:prob_SD}. A connected component $S$ of $\pa M$ is called an {\em horizon}. An horizon is said to be:
\begin{itemize}
\item of {\em cosmological type} if: \qquad \,\ \ 
$\kappa(S)
\,<\, \sqrt{n}$, 
\item of {\em black hole type} if: \qquad\quad\;\;\  
$\kappa(S) 
\,>\, \sqrt{n}$, 
\item of {\em cylindrical type} if: \qquad \quad\,\;\
$\kappa(S)
\,=\, \sqrt{n}$
\end{itemize}
where $\kappa(S)$ is the surface gravity of $S$ defined in~\eqref{eq:surf_grav_normalization}.
A connected component $N$ of $M \setminus {\rm MAX}(u)$ is called {\em region} and we will denote by $\pa N$ the collection of the horizons of $M$ that lie in $N$, namely
$$
\pa N\,=\,\pa M\cap N\,.
$$
A region $N$ is said to be:
\begin{itemize}
\item an {\em outer region} if all of its horizons are of cosmological type, i.e., if 
$$ \max_{S \in\pi_0(\pa N)} \kappa(S) \,<\, \sqrt{n}\,,$$
\item an {\em inner region} if it has at least one horizon of black hole type, i.e., if
$$ \max_{S \in\pi_0(\pa N)} \kappa(S)  \,>\, \sqrt{n}\,,$$
\item a {\em cylindrical region} if there are no horizons of black hole type and there is at least one horizon of cylindrical type, i.e., if
$$ \max_{S \in\pi_0(\pa N)} \kappa(S)  \,=\, \sqrt{n}\,.$$
\end{itemize}

\end{definition}

\subsection{Rotationally symmetric solutions.}
\label{sub:rotsol} 
In this subsection, we briefly recall the rotationally symmetric solutions to~\eqref{eq:prob_SD}. These have three different qualitative behaviour, depending on the value of the mass parameter $m$, which is allowed to vary in the real interval $[0, \mmax]$, where 
\begin{equation}
\label{eq:mmax_SD}
\mmax=\sqrt{\frac{(n-2)^{n-2}}{n^n}} \, .
\end{equation}
We observe that if the number $\mmax$ is defined as above, then for every $0<m<\mmax$ the equation 
$f_m(r) = 0$, where $f_m(r)=1 - r^2 - 2m\,r^{2-n}$, has exactly two positive solutions $0 < r_-(m) < r_+(m)< 1$.
Moreover, in the interval $[r_-(m),r_+(m)]$ the function $f_m(r)$ assumes its maximum value at $r_0(m)=[(n-2)m]^{1/n}$.
For $m=0$, one has that $r_0(0)=r_-(0)=0$ and $r_+(0)=1$, whereas for $m=\mmax$, one has $r_0(\mmax)=r_-(\mmax) = r_+(\mmax) = [(n-2)/n]^{1/2}$.

\begin{figure}
 \centering
 \subfigure[de Sitter\label{fig:dS}]
   {\includegraphics[width=4.5cm]{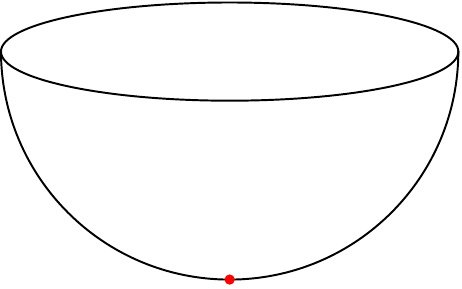}}
 \hspace{10mm}
 \subfigure[Schwarzschild--de Sitter\label{fig:SdS}]
   {\includegraphics[width=3.8cm]{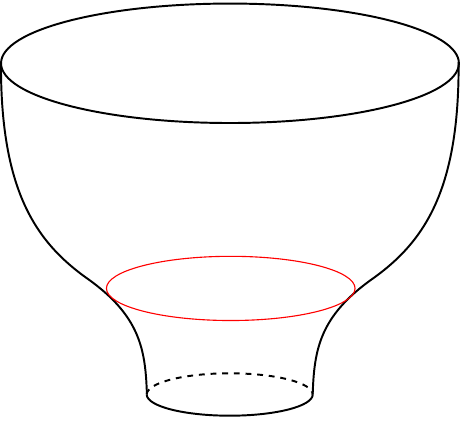}}
   \hspace{10mm}
 \subfigure[Nariai\label{fig:Nariai}]{\includegraphics[height=4cm,width=3cm]{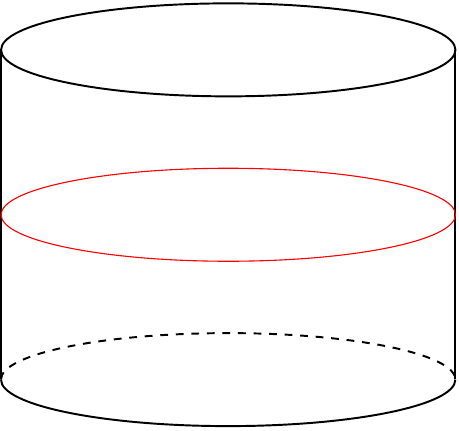}}
 \caption{\small Rotationally symmetric solutions to problem~\eqref{eq:prob_SD}. The red dot and red lines represent the set ${\rm MAX}(u)$ for the three models.
 \label{fig:rot_sym_sol}}
 \end{figure}

\begin{itemize}
\item \underline{de Sitter solution~\cite{DeSitter} ($m=0$), \figurename~\ref{fig:dS}.}
\begin{align}
\label{eq:D}
\nonumber M \, = \, \overline{B(0,1)}\subset\R^n \, ,  \qquad g_0 \, = \, \frac{d|x|\otimes d|x|}{1-|x|^2}+|x|^2 g_{\Sph^{n-1}} \, , \\
u \, = \, \sqrt{1-|x|^2} \, .\phantom{\qquad\qquad\qquad\qquad\qquad}
\end{align}
 It is not hard to check that both  the metric $g_0$ and the function $u$, which a priori are well defined only in the interior of $M\setminus \{ 0\}$, extend smoothly up to the boundary and through the origin. This model solution can be seen as the limit of the following Schwarzschild--de Sitter solutions~\eqref{eq:SD}, when the parameter $m \to 0^+$. The de Sitter solution is such that the maximum of the potential is $\umax=1$, and it is achieved at the origin. Moreover, this solution has only one connected horizon with surface gravity 
\begin{equation*}
|\D u| \,\, \equiv \,\, 1  \qquad \hbox{on} \quad \pa M \, .
\end{equation*} 
Hence, according to Definition~\ref{def:horiz} below,  this horizon is of cosmological type.  

\smallskip

\item \underline{Schwarzschild--de Sitter solutions~\cite{Kottler} ($0<m< \mmax$), \figurename~\ref{fig:SdS}.}
\begin{align}
\label{eq:SD}
\nonumber 
\phantom{\qquad}M \, = \, \overline{B(0,r_+(m))} \setminus B(0, r_-(m))  \subset\R^n \, ,  \qquad g_0 \, = \, \frac{d|x|\otimes d|x|}{1-|x|^2- 2m |x|^{2-n}}+|x|^2 g_{\Sph^{n-1}} \, , \\
u 
\, = \, \sqrt{1-|x|^2- 2m |x|^{2-n}}\, .
\phantom{\qquad\qquad\qquad\qquad\qquad\qquad}
\end{align}
Here $r_-(m)$ and $r_+(m)$ are the two positive solutions to $1-r^2-\!2mr^{2-n}=0$. We notice that, for $r_-(m),r_+(m)$ to be real and positive, one needs~\eqref{eq:mmax_SD}. It is not hard to check that both  the metric $g_0$ and the function $u$, which a priori are well defined only in the interior of $M$, extend smoothly up to the boundary. This latter has two connected components with different character 
\begin{equation*}
\pa M_+  = \,\, \{ |x| = r_+(m)\} \qquad \hbox{and} \qquad \pa M_-  = \,\, \{ |x| = r_-(m)\} \, .
\end{equation*}
In fact, it is easy to check (see formul\ae~\eqref{eq:k+} and~\eqref{eq:k-}) that  the normalized surface gravities satisfy
\begin{equation*}
\kappa(\pa M_+)\,\,=\,\,\frac{|\D u |_{|_{\pa M_+}}}{\umax} \,\, < \,\, \sqrt{n}    \qquad \hbox{and} \qquad \kappa(\pa M_-)\,\,=\,\,\frac{|\D u |_{|_{\pa M_-}}}{\umax} \,\, > \,\, \sqrt{n}  \, .
\end{equation*}
Hence, according to Definition~\ref{def:horiz} below, one has that $\pa M_+$ is of cosmological type, whereas $\pa M_-$ is of black hole type. Furthermore, it holds
\begin{equation}
\label{eq:umax_SD}
\umax\,=\,\sqrt{1-\left(\frac{m}{\mmax}\right)^{\!\!{2}/{n}}},\qquad {\rm MAX}(u)\,=\,\left\{|x|=r_0(m)\right\}\,,
\end{equation}
where we recall that $r_0(m)=[(n-2)m]^{1/n}$. Notice that $M \setminus {\rm MAX} (u)$ has exactly two connected components: $M_+$ with boundary $\pa M_+$ and $M_-$ with boundary $\pa M_-$. According to Definition~\ref{def:horiz}, we have that $M_+$ is an outer region, whereas $M_-$ is an inner region.

\begin{figure}
	\centering
	\includegraphics[scale=0.5]{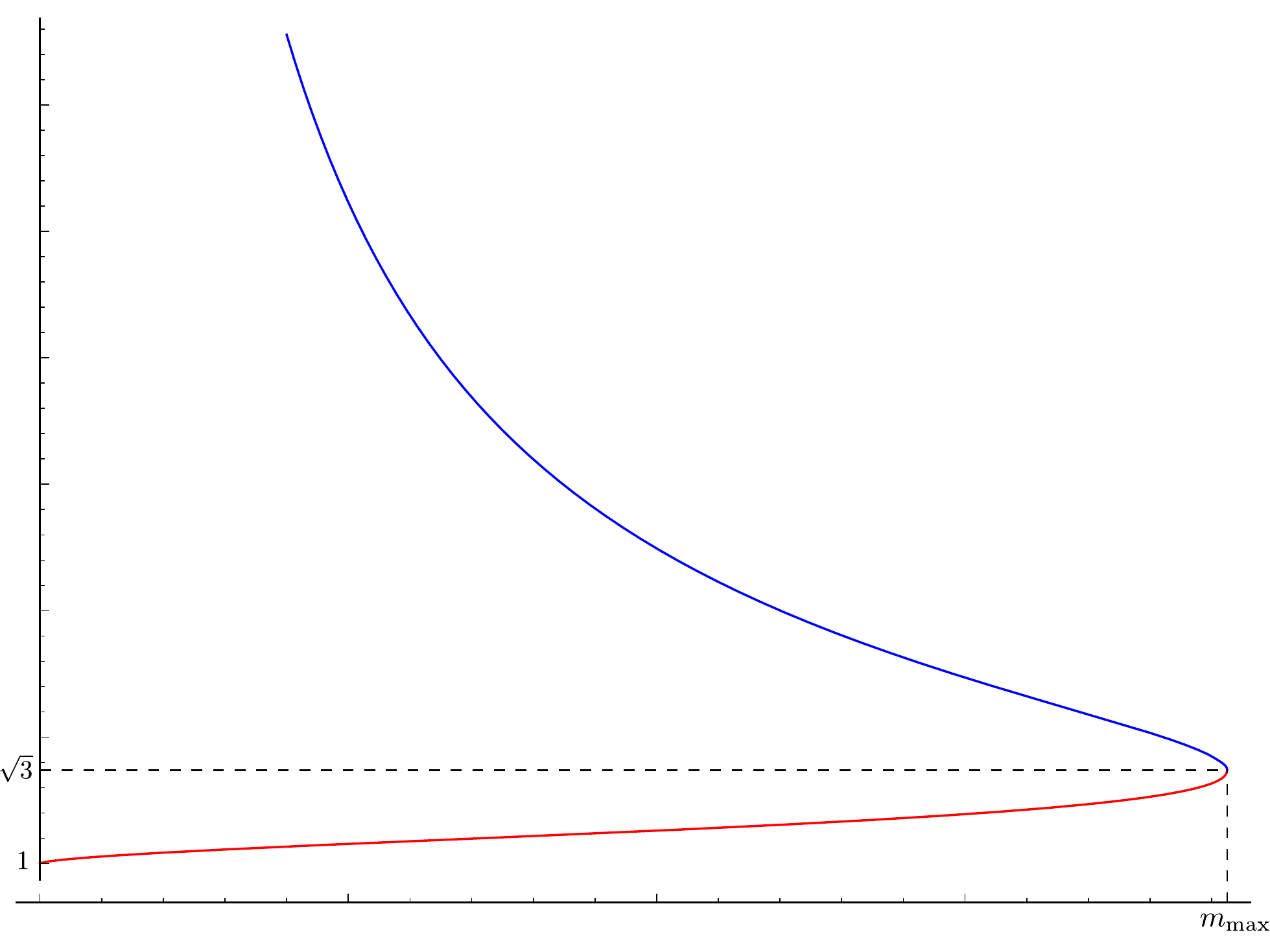}
	\caption{\small
Plot of the surface gravities $|\D u|/\umax$ of the two boundaries of the Schwarzschild--de Sitter solution~\eqref{eq:SD} as a function of the mass $m$ for $n=3$. The red line represents the surface gravity of the boundary $\pa M_+ = \{r = r_+(m)\}$, whereas the blue line represents the surface gravity of the boundary $\pa M_- = \{r = r_-(m)\}$. Notice that for $m=0$ we recover the constant value $|\D u| \equiv 1$ of the surface gravity on the (connected) cosmological horizon of the de Sitter solution~\eqref{eq:D}.
The other special situation is when $m=\mmax$. In this case the plot assigns to $\mmax= 1/(3 \sqrt{3})$ the unique value $\sqrt{3}$ achieved by the surface gravity on both the connected components of the boundary of the Nariai solution~\eqref{eq:cylsol_D}.
	} 
	\label{fig:surfacegravity_D}
\end{figure}

\smallskip

\item \underline{Nariai solution~\cite{Nariai} ($m=\mmax$), \figurename~\ref{fig:Nariai}.}
\begin{align}
\label{eq:cylsol_D}
\nonumber 
\phantom{\qquad\qquad}M \, = \, [0,\pi]\times\Sph^{n-1}\,, \qquad  g_0 \, = \, \frac{1}{n}\,\big[d r\otimes d r+(n-2)\,g_{\Sph^{n-1}}\big]\,, \\
u \, = \, \sin (r) \, .\phantom{\qquad\qquad\qquad\qquad\qquad\qquad\quad}
\end{align}
This model solution can be seen as the limit of the previous Schwarzschild--de Sitter solutions, when the parameter $m \to \mmax^-$, after an appropriate rescaling of the coordinates and potential $u$ (this was shown for $n=3$ in~\cite{Gin_Per} and then generalized to all dimensions $n\geq 3$ in~\cite{Car_Dia_Lem}, see also~\cite{Bousso,Bou_Haw}). In this case, we have $\umax=1$ and ${\rm MAX}(u)=\{\pi/2\}\times\Sph^{n-1}$. Moreover, the boundary of $M$ has two connected components with the same constant value of the surface gravity, namely
\begin{equation*}
|\D u| \,\, \equiv \,\, \sqrt{n} \qquad \hbox{on} \quad \pa M \, .
\end{equation*}
\end{itemize}

\noindent In Subsection~\ref{sub:surfmass}, we are going to use the above listed solutions as reference configurations in order to define the concept of {\em virtual mass} of a solution $(M, g_0, u)$ to~\eqref{eq:prob_SD}. To this aim, it is useful to introduce the functions $k_+$ and $k_-$, whose graphs are plotted, for $n=3$, in Figure~\ref{fig:surfacegravity_D}. They represent the normalized surface gravities of the model solutions as functions of the mass parameter $m$.
\begin{itemize}
\item The {\em outer surface gravity function}   
\begin{equation}
\label{eq:k+}
k_+  :  [\, 0, \mmax) \longrightarrow [\, 1, \sqrt{n} \, )
\end{equation}
is defined by
\begin{align*}
k_+(0) \,& = \, 1 \, , & \hbox{for $m=0$}\,, \phantom{\qquad\quad}  \\
k_+(m) \,&=\,
\sqrt{\frac{r_+^2(m)\left[1-\big(r_0(m)/r_+(m)\big)^n\right]^2}{1-\left( m / {\mmax}\right)^{2/n}}}\,, & \hspace{-1.5cm}\hbox{if $0<m<\mmax$} \, ,
\end{align*}
where $r_+(m)$ is the largest positive root of the polynomial $P_m(r) = r^{n-2}-r^n-2m$. Loosely speaking, $k_+(m)$ is nothing but the constant value of $|\D u|/\umax$ at $\{ |x| = r_+(m) \}$ for the Schwarzschild--de Sitter solution with mass parameter equal to $m$. We also observe that $k_+$ is continuous, strictly increasing and $k_+(m) \to \sqrt{n}$, as $m \to \mmax^-$.

\smallskip

\item The {\em inner surface gravity function}   
\begin{equation}
\label{eq:k-}
k_-  :  ( 0, \mmax \,] \longrightarrow [\, \sqrt{n}, +\infty \, )
\end{equation}
is defined by
\begin{align*}
k_-(\mmax) \,& = \, \sqrt{n} \, , & \hbox{for $m=\mmax$}\,, \phantom{\quad}  \\
k_-(m) \,&=\,
\sqrt{\frac{r_-^2(m)\left[1-\big(r_0(m)/r_-(m)\big)^n\right]^2}{1-\left( m / {\mmax}\right)^{2/n}}}\,, & \hspace{-1.5cm}\hbox{if $0<m<\mmax$} \, ,
\end{align*}
where $r_-(m)$ is the smallest positive root of the polynomial $P_m(r) = r^{n-2}-r^n-2m$. Loosely speaking, $k_-(m)$ is nothing but the constant value of $|\D u|/\umax$ at $\{ |x| = r_-(m) \}$ for the Schwarzschild--de Sitter solution with mass parameter equal to $m$. We also observe that $k_-$ is continuous, strictly decreasing and $k_-(m) \to + \infty$, as $m \to 0^+$.
\end{itemize}

\noindent 
This concludes the list of rotationally symmetric solutions. However, it is worth mentioning that in higher dimensions there is a simple generalization of the above model triples.
In fact, one can replace the spherical fibers in the Schwarzschild--de Sitter solution~\eqref{eq:SD} with any $(n-1)$-dimensional Einstein manifold $(E^{n-1},g_{E^{n-1}})$ with $\Ric_{E^{n-1}}=(n-2)g_{E^{n-1}}$. The resulting triple is still a solution to~\eqref{eq:prob_SD}, and it will be called {\em generalized Schwarzschild--de Sitter solution}
\begin{align}
\label{eq:gen_SD}
\nonumber 
\phantom{\qquad}M \, = \, [r_-(m),r_+(m)]\times E^{n-1} \, ,  \qquad \go \, = \, \frac{dr\otimes dr}{1-r^2- 2m r^{2-n}}+r^2 g_{E^{n-1}} \, , \\
u 
\, = \, \sqrt{1-r^2- 2m r^{2-n}}\, .
\phantom{\qquad\qquad\qquad\qquad\qquad\qquad}
\end{align}
Analogously, one can define the {\em generalized Nariai solution} as the triple
\begin{align}
\label{eq:gen_cylsol_D}
\nonumber 
\phantom{\qquad\qquad}M \, = \, [0,\pi]\times E^{n-1}\,, \qquad  g_0 \, = \, \frac{1}{n}\,\big[dr\otimes dr+(n-2)\,g_{E^{n-1}}\big]\,, \\
u \, = \, \sin (r) \, ,\phantom{\qquad\qquad\qquad\qquad\qquad\qquad\quad}
\end{align}
where, again, $(E^{n-1},g_{E^{n-1}})$ is an $(n-1)$-dimensional Einstein manifold with $\Ric_{E^{n-1}}=(n-2)g_{E^{n-1}}$.
Of course, the generalized solutions~\eqref{eq:gen_SD} and~\eqref{eq:gen_cylsol_D} are relevant only for $n\geq 5$, since for $n=3,4$ the only $(n-1)$-dimensional Einstein manifold with $\Ric_{E^{n-1}}=(n-2)g_{E^{n-1}}$ is the round sphere $(\Sph^{n-1},g_{\Sph^{n-1}})$.
We also mention that, exploiting a previous work of Bohm about the existence of 'non round' Einstein metrics on spheres~\cite{Bohm}, Gibbons, Hartnoll and Pope in~\cite{Gib_Har_Pop} were able to exhibit infinite families of solutions to problem~\eqref{eq:prob_SD}, in dimension $4\leq n\leq 8$. These solutions are such that their boundary is connected and diffeomorphic to a $(n-1)$-dimensional sphere. However, they do not have a warped product structure. This suggests that a complete classification of the solutions to problem~\eqref{eq:prob_SD} in dimension $n\geq 4$ is a very hard task. On the other hand, in dimension $n=3$, the only known  solutions are the de Sitter, Schwarzschild--de Sitter and Nariai triple. The question of whether these are the only ones is still open, although there are some partial results.
For instance, in~\cite{Kobayashi,Lafontaine} it is proven that these models are the only locally conformally flat  static metrics, in~\cite{Qin_Yua} this result has been extended to the Bach-flat case and in~\cite{daS_Bal} the case of cyclic parallel Ricci tensor has been discussed. Some pinching conditions implying the same classification are provided in~\cite{Ambrozio,Bal_Rib}.
Moreover, some further characterizations of the de Sitter metric have been proven in~\cite{Bou_Gib_Hor,Chrusciel_2,Hij_Mon_Rau}.

%

Since it will be of some importance in the forthcoming discussion, we conclude this section recalling the definition of Schwarzschild metric with mass parameter equal to $m>0$. This is the simplest (and also the early) example of a non flat static metric in the case where the cosmological constant in the Einstein Field Equations~\eqref{eq:EFE} is taken to be zero.
\begin{itemize}
\item \underline{Schwarzschild solutions~\cite{Schwarzschild} ($m>0$).}
\begin{align}
\label{eq:S}
\nonumber M \, = \, \R^n \setminus {B(0,r_s(m))}  \subset\R^n \, ,  \qquad g_0 \, = \, \frac{d|x|\otimes d|x|}{1- 2m |x|^{2-n}}+|x|^2 g_{\Sph^{n-1}} \, , \\
u \, = \, \sqrt{1- 2m |x|^{2-n}} \, .\phantom{\qquad\qquad\qquad\qquad\qquad}
\end{align}
Here, the so called Schwarzschild radius $r_s(m) = (2m)^{1/(n-2)}$ is the only positive solution to $1-2mr^{2-n}=0$. It is not hard to check that both  the metric $g_0$ and the function $u$, which a priori are well defined only in the interior of $M$, extend smoothly up to the boundary.
\end{itemize}


\subsection{The virtual mass.}
\label{sub:surfmass}

As already discussed in~\cite{Bor_Maz_2-I}, in the case of a positive cosmological constant there does not seem to be a general consensus about what the right notion of mass should be. For some possible approaches, as well as for more insights on the problems posed by the case $\Lambda>0$, we refer the reader to the following references~\cite{Abb_Des,Anninos,Ash_Bon_Kes,
Bal_DeB_Min,Chr_Jez_Kij,Kas_Tra,
Luo_Xie_Zha,Shiromizu,
Shi_Ida_Tor,Witten_book}.
In our previous work~\cite{Bor_Maz_2-I}, we have introduced a different point of view, leading to a new notion of mass, that we now recall.


\begin{definition}[Virtual Mass]
\label{def:virtual_mass} Let $(M, g_0, u)$ be a solution to~\eqref{eq:prob_SD} and let $N$ be a connected component of $M \setminus {\rm MAX} (u)$.
The virtual mass of $N$ is denoted by $\mu(N, g_0,u)$ and it is defined in the following way:
\begin{itemize}
\item[(i)] If $N$ is an outer region, then we set
\begin{equation}
\mu (N,g_0,u) \,\, = \,\, k_+^{-1} \left(  \max_{\pa N} \frac{|\D u |}{\umax}  \right) \, ,
\end{equation}
where $k_+$ is the outer surface gravity function defined in~\eqref{eq:k+}.
\smallskip
\item[(ii)] If $N$ is an inner region, then we set
\begin{equation}
\mu (N, g_0, u) \,\, = \,\, k_-^{-1} \left(  \max_{\pa N} \frac{|\D u |}{\umax}  \right)\, ,
\end{equation}
where $k_-$ is the inner surface gravity function defined in~\eqref{eq:k-}.
\end{itemize}
\end{definition}
\noindent In other words, the virtual mass of a connected component $N$ of $M\setminus {\rm MAX} (u)$ can be thought as the mass (parameter) that on a model solution would be responsible for (the maximum value of) the surface gravity measured at $\pa N$. In this sense the rotationally symmetric solutions described in Subsection~\ref{sub:rotsol} are playing here the role of reference configurations. As it is easy to check, if $(M, g_0, u)$ is either the de Sitter, or the Schwarzschild--de Sitter, or the Nariai  solution, then the virtual mass coincides with the explicit mass parameter $m$ that appears in Subsection~\ref{sub:rotsol}.

It is important to notice that it is not {\em a priori} guaranteed that the above definition is well posed. In fact, it could happen that the boundary of a connected component is empty or that the value of the normalized surface gravity does not lie in the range of either $k_+$ or $k_-$. The first possibility can be easily excluded arguing as in the No Island Lemma (see~\cite[Lemma 5.1]{Bor_Maz_2-I}), whereas to exclude the second possibility we need to invoke~\cite[Theorem~2.2]{Bor_Maz_2-I}. This result tells us that, on any region $N$ of a solution $(M,\go,u)$, it holds
\begin{equation*}
\max_{S\in\pi_0(\pa N)}\kappa(S)\,\,=\,\,\max_{\pa N}\frac{|\D u|}{\umax} \,\, \geq \,\, 1  \, ,
\end{equation*} 
and the equality is fulfilled only if $(M,\go,u)$ is isometric to the de Sitter solution~\eqref{eq:D}.
As an immediate consequence we obtain the following Positive Mass Statement for static metrics with positive cosmological constant.

\begin{theorem}[Positive Mass Statement for Static Metrics with Positive Cosmological Constant]
\label{thm:PMS}
Let $(M,\go,u)$ be a solution to problem~\eqref{eq:prob_SD}. Then, every connected component of $M\setminus{\rm MAX}(u)$ has well--defined and thus nonnegative virtual mass. Moreover, as soon as the virtual mass of some connected component vanishes, the entire solution $(M, \go,u)$ is isometric to the de Sitter solution~\eqref{eq:D}.
\end{theorem}

\noindent
We refer the reader to~\cite{Bor_Maz_2-I} for a more detailed discussion about the above statement as well as for a comparison with the classical Positive Mass Theorem proved by Schoen and Yau~\cite{Sch_Yau,Sch_Yau_2} (and with a different proof by Witten~\cite{Witten}) for the ADM-mass of asymptotically flat manifolds with nonnegative scalar curvature.

\subsection{Area bounds.} 
\label{sub:area_bounds_intro}
An important feature of the above positive mass statement is that it gives a complete characterisation of the zero mass solutions. Another very interesting and nowadays classical  characterisation of the de Sitter solution is given by the Boucher-Gibbons-Horowitz area bound~\cite{Bou_Gib_Hor}, which in our framework can be phrased as follows
\begin{theorem}[Boucher-Gibbons-Horowitz Area Bound]
\label{thm:BGH}
Let $(M^3,\go,u)$ be a $3$-dimensional solution to problem~\eqref{eq:prob_SD} with connected boundary $\pa M$. Then, the following inequality holds
\begin{equation}
\label{eq:BGH}
|\pa M| \, \leq \, 4 \pi \,.
\end{equation}
Moreover, the equality is fulfilled if and only if $(M^3, g_0, u)$ is isometric to the de Sitter solution~\eqref{eq:D}.
\end{theorem}
Having at hand Theorem~\ref{thm:PMS} and Theorem~\ref{thm:BGH}, it is natural to ask if in the case where the virtual mass is strictly positive and the boundary of $M$ is allowed to have several connected components, it is possible to provide a refined version of both statements, whose rigidity case characterises now the Schwarzschild--de Sitter solutions described in~\eqref{eq:SD} instead of the de Sitter solution. In accomplishing this program, we are inspired by the well known relation between the Positive Mass Theorem and the Riemannian Penrose Inequality as they are stated in the classical setting, where $M^3$ is an asymptotically flat Riemannian manifold with nonnegative scalar curvature. To be more concrete, we report a simplified version of these statements in the case where the $3$-manifold has one end and at most one compact horizon.

\begin{theorem}
Let $(M^3, \go)$ be a $3$-dimensional complete asymptotically flat Riemannian manifold with nonnegative scalar curvature and ADM-mass $m_{ADM}(M^3,\go)$ equal to $m\in \R$. Then, the following statements hold.
\smallskip
\begin{itemize}
\item[(i)] \underline{Positive Mass Theorem (Schoen-Yau~\cite{Sch_Yau,Sch_Yau_2}, Witten~\cite{Witten}).} The number $m$ is always nonnegative
\begin{equation*}
0 \, \leq \, m \, .
\end{equation*}
Moreover, the equality is fulfilled if and only if $(M^3, g_0, u)$ is isometric to the flat Euclidean space with $u \equiv 1$.
\smallskip
\item[(ii)] \underline{Riemannian Penrose Inequality (Huisken-Ilmanen~\cite{Hui_Ilm}, Bray~\cite{Bray}).} Assume that the boundary of $M$ is non empty and given by a connected, smooth and compact outermost minimal surface. Then, 
the following inequality holds
\begin{equation}
\label{eq:RPI}
\sqrt{\frac{|\pa M|}{16 \pi}} \, \leq \, m \,.
\end{equation}
Moreover, the equality is fulfilled if and only if $(M^3, g_0, u)$ is isometric to the Schwarzschild solution~\eqref{eq:D} with mass parameter equal to $m$.
\end{itemize}
\end{theorem}
For the precise definitions of {\em asymptotically flat manifold} and {\em ADM-mass}, we refer the reader to the above cited references.
We also observe that in the original statement of the Positive Mass Theorem, the $3$-manifold $(M^3, g)$ is {\em a priori} allowed to have a finite number of ends and that the rigidity statement holds in a stronger way, meaning that 
as soon as the mass of one end is vanishing, then the whole manifold is isometric to the Euclidean space. Concerning the Riemannian Penrose Inequality, it is worth pointing out that in the original statement by Huisken and Ilmanen~\cite[Main Theorem]{Hui_Ilm}, the boundary of $M$ is {\em a priori} allowed to have a finite number of connected component, namely $\pa M = S_0 \, \sqcup\, S_1\, \sqcup\, \ldots \, \sqcup\,S_K $, and the authors are able to prove the following inequality 
\begin{equation*}
\sqrt{\frac{\max_{0\leq j \leq K}|S_j|}{16 \pi}} \, \leq \, m \, ,
\end{equation*}
where $m = m_{ADM}(M^3,g)$. With a different proof, Bray is able to recover in~\cite[Theorem~1]{Bray} a stronger version of the above inequality, namely
\begin{equation*}
\sqrt{\frac{|S_0| + \ldots+ |S_K|}{16 \pi}} \, \leq \, m \, .
\end{equation*}
Of course, when $\pa M$ is connected, the two inequalities are the same and they reduce to~\eqref{eq:RPI}. To introduce our first main result, we focus on this simple version of the Riemannian Penrose Inequality and we observe that, using the definition of the Schwarzschild radius given below formula~\eqref{eq:S}, it can be rephrased as follows
\begin{equation*}
|\pa M| \, \leq \, 16 \pi m^2 \, = \, 4\pi (2m)^2 \, = \, 4 \pi r_0^2(m) \, ,
\end{equation*}
where $m = m_{ADM}(M^3,g)$. Having these considerations in mind, we can now state one of the main results of the present paper.

%
%
%
%

\begin{theorem}[Refined Area Bounds]
\label{thm:area_bound_SD}
Let $(M^3,\go,u)$ be a $3$-dimensional solution to problem~\eqref{eq:prob_SD} and let $N$ be a connected component of $M^3\setminus{\rm MAX}(u)$ with virtual mass
\begin{equation*}
m \, = \, \mu(N, \go , u)\,\in\, \left(0, 1/(3\sqrt{3})\right]\,.
\end{equation*}
Let $S\subseteq\pa N$ be the horizon with the largest surface gravity in $N$, namely
\begin{equation*}
\kappa(S) \, = \, 
\begin{dcases}
k_+(m) & \hbox{ if $N$ is outer}\,,
\\
k_-(m) & \hbox{ if $N$ is inner}\,,
\\
\sqrt{n} & \hbox{ if $N$ is cylindrical}\,.
\end{dcases}
\end{equation*}
Then, $S$ is diffeomorphic to the sphere $\Sph^2$. Moreover, the following inequalities hold:
\begin{itemize}
\smallskip
\item[(i)] \underline{Cosmological Area Bound.} If $N$ is an outer region, then
\begin{equation}
	\label{eq:area_bound_+_SD}
	|S|\,\leq\,4\pi r_+^{2}(m)\,.
	\end{equation}
Moreover, if the equality is fulfilled and $S=\pa N$, then the triple $(M^3,\go ,u)$ is isometric to the Schwarzschild--de Sitter  solution~\eqref{eq:SD} with mass $m$.
\smallskip
\item[(ii)] \underline{Riemannian Penrose Inequality.} If $N$ is an inner region, then
\begin{equation}
	\label{eq:area_bound_-_SD}
	|S|\,\leq\,4\pi r_-^{2}(m)\,.
	\end{equation}
Moreover, if the equality is fulfilled and $S=\pa N$, then the triple $(M^3,\go ,u)$ is isometric to the Schwarz\-schild--de Sitter solution~\eqref{eq:SD} with mass $m$.
\smallskip
\item[(iii)] \underline{Cylindrical Area Bound.} If $N$ is a cylindrical region, then
	\begin{equation}
	\label{eq:area_bound_N_SD}
	|S|\,\leq\,\frac{4\pi}{3}\,,
	\end{equation}
Moreover, if the equality is fulfilled and $S=\pa N$, then the triple $(M^3,\go ,u)$ is covered by the Nariai solution~\eqref{eq:cylsol_D}.
\end{itemize}

\end{theorem}


\begin{remark}
Notice that the rigidity statements are only in force when $\pa N$ is connected. Concerning the rigidity statement in point (iii) of the above theorem, we observe that there is only one orientable triple which is not isometric to the $3$-dimensional Nariai solution but that is covered by it, which is the quotient of the Nariai triple by the involution
$$
\iota:\,[0,\pi]\times\Sph^2\to [0,\pi]\times\Sph^2\,,\qquad \iota(t,x)\,=\,(\pi-t,-x)\,,
$$
where we have denoted by $-x$ the antipodal point of $x$ on $\Sph^2$. The existence of this solution was pointed out in~\cite[Section~7]{Ambrozio}.
\end{remark}

About the previous statement some comments are in order. First, the fact that $S$ is necessarily diffeomorphic to a sphere is not a new result. In fact, a stronger result is already known from~\cite[Theorem~B]{Ambrozio}, where it is shown that every connected component of the boundary of a static solution to problem~\eqref{eq:prob_SD} is diffeomorphic to a sphere. Our approach allows to prove the same topological result, but only in the case where the horizons of $(M^3, g_0, u)$ are somehow separated from each other by the locus ${\rm MAX}(u)$. Concerning the area bounds, we observe that, conceptually speaking, the inequality~\eqref{eq:area_bound_+_SD} should be compared with the Boucher-Gibbons-Horowitz Area Bound~\eqref{eq:BGH}, since it involves the cosmological horizons of the solution, whereas the inequality~\eqref{eq:area_bound_-_SD} should be compared with~\eqref{eq:RPI} since it is a statement about horizons of black hole type.

An analogous statement holds in higher dimension, giving the natural analog of the inequality 
\begin{equation}
\label{eq:Chru}
|\pa M| \, \leq \, \int_{\pa M}
	\frac{\RRR^{\pa M}}{(n-1)(n-2)}\,\rmd\sigma \,,
\end{equation}
which has been obtained by Chru\`sciel in~\cite[Section~6]{Chrusciel_2} in the case of connected boundary, extending the Boucher-Gibbons-Horowitz method to every dimension $n\geq3$. Of course, in the above inequality $\RRR^{\pa M}$ stands for the scalar curvature of the boundary. Moreover, the equality is fulfilled if and only if $(M,g_0,u)$ coincides with the de Sitter solution.

\begin{theorem}
	\label{thm:scalarcurvature_bound_SD}
Let $(M,\go,u)$ be a solution to problem~\eqref{eq:prob_SD} of dimension $n\geq 3$, and let $N$ be  a connected component of $M\setminus{\rm MAX}(u)$ with connected smooth compact boundary $\pa N$. We then let $m \in (0, \mmax]$ be the virtual mass of $N$, namely
\begin{equation*}
m \, = \, \mu(N, g_0, u)\,.
\end{equation*}
Let $S\subseteq\pa N$ be the horizon with the largest surface gravity in $N$, namely
\begin{equation*}
\kappa(S) \, = \, 
\begin{dcases}
k_+(m) & \hbox{ if $N$ is outer}\,,
\\
k_-(m) & \hbox{ if $N$ is inner}\,,
\\
\sqrt{n} & \hbox{ if $N$ is cylindrical}\,.
\end{dcases}
\end{equation*}
Then, the following inequalities hold:\begin{itemize}
\smallskip
\item[(i)] If $N$ is an outer region, then
\begin{equation}
    \label{eq:scalarcurvature_bound_+_SD}
	|S|\,\leq\, \left(\int_{S}
	\frac{\RRR^{S}}{(n-1)(n-2)}\,\rmd\sigma \right)  \, r_+^{2}(m)\,,
	\end{equation}
Moreover, if the equality is fulfilled and $S=\pa N$, then $(M,g_0,u)$ is isometric to the Schwarzschild--de Sitter solution~\eqref{eq:SD} with mass $m$, for $n=3,4$. Whereas for $n\geq 5$ it is isometric to some generalized Schwarzschild--de Sitter solution~\eqref{eq:gen_SD} with Einstein fiber.
\smallskip
\item[(ii)] If $N$ is an inner region, then
\begin{equation}
	\label{eq:scalarcurvature_bound_-_SD}
	|S|\,\leq\, \left(\int_{S}
	\frac{\RRR^{S}}{(n-1)(n-2)}\,\rmd\sigma \right)  \, r_-^{2}(m)\,.
	\end{equation}
Moreover, if the equality is fulfilled and $S=\pa N$, then $(M,g_0,u)$ is isometric to the Schwarzschild--de Sitter solution~\eqref{eq:SD} with mass $m$, for $n=3,4$. Whereas for $n\geq5$ it is isometric to some generalized Schwarzschild--de Sitter solution~\eqref{eq:gen_SD} with Einstein fiber.

\smallskip
\item[(iii)] If $N$ is a cylindrical region, then
	\begin{equation}
	\label{eq:scalarcurvature_bound_N_SD}
	|S|\,\leq\, \int_{S}\frac{\RRR^{S}}{n(n-1)}\,\rmd\sigma\,.
	\end{equation}
Moreover, if the equality is fulfilled and $S=\pa N$, then $(M,\go,u)$ is covered by the Nariai solution~\eqref{eq:cylsol_D}, for $n=3,4$. Whereas for $n\geq 5$ it is covered by some generalized Nariai solution~\eqref{eq:gen_cylsol_D} with Einstein fiber.
\end{itemize}
\end{theorem}
The proof of the above statement will be given in Section~\ref{sec:area_bounds}, except for the rigidity statements, whose proof will be discussed in Section~\ref{sec:consequences_SD}, and for the cylindrical case, that will be discussed in Section~\ref{sec:nariai_SD}. It is clear that Theorem~\ref{thm:area_bound_SD} follows directly from Theorem~\ref{thm:scalarcurvature_bound_SD}, applying the Gauss-Bonnet formula. We also mention that the rigidity statement for Theorem~\ref{thm:scalarcurvature_bound_SD} will be deduced by some more general statements (see Corollaries~\ref{cor:intid_rewr_+_SD},~\ref{cor:intid_rewr_-_SD} and~\ref{cor:intid_rewr_N_SD})
which correspond to some balancing formulas, in the case where the boundary of $N$ is allowed to have several connected components.
The inequalities proven in Theorem~\ref{thm:scalarcurvature_bound_SD} share some analogies with the ones developed in~\cite{Fan_Yua,Yuan}, see in particular~\cite[Theorem~B]{Yuan}.

Our approach will also allow us to prove some area lower bounds on the horizons. These lower bounds do not require the connectedness of the boundary of our region $N$ and depend on the area of the hypersurface $\Sigma_N\subseteq{\rm MAX}(u)$ that separates $N$ from the rest of the manifold.

\begin{theorem}[Area Lower Bound]
\label{thm:lower_bound}
Let $(M,\go,u)$ be a solution to problem~\eqref{eq:prob_SD} of dimension $n\geq 3$, and let $N$ be  a connected component of $M\setminus{\rm MAX}(u)$ with connected smooth compact boundary $\pa N$. We let $m \in (0, \mmax]$ be the virtual mass of $N$, namely
\begin{equation*}
m \, = \, \mu(N, g_0, u)\,.
\end{equation*}
Let $\Sigma_N=\overline{N}\cap\overline{M\setminus \overline{N}}$ be the possibly stratified hypersurface separating $N$ from the rest of the manifold $M$.
Then, the following inequalities hold:\begin{itemize}
\smallskip
\item[(i)] If $N$ is an outer region, then
\begin{equation}
\label{eq:lower_bound_+_SD}
|\pa N|\,\geq\, \left[\frac{r_+(m)}{r_0(m)}\right]^{n-1}\,|\Sigma_N|\,,
\end{equation}
and the equality is fulfilled if and only if $(M,g_0,u)$ is isometric to the Schwarzschild--de Sitter solution~\eqref{eq:SD} with mass $m$, for $n=3,4$. Whereas for $n\geq 5$ it is isometric to some generalized Schwarzschild--de Sitter solution~\eqref{eq:gen_SD} with Einstein fiber.
\smallskip
\item[(ii)] If $N$ is an inner region, then
\begin{equation}
\label{eq:lower_bound_-_SD}
|\pa N|\,\geq\, \left[\frac{r_-(m)}{r_0(m)}\right]^{n-1}\,|\Sigma_N|\,,
\end{equation}
and the equality is fulfilled if and only if $(M,g_0,u)$ is isometric to the Schwarzschild--de Sitter solution~\eqref{eq:SD} with mass $m$, for $n=3,4$. Whereas for $n\geq5$ it is isometric to some generalized Schwarzschild--de Sitter solution~\eqref{eq:gen_SD} with Einstein fiber.

\smallskip
\item[(iii)] If $N$ is a cylindrical region, then
\begin{equation}
\label{eq:lower_bound_N_SD}
|\pa N|\,\geq\, |\Sigma_N|\,,
\end{equation}
and the equality is fulfilled if and only if $(M,\go,u)$ is covered by the Nariai solution~\eqref{eq:cylsol_D}, for $n=3,4$. Whereas for $n\geq 5$ it is covered by some generalized Nariai solution~\eqref{eq:gen_cylsol_D} with Einstein fiber.
\end{itemize}
\end{theorem}

In the notations of Theorem~\ref{thm:lower_bound}, if we also assume that $\pa N$ is connected we can combine the lower and upper bounds proved in Theorems~\ref{thm:area_bound_SD} and~\ref{thm:lower_bound} to obtain an area lower bound for the hypersurface $\Sigma_N$. The general statement of this result is given in Theorem~\ref{thm:area_bound_dischyp_gen}. Here we report the special $3$-dimensional case, in which the bound turns out to be particularly nice.

\begin{corollary}
\label{cor:lower_bound}
Let $(M,\go,u)$ be a $3$-dimensional solution to problem~\eqref{eq:prob_SD}, and let $N$ be  a connected component of $M\setminus{\rm MAX}(u)$ with connected smooth compact boundary $\pa N$. We let $m \in (0, \mmax]$ be the virtual mass of $N$, namely
\begin{equation*}
m \, = \, \mu(N, g_0, u)\,.
\end{equation*}
Let $\Sigma_N=\overline{N}\cap\overline{M\setminus \overline{N}}$ be the possibly stratified hypersurface separating $N$ from the rest of the manifold $M$. Then it holds
$$
|\Sigma_N|\,\leq\,4\,\pi\, r_0^2(m)\,,
$$
and the equality is fulfilled if and only if $(M,\go,u)$ is either isometric to the Schwarzschild--de Sitter solution~\eqref{eq:SD} with mass $0<m<\mmax$ or $(M,\go,u)$ is covered by the Nariai solution~\eqref{eq:cylsol_D}.
\end{corollary}

We conclude this subsection with a comparison of our Theorem~\ref{thm:area_bound_SD} with the following recent result due to Ambrozio.

\begin{theorem}[{\cite[Theorem~C]{Ambrozio}}]
\label{thm:area-bound_Ambrozio}
Let $(M,\go,u)$ be a $3$-dimensional solution to problem~\eqref{eq:prob_SD}, let $S_0,\dots, S_p$ be the connected components of $\pa M$ and let $\kappa_0,\dots,\kappa_p$ be their surface gravities. If $(M,\go,u)$ is not isometric to the de Sitter solution~\eqref{eq:S}, then 
\begin{equation}
\label{eq:area-bound_Ambrozio}
\frac{\sum_{i=0}^p \kappa_i|S_i|}{\sum_{i=0}^p\kappa_i}\,\leq\,\frac{4\pi}{3}\,.
\end{equation}
Moreover, if the equality holds, then $(M,\go,u)$ is isometric to the Nariai solution~\eqref{eq:cylsol_D}.
\end{theorem}

\noindent
Of course Ambrozio's result is slightly different from ours under certain aspects, as Theorem~\ref{thm:area-bound_Ambrozio} does not require any assumption on ${\rm MAX}(u)$ and has a global nature, whereas our Theorem~\ref{thm:area_bound_SD} uses the locus ${\rm MAX}(u)$ to decompose the manifold into several connected components and provides on each of these components a (local) weighted inequality in the spirit of the above~\eqref{eq:area-bound_Ambrozio}. Let us compare the two statements in a couple of special cases. First of all, if our solution $(M,\go,u)$ has a single horizon and it is not isometric to the de Sitter solution, then Theorem~\ref{thm:area-bound_Ambrozio} gives
$$
|\pa M|\,\,\leq\,\,\frac{4\pi}{3}\,,
$$
which is a neat improvement of the classical Boucher-Gibbons-Horowitz inequality~\eqref{eq:BGH}. In this respect, our Theorem~\ref{thm:area_bound_SD} gives the same inequality if the horizon is of cylindrical type, a stronger inequality when the horizon is of black hole type and a worse result if the horizon is of cosmological type. 

\begin{figure}
\centering
\includegraphics[scale=0.5]{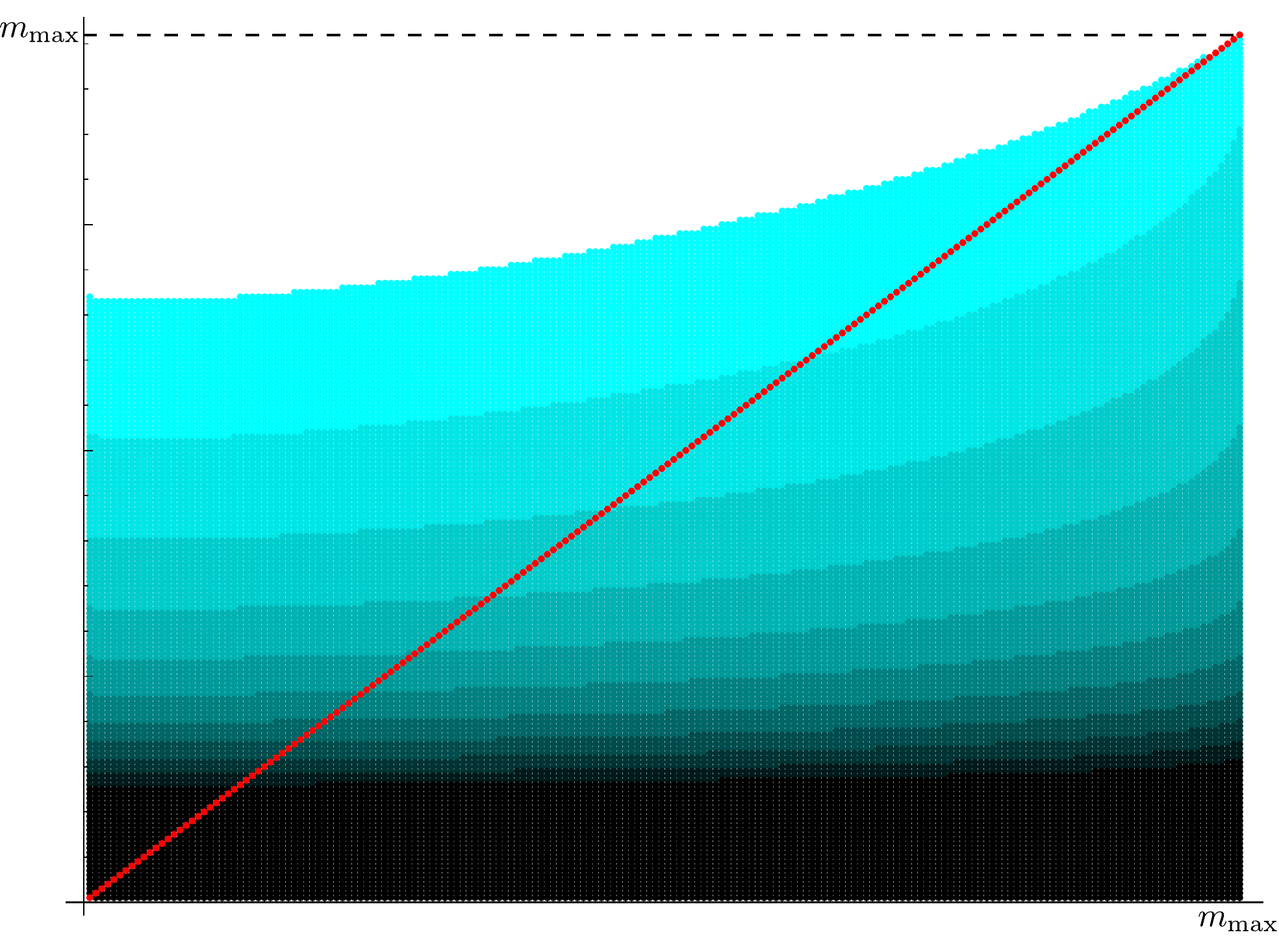}
\caption{\small
In this plot we have numerically analyzed the relation between formul\ae~\eqref{eq:special_Ambrozio} and~\eqref{eq:special_Us}, in function of the values of $m_+$ (on the $x$-axis) and of $m_-$ (on the $y$-axis). The red line represents the points where $m_+=m_-$, so that the Schwarzschild--de Sitter solutions lie on this line. The coloured region is the one where~\eqref{eq:special_Us} is stronger than~\eqref{eq:special_Ambrozio}. The darker the colour, the better our formula is. To give also a quantitative idea, the black region at the bottom is where the difference between the right hand side of~\eqref{eq:special_Ambrozio} and the right hand side of~\eqref{eq:special_Us} is greater than $3$.
} 
\label{fig:ambro}
\end{figure}

Let us now compare the two statements in the case upon which our result is modelled, that is, suppose that our solution $(M,\go,u)$ is such that
$$
M\setminus{\rm MAX}(u)\,=\,M_+\sqcup M_-\,,
$$
where $M_+$ is an outer region with connected boundary $\pa M_+$ and $M_-$ is an inner region with connected boundary $\pa M_-$. If we denote by
$$
m_+\,=\,\mu(M_+,\go,u)\,,\qquad m_-\,=\,\mu(M_-,\go,u)\,,
$$
the virtual masses of $M_+$ and $M_-$, then inequality~\eqref{eq:area-bound_Ambrozio} in Theorem~\ref{thm:area-bound_Ambrozio} writes as
\begin{equation}
\label{eq:special_Ambrozio}
k_+(m_+)\,|\pa M_+|\,+\, k_-(m_-)\,|\pa M_-|\,\,\leq\,\, \frac{4\pi}{3}\left[k_+(m_+)\,+\,k_-(m_-)\right]\,. \phantom{\qquad\qquad \,\,\,\,\,}
\end{equation}
On the other hand, inequalities~\eqref{eq:area_bound_+_SD} and~\eqref{eq:area_bound_-_SD} in Theorem~\ref{thm:area_bound_SD} give
\begin{equation}
\label{eq:special_Us}
\phantom{\qquad\qquad\,\,}k_+(m_+)\,|\pa M_+|\,+\, k_-(m_-)\,|\pa M_-|\,\,\leq\,\, 4\pi\left[k_+(m_+)r_+^2(m_+)\,+\,k_-(m_-)r_-^2(m_-)\right]\,.
\end{equation}
The two inequalities~\eqref{eq:special_Ambrozio},~\eqref{eq:special_Us} are compared in \figurename~\ref{fig:ambro}, where we have highlighted the values of $m_+,m_-$ for which our formula~\eqref{eq:special_Us} improves~\eqref{eq:special_Ambrozio}.
This comparison suggests that our result is particularly effective when the set ${\rm MAX}(u)$ separates the manifold into an outer region and an inner one, and motivates in turn our definition of a $2$-sided solution to problem~\eqref{eq:prob_SD} (see Definition~\ref{def:2-sided} below), providing us with the natural setting for the uniqueness statement described in the next subsection.

\subsection{Uniqueness results.}
\label{sub:BHU_intro}
In this subsection, we discuss a characterization of both the Schwarz\-schild--de Sitter and the Nariai solution, which is in some ways reminiscent of the well known Black Hole Uniqueness Theorem proved in different ways by Israel~\cite{Israel}, zum Hagen, Robinson and Seifert~\cite{ZHa_Rob_Sei}, Robinson~\cite{Robinson}, Bunting and Masood-ul Alam~\cite{Bun_Mas} and recently by the second author in collaboration with Agostiniani~\cite{Ago_Maz_2}.
This classical result states that when the cosmological constant is zero, the only asymptotically flat static solutions with nonempty boundary are the Schwarzschild triples described in~\eqref{eq:S}. In order to clarify what should be expected to hold in the case of positive cosmological constant, let us briefly comment the {\em asymptotic flatness} assumption.
Without discussing the physical meaning of this assumption nor reporting its precise definition -- which on the other hand can be easily found in the literature -- we underline 
the fact that it amounts to both a topological and a geometric requirement. More precisely, each end of the manifold is {\em a priori} forced to be diffeomorphic to $[ \,0,\! + \infty) \times \Sph^{n-1}$ and the metric has to converge to the flat one at a suitable rate, so that, up to a convenient rescaling, the boundary at infinity of the end is isometric to a round sphere. Another important feature of the {\em asymptotic flatness} assumption is that the 
static potential approaches its maximum value at infinity. 

From this last property, it seems natural to guess that the boundary at infinity of an asymptotically flat static solution with $\Lambda = 0$ should correspond in our framework to the set ${\rm MAX}(u)$. The same analogy is also proposed in~\cite[Appendix]{Bou_Haw}, where it is used to justify the physical meaning of the normalization~\eqref{eq:surf_grav_normalization} for the surface gravity. 
Before presenting the precise statement of this uniqueness result, it is important to underline another feature of the set ${\rm MAX}(u)$, that is peculiar of our setting. In fact, in sharp contrast with the $\Lambda = 0$ case, we observe that ${\rm MAX}(u)$ may in principle disconnect our manifold. On the other hand, this situation is not only possible but even natural, since it is realized in the model examples given by the Schwarzschild--de Sitter solutions~\eqref{eq:SD} and the Nariai solutions~\eqref{eq:cylsol_D}. Here, the set ${\rm MAX}(u)$ separates the manifold into two regions, one of which is either outer or cylindrical, while the other is either inner or cylindrical. Having this in mind, it is natural to introduce the notion of a $2$-sided solution to problem~\eqref{eq:prob_SD}.

\begin{definition}[2-Sided Solution]
\label{def:2-sided}
A triple $(M,\go,u)$ is said to be a {\em 2-sided solution} to problem~\eqref{eq:prob_SD} if
$$
M\setminus{\rm MAX}(u)\,\,=\,\, M_+\sqcup M_-\,,
$$
where $M_+$ is either an outer or a cylindrical region, that is
$$
\max_{S\in\pi_0(\pa M_+)}\kappa(S)\,=\,\max_{\pa M_+}\frac{|\D u|}{\umax}\,\leq\,\sqrt{n}\,,
$$ 
and $M_-$ is either an inner or a cylindrical region, that is
$$
\max_{S\in\pi_0(\pa M_-)}\kappa(S)\,=\,\max_{\pa M_-}\frac{|\D u|}{\umax}\,\geq\,\sqrt{n}\,.
$$
\end{definition}
%

The generic shape of a $2$-sided solution is shown in \figurename~\ref{fig:2sided_SD}.
 We recall that,
by a classical theorem of {\L}ojasiewicz~\cite{Lojasiewicz_2}, the set ${\rm MAX}(u)$ is given {\em a priori} by a possibly disconnected stratified analytic subvariety of dimensions ranging from $0$ to $(n-1)$.
In particular, it follows that a $2$-sided solution contains a stratified (possibly disconnected) hypersurface $\Sigma\subseteq{\rm MAX}(u)$ which separates $M_+$ and $M_-$, that is, $\overline{M}_+\cap\overline{M}_-=\Sigma$. 
This hypersurface will play an important role in our analysis, as it represents the junction between the regions $M_+$ and $M_-$. We are now ready to state the main result of this subsection.

\begin{figure}
	\centering
	\includegraphics[scale=0.8]{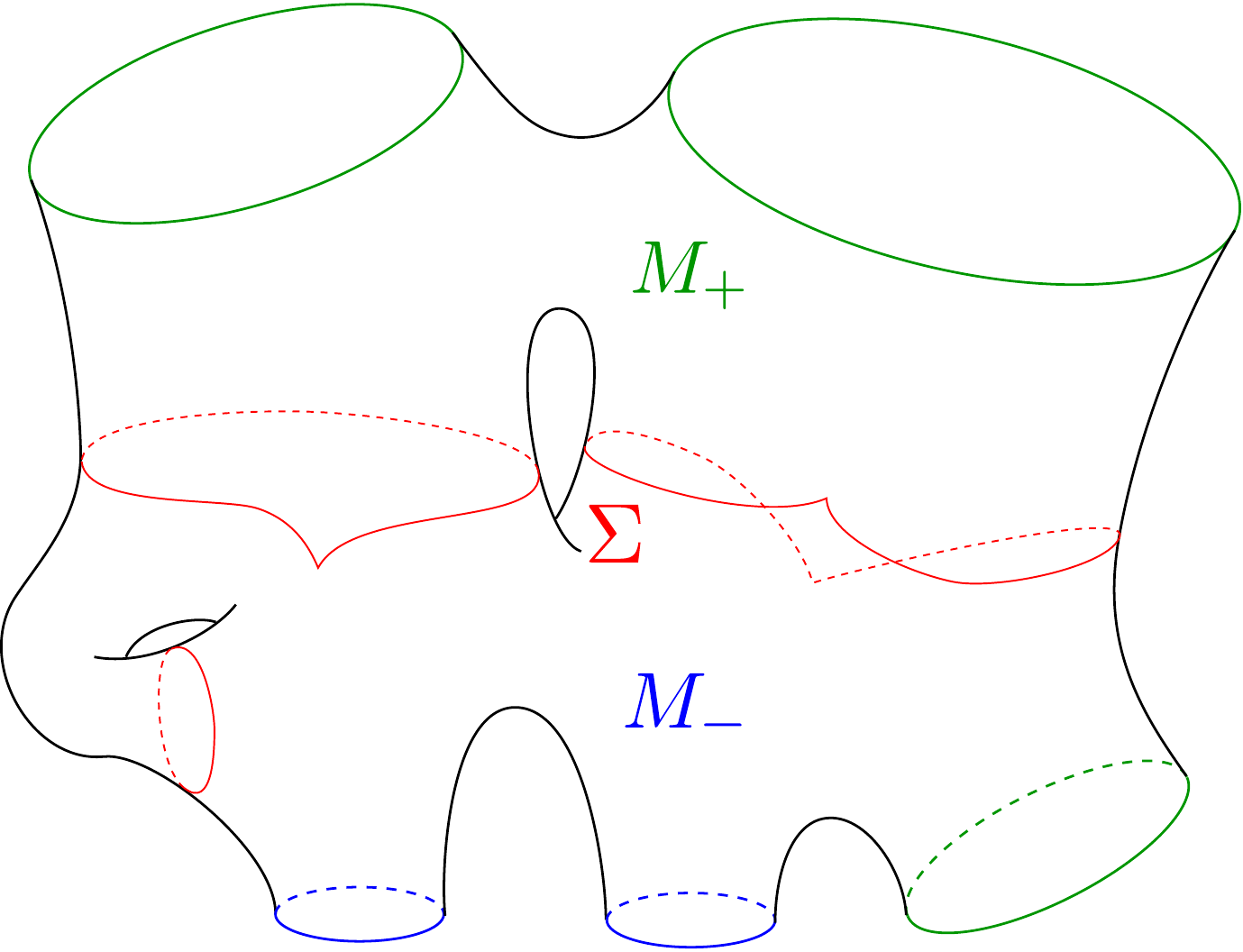}
	\caption{\small The drawing represents the possible structure of a generic $2$-sided solution to problem~\eqref{eq:prob_SD}. The red line represents the set ${\rm MAX}(u)$, with the separating stratified hypersurface $\Sigma$ put in evidence. The blue colour of a boundary component indicates a black hole horizon, whereas the green colour indicates a cosmological horizon. Cylindrical horizons are not considered in this figure since they are non generic.}
	\label{fig:2sided_SD}
\end{figure}

A careful analysis along $\Sigma$, combined with the area upper and lower bounds for the horizons stated in Subsection~\ref{sub:area_bounds_intro}, will lead to the proof of the following $3$-dimensional Black Hole Uniqueness Theorem:

\begin{theorem}
\label{thm:BHU3D}
Let $(M,\go,u)$ be a $3$-dimensional $2$-sided solution to problem~\eqref{eq:prob_SD}, and let $\Sigma\subseteq {\rm MAX}(u)$ be the stratified hypersurface separating $M_+$ and $M_-$. Let also
$$
m_+\,=\,\mu(M_+,\go,u)\,,\qquad \hbox{and} \qquad m_-\,=\,\mu(M_-,\go,u)
$$
be the virtual masses of $M_+$ and $M_-$, respectively. 
Suppose that the following conditions hold
\begin{itemize}
\smallskip
\item \underline{mass compatibility}\qquad\qquad\qquad\quad $\,m_+=m= m_-$ for some $0<m\leq \mmax$,
\smallskip
\item \underline{connected cosmological horizon}\quad\ \ $\,\pa M_+$ is connected.
\smallskip
\end{itemize}
Then the triple $(M,\go,u)$ is isometric to either the Schwarzschild--de Sitter solution~\eqref{eq:SD} with mass $0<m<\mmax$ or to the Nariai solution~\eqref{eq:cylsol_D} with mass $m=\mmax$.
\end{theorem}


The hypothesis of connected cosmological horizon 
is motivated by the beautiful result in~\cite[Theorem~B]{Ambrozio}, where it is proven that any static solution $(M,\go,u)$ admits at most one unstable horizon. From a physical perspective, one may expect that the unstable horizons should be the ones of cosmological type, whereas the horizons of black hole type should be stable. This is what happens for the model solutions, as one can easily check. 
This observation leads us to formulate the following conjecture, which, if proven to be true, would allow to remove the assumption of connected cosmological horizon from Theorem~\ref{thm:BHU3D}.
\begin{conjecture}
An horizon of cosmological type is necessarily unstable. In particular, every static solution to problem~\eqref{eq:prob_SD} has at most one horizon of cosmological type.
\end{conjecture}

\subsection{Summary.}

In the remainder of the paper we will prove the results stated in this introduction.
We will first focus on outer and inner regions, since the analysis of these two cases is similar. 
Our study is based on the so called {\em cylindrical ansatz}, introduced in~\cite{Ago_Maz_1,Ago_Maz_3,Ago_Maz_2}, which consists is finding an appropriate conformal change of the original metric $\go$ in terms of the static potential $u$. 

After some preliminaries (Section~\ref{sec:pre}) in~\ref{sec:cyl_ans} we will describe this method, we will set up the formalism and we will provide some preliminary lemmata and computations that will be used throughout the paper.
Building on this, we will prove in Section~\ref{sec:int_id} a couple of integral identities in the conformal setting.

In Section~\ref{sec:area_bounds} we will proceed to the proof of the inequalities in Theorems~\ref{thm:area_bound_SD},~\ref{thm:scalarcurvature_bound_SD} and~\ref{thm:lower_bound}, for both the cases of outer and inner regions.
In Section~\ref{sec:consequences_SD} we will translate the integral identities proven in Section~\ref{sec:int_id} in terms of the original metric $\go$. As a consequence, we will prove the rigidity statements for Theorems~\ref{thm:area_bound_SD},~\ref{thm:scalarcurvature_bound_SD}, together with some weighted area inequalities for the horizons.

In Section~\ref{sec:ass_SD} we will show that our analysis can be improved under the assumption that the solution is $2$-sided, and this will lead us to the proof of Theorem~\ref{thm:BHU3D} stated in Subsection~\ref{sub:BHU_intro}, in the case where $m_+<\mmax$. 

Finally, in Section~\ref{sec:nariai_SD} we will focus on the cylindrical regions. The analysis of the cylindrical case is slightly different, as our model solution will be the Nariai triple instead of the Schwarzschild--de Sitter triple, however the ideas behind our analysis are completely analogous. In this section we will establish the results stated in Subsection~\ref{sub:area_bounds_intro} for cylindrical regions and we will complete the proof of Theorem~\ref{thm:BHU3D} by studying the case $m_+=\mmax$.

\section{Analytic preliminaries}
\label{sec:pre}

This section is devoted to the setup of the {\em cylindrical ansatz}, which will be the starting point of the proofs of our main results. We will work on a single region $N$ of our manifold $M$, and we will always suppose that $N$ is not cylindrical, that is
$$
\max_{S\in\pi_0(\pa N)}\kappa(S)
\,\neq\,\sqrt{n}\,.
$$
The case of equality requires a different analysis, and will be studied separately in Section~\ref{sec:nariai_SD}.

The cylindrical ansatz is inspired by the analogous technique used in~\cite{Ago_Maz_1,Ago_Maz_3,
Ago_Maz_2}, and consists in an appropriate conformal change of the original triple. The idea comes from the observation that the Schwarzschild--de Sitter metric can be made cylindrical via a division by $|x|^2$. In fact, the metric
$$
\frac{1}{|x|^2}\left(\frac{d|x|\otimes d|x|}{1-|x|^2- 2m |x|^{2-n}}+|x|^2 g_{\Sph^{n-1}}\right)\,=\,
\frac{d|x|\otimes d|x|}{|x|^2(1-|x|^2- 2m |x|^{2-n})}+ g_{\Sph^{n-1}}\,,
$$
after a rescaling of the coordinate $|x|$, is just the standard metric of the cylinder $\R\times\Sph^{n-1}$. We would like to perform a similar change of coordinates on a general solution $(M,\go,u)$.

To this end, in Subsection~\ref{sub:pseudo_radial} we are going to define on a region $N$ of a general triple $(M,\go,u)$ a {\em pseudo-radial function} $\Psi:N\to \R$. The function $\Psi$ will be constructed starting from the static potential $u$, and in the case where $u$ is as in the Schwarzschild--de Sitter solution~\eqref{eq:SD}, it will simply coincide with $|x|$. 

Subsection~\ref{sub:preparatory_estimates_SD} is devoted to the proof of the relevant properties of the pseudo-radial function. Most of the results in this subsection are quite technical, and the reader is advised to simply ignore this part of the work and to come back only when needed. 
However, there is one result that deserves to be mentioned. 
In Proposition~\ref{pro:rev_loj_SD} we will prove that static potentials satisfy a reverse {\L}ojasiewicz inequality. The proof does not depend so deeply on the equations in~\eqref{eq:prob_SD}, and can thus be adapted to a much larger family of functions. This will be the object of further studies. For the purposes of this work, the reverse {\L}ojasiewicz inequality will be crucial in the Minimum Principle argument that leads to Proposition~\ref{pro:min_pr_SD}. It is interesting to notice that Proposition~\ref{pro:min_pr_SD}, in turn, will allow to improve the reverse {\L}ojasiewicz inequality, as explained in Remark~\ref{rem:bound_Du}. However, since the proof of Proposition~\ref{pro:min_pr_SD} exploits the equations in~\eqref{eq:prob_SD}, we do not know if the improved {\L}ojasiewicz inequality still holds outside the realm of static potentials.

\subsection{The pseudo-radial function.}
\label{sub:pseudo_radial}

Let $(M,\go,u)$ be a solution to problem~\eqref{eq:prob_SD}, and let $N$ be a connected component of $M\setminus{\rm MAX}(u)$. 
As already discussed above, in this subsection we focus on inner and outer region. In other words, the quantity 
$$
\max_{S\in\pi_0(\pa N)}\kappa(S)\,=\,\max_{\pa N}\frac{|\D u|}{\umax}
$$ 
will always be supposed to be different from $\sqrt{n}$. In particular, the virtual mass 
$$
m\,=\,\mu(N,\go,u)\,,
$$
is strictly less than $\mmax$. The special case $m=\mmax$ will be discussed later, in Section~\ref{sec:nariai_SD}.

The aim of this subsection is that of defining a pseudo-radial function, that is, a function that mimic the behavior of the radial coordinate $|x|$ in the Schwarzschild--de Sitter solution.
First of all, we recall that our problem is invariant under a normalization of $u$, hence we first rescale $u$ in such a way that its maximum is the same as the maximum of the Schwarzschild--de Sitter solution with mass $m$.

\begin{notation}
	\label{not:SD}
We will make use of the notations $\mmax,\umax$ introduced in~\eqref{eq:mmax_SD},~\eqref{eq:umax_SD}. We recall their definitions here
	$$
	\mmax=\sqrt{\frac{(n-2)^{n-2}}{n^n}}\,,\qquad
	\umax(m)\,=\,\sqrt{1-\left(\frac{m}{\mmax}\right)^{2/n}}\,.
	$$
We emphasize that $\umax=\umax(m)$ is a function of the virtual mass $m$ of $N$. We will explicitate that dependence only when it will be significative.
\end{notation}

\begin{normalization}
We normalize $u$ in such a way that its maximum is $\umax(m)$, where $m$ is the virtual mass of $N$ and $\umax(m)$ is defined as in Notation~\ref{not:SD}.
\end{normalization}

\noindent As usual, we let $r_+(m)>r_-(m)\geq 0$ be the two positive roots of the polynomial $P_m(x)=x^{n-2}-x^n-2m$, we set $r_0(m)=[(n-2)m]^{1/n}$ and we define the function
\begin{align*}
F_m:[0,\umax(m)]\times [r_-(m),r_+(m)]\,\,&\longrightarrow\,\, \R
\\
(u,\psi)\,\,\longmapsto\,\, F_m(u,\psi)\,=\,u^2-1&+\psi^2+2m\psi^{2-n}
\end{align*}
It is a simple computation to show that $\pa F_m/\pa\psi=0$ if and only if $\psi=0$ or $\psi=r_0(m)$. Therefore, as a consequence of the Implicit Function Theorem we have the following.  

\begin{proposition}
Let $u$ be a positive function and let $\umax$ be its maximum value. Then there exist functions
$$
\psi_-:[0,\umax]\,\,\longrightarrow\,\,\left[r_-(m),r_0(m)\right]\,\,,\qquad
\psi_+:[0,\umax]\,\,\longrightarrow\,\,\left[r_0(m),r_+(m)\right]\,\,,
$$
such that $F_m(u,\psi_-(u))=F_m(u,\psi_+(u))=0$ for all $u\in[0,\umax(m)]$.
\end{proposition}

\noindent
Let us make a list of the main properties of $\psi_+$ and $\psi_-$, that can be derived easily from their definition. 
\begin{itemize}
\item First of all, we can compute $\psi_+$, $\psi_-$ and their derivatives using the following formul\ae
\begin{equation}
\label{eq:def_Psi_SD}
u^2\,=\,1-\psi_{\pm}^2-2m\psi_{\pm}^{2-n}\,.
\end{equation}
\begin{equation}
\label{eq:dpsideu_SD}
\dot{\psi}_{\pm}\,=\,-\frac{u}{\psi_{\pm}\big[1-\big(r_0(m)/\psi_{\pm}\big)^{n}\big]}\,,\qquad \ddot{\psi}_{\pm}\,=\,n\frac{\dot{\psi}_{\pm}^3}{u}+(n-1)\frac{\dot{\psi}_{\pm}^2}{{\psi}_{\pm}}+\frac{\dot{\psi}_{\pm}}{u}\,.
\end{equation}

\item The function $\psi_-$ takes values in $[r_-(m),r_0(m)]$, hence $\psi_-^n\leq r_0^n(m)=(n-2)m$ and from~\eqref{eq:dpsideu_SD} we deduce
$$
\dot\psi_-\,\geq\,0\,,\qquad \ddot\psi_-\,\geq\, 0\,,\qquad \lim_{u\to\umax^-}\dot\psi_-\,=\,+\infty\,.
$$

\item The function $\psi_+$ takes values in $[r_0(m),r_+(m)]$, hence $\psi_+^n\geq r_0^n(m)=(n-2)m$ and from the first formula in~\eqref{eq:dpsideu_SD} we deduce  that $\dot\psi_+$ is nonpositive and diverges as $u$ approaches $\umax$. Moreover, the second formula in~\eqref{eq:dpsideu_SD} can be rewritten as
$$
\ddot\psi_+\,=\,\frac{\dot\psi_+}{u}\bigg\{1+\left[1+(n-1)(n-2)m\psi_+^{-n}\right]\dot\psi_+^2\bigg\}\,,
$$ 
from which it follows $\ddot\psi_+\leq 0$.
Summing up, we have
$$
\dot\psi_+\,\leq\,0\,,\qquad \ddot\psi_+\leq 0\,,\qquad \lim_{u\to\umax^-}\dot\psi_+\,=\,-\infty\,.
$$
\end{itemize}

\noindent
Let us now come back to our case of interest, that is, let us consider a region $N\subseteq M\setminus{\rm MAX}(u)$. We want to use the functions $\psi_\pm$ in order to define a {\em pseudo-radial function} on $N$. To this end, we distinguish between the case where $N$ is an outer or an inner region, according to Definition~\ref{def:horiz}.

\begin{itemize}
\item If $N$ is an outer region, then our reference model will be the outer region of the Schwarzschild--de Sitter solution~\eqref{eq:SD}. Accordingly, we define the pseudo-radial function $\Psi_+$ as
\begin{equation}
\label{eq:pr_function_+}
\begin{split}
\Psi_+:\,N &\,\longrightarrow\, \left[r_0(m),r_+(m)\right]
\\
p&\,\longmapsto\, \Psi_+(p):=\psi_+(u(p))\,. 
\end{split}
\end{equation}
Notice that, if $N$ is the outer region of the Schwarzschild--de Sitter solution~\eqref{eq:SD} with mass $m$, for every $p\in N$ the value of $\Psi_+(p)$ is equal to the value of the radial coordinate $|x|$ at $p$.
\smallskip
\item If $N$ is an inner region, then our reference model will be the inner region of the Schwarzschild--de Sitter solution~\eqref{eq:SD}. Accordingly, we define the pseudo-radial function $\Psi_-$ as
\begin{equation}
\label{eq:pr_function_-}
\begin{split}
\Psi_-:\,N &\,\longrightarrow\, \left[r_-(m),r_0(m)\right]
\\
p&\,\longmapsto\, \Psi_-(p):=\psi_-(u(p))\,. 
\end{split}
\end{equation}
Notice that, if $N$ is the inner region of the Schwarzschild--de Sitter solution~\eqref{eq:SD} with mass $m$, for every $p\in N$ the value of $\Psi_-(p)$ is equal to the value of the radial coordinate $|x|$ at $p$.
\end{itemize}
\noindent
In the case of $2$-sided solutions we will need a global version of the definition above.
\begin{itemize}
\item If $(M,\go,u)$ is a $2$-sided solution in the sense of Definition~\ref{def:2-sided}, then we define the {\em global pseudo-radial function} as
\begin{equation}
\label{eq:pr_function_global}
\begin{split}
\Psi:\,M &\,\longrightarrow\, \left[r_-(m),r_+(m)\right]
\\
p&\,\longmapsto\, \Psi(p):=
\begin{dcases}
\psi_+(u(p))   & \hbox{ if } p\in M_+\,,
\\
\psi_-(u(p))   & \hbox{ if } p\in M_-\,,
\\
r_0(m) & \hbox{ if } p\in {\rm MAX}(u)\,.
\end{dcases}
\end{split}
\end{equation}
If $(M,\go,u)$ is isometric to the Schwarzschild--de Sitter solution~\eqref{eq:SD} with mass $m$, then $\Psi$ coincides with the radial coordinate $|x|$. The function $\Psi$ is continuous by construction, but {\em a priori} we have no more information about its regularity near the set ${\rm MAX}(u)$. However, in Subsection~\ref{sub:preparatory_estimates_SD} we will prove that $\Psi$ is always Lipschitz. Moreover, we will also show that $\Psi$ is $\mathscr{C}^2$ at the points of the top stratum of the hypersurface $\Sigma\subseteq{\rm MAX}(u)$ that separates $M_+$ and $M_-$.
\end{itemize} 

By definition, we have the following relation between the derivatives of the pseudo-radial function $\Psi$ and the potential $u$.
\begin{equation}
\label{eq:rel_u_gradient_Psi}
\D \Psi_{\pm}\,=\,(\dot{\psi}_\pm\circ u)\,\D u\,,\qquad\DD \Psi_\pm\,=\,(\dot{\psi}_\pm\circ u)\,\DD u\,+\,(\ddot{\psi}_\pm\circ u)\,du\otimes du\,.
\end{equation}

\begin{notation}
In the following sections, we will perform several formal computations. In order to simplify the notations, we will avoid to indicate the subscript $\pm$, and we will simply denote by $\Psi=\psi\circ u$ the pseudo-radial function on a region $N$ of $M\setminus{\rm MAX}(u)$, where we understand that $\Psi$ is defined by~\eqref{eq:pr_function_+} if we are in an outer region and by~\eqref{eq:pr_function_-} if we are in an inner region. When there is no risk of confusion, we will also avoid to explicitate the composition with $u$, namely, we will write $\psi$ instead of $\psi\circ u$. For instance, the formul\ae\ in~\eqref{eq:rel_u_gradient_Psi} will be simply written as
$$
\D \Psi\,=\,\dot{\psi}\,\D u\,,\qquad\DD \Psi\,=\,\dot{\psi}\,\DD u\,+\,\ddot{\psi}\,du\otimes du\,,
$$
\end{notation}

\begin{figure}
\centering
\includegraphics[scale=0.5]{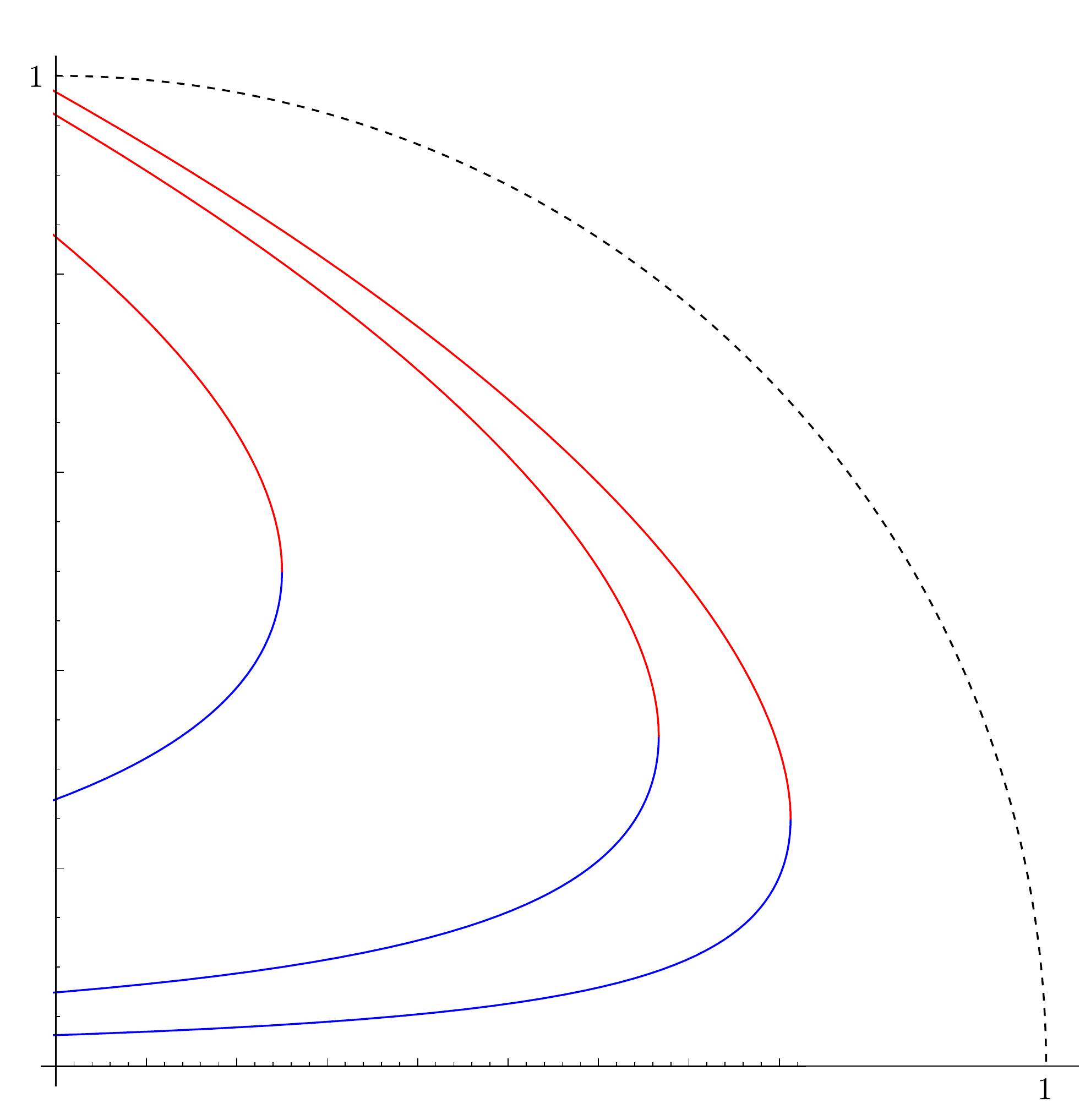}
\caption{\small Relation between $u^2$ (on the $x$-axis) and the pseudo-radial functions (on the $y$-axis) for different values of the virtual mass $m$. The blue lines represent the relation with $\psi_-$ whereas the red lines represent the relation with $\psi_+$. We have also included in the plot a dashed line showing the relation between the radial coordinate and the static potential in the de Sitter solution~\eqref{eq:D}, which represents the limit case when $m\to 0$.}
\label{fig:psi_SD}
\end{figure}

\subsection{Preparatory estimates}
\label{sub:preparatory_estimates_SD}

Here we collect some lemmata that will be useful in the following. The first one shows an important connection between the value of the pseudo-radial function at the boundary and the surface gravity.

\begin{lemma}
\label{le:bound_psi_SD}
Let $(M,\go,u)$ be a solution to problem~\eqref{eq:prob_SD} and let $N\subseteq M\setminus{\rm MAX}(u)$ be a region with virtual mass $m=\mu(N,\go,u)<\mmax$.
If $N$ is outer, then
$$
\max_{\pa N}\left|\frac{\D u}{r_+(m)\big[1-\big(r_0(m)/r_+(m)\big)^n\big]}\right|\,=\,1\,.
$$
If $N$ is inner, then
$$
\max_{\pa N}\left|\frac{\D u}{r_-(m)\big[1-\big(r_0(m)/r_-(m)\big)^n\big]}\right|\,=\,1\,.
$$
\end{lemma}

\begin{proof}
The proof is an easy computation. We recall from the definition of the virtual mass $m$ of $N$, that $\max_{\pa N}|\D u|/\umax= k_\pm(m)$, where $k_\pm$ are the surface gravity functions defined by~\eqref{eq:k+} and~\eqref{eq:k-}, and the sign $\pm$ depends on whether $N$ is an outer or inner region. Therefore, we have
\begin{align*}
	\max_{\pa N}\left|\frac{\D u}{r_\pm(m)\big[1-\big(r_0(m)/r_\pm(m)\big)^n\big]}\right|\,&=\,\frac{\umax}{r_\pm(m)\big|1-\big(r_0(m)/r_\pm(m)\big)^n\big|}\,\max_{\pa N}\frac{|\D u|}{\umax}\,
	\\
	&=\,\frac{\umax}{r_\pm(m)\big|1-\big(r_0(m)/r_\pm(m)\big)^n\big|}\,k_\pm(m)
	\\
	&=\,1\,,
\end{align*}
where the last equality follows from the definition of $k_+$ and $k_-$.
\end{proof}

\begin{remark}
\label{rem:bound_psi_SD}
Following the proof of Lemma~\ref{le:bound_psi_SD}, it is easy to see that, if $N$ is an outer region, then, for every $\mu(N,\go,u)\leq m\leq\mmax$ it holds
$$
\max_{\pa N}\left|\frac{\D u}{r_+(m)\big[1-\big(r_0(m)/r_+(m)\big)^n\big]}\right|\,\leq 1\,.
$$
Similarly, if $N$ is an inner region, one can see that for every $0\leq m\leq \mu(N,\go,u)$ it holds
$$
\max_{\pa N}\left|\frac{\D u}{r_-(m)\big[1-\big(r_0(m)/r_-(m)\big)^n\big]}\right|\,\leq 1\,.
$$
This remark will be useful in Section~\ref{sec:ass_SD}, where we will work with parameters $m$ that do not necessarily coincide with the virtual mass.
\end{remark}

We now pass to discuss an estimate for the gradient of the potential $u$ near the maximum points. This estimate will be an important ingredient in the proof of Lemma~\ref{le:estimate_upsi_SD} below, which is the result that we will actually need in the following. However, Proposition~\ref{pro:rev_loj_SD} is also interesting on its own. In fact, it can be interpreted as a reverse {\L}ojasiewicz inequality for the function $u$ (for the original {\L}ojasiewicz inequality, see~\cite[Th{\'e}or{\`e}m~4]{Lojasiewicz_1}). Proposition~\ref{pro:rev_loj_SD} is stated for solutions to problem~\eqref{eq:prob_SD}, but we emphasize that a similar property can be proven for a much larger class of functions, as it will be discussed in a forthcoming work.

\begin{proposition}
\label{pro:rev_loj_SD}
Let $(M,\go,u)$ be a solution to problem~\eqref{eq:prob_SD} and let $\umax$ be the maximum of $u$. Then, for every $0<\beta<1$, there exists a constant $K_{\beta}$ and an open neighborhood $\Omega_{\beta}\supset{\rm MAX}(u)$ such that  
$$
|\D u|^2(x)\,\leq\,K_{\beta}\,[\umax-u(x)]^\beta\,,
$$
for all $x\in \Omega_{\beta}$.
\end{proposition}

\begin{proof}
We consider the function
$$
w\,=\,|\D u|^2-K(\umax-u)^\beta\,,
$$
where $K>0$ is a constant that will be chosen conveniently later. We compute
$$
\D w\,=\,\D|\D u|^2\,+\,\beta K (\umax-u)^{-(1-\beta)}\D u\,,
$$
and diverging the above formula
\begin{align*}
\De w\,&=\,\De|\D u|^2+\beta K (\umax-u)^{-(1-\beta)}\De u+\beta(1-\beta) K (\umax-u)^{-(2-\beta)}|\D u|^2
\\
&=\,2|\DD u|^2+2\Ric(\D u,\D u)+2\langle \D\De u\,|\,\D u \rangle+\beta K \frac{\De u}{(\umax-u)^{1-\beta}}+\beta(1-\beta) K \frac{|\D u|^2}{(\umax-u)^{2-\beta}}\,,
\end{align*}
where in the second equality we have used Bochner formula. Since $|\D u|$ goes to zero as we approach ${\rm MAX}(u)$, so does the quantity $h=2\Ric(\D u,\D u)+2\langle \D\De u|\D u \rangle$.
Moreover, we have $|\DD u|\geq (\De u)^2/n=n u^2>0$ in a neighborhood of ${\rm MAX}(u)$. From the compactness of the level sets of $u$, it follows that we can choose $\eta>0$ small enough such that 
$$
|h|\,\leq\, 2\,|\DD u|^2\quad \hbox{ on }\,
\{\umax-\eta\leq u\leq \umax\}\,.
$$
Therefore, from the identity above we find
\begin{align*}
\De w\,
&\geq\,\beta K \frac{\De u}{(\umax-u)^{1-\beta}}+\beta(1-\beta) K \frac{|\D u|^2}{(\umax-u)^{2-\beta}}
\\
&=\,\beta K \frac{\De u}{(\umax-u)^{1-\beta}}+\beta(1-\beta) K \frac{w}{(\umax-u)^{2-\beta}}+\beta(1-\beta) K^2 \frac{1}{(\umax-u)^{2-2\beta}}\,,
\end{align*}
where in the second equality we have used $|\D u|^2=w+K(\umax-u)^{\beta}$. It follows that, on $\{\umax-\eta\leq u\leq \umax\}$, it holds
\begin{equation}
\label{eq:ellipticine}
\De w-\beta(1-\beta)K\frac{1}{(\umax-u)^{2-\beta}}w\,\geq\,\beta X\left[\De u+(1-\beta)X\right]\,,
\end{equation}
where 
$$
X=\frac{K}{(\umax-u)^{1-\beta}}\,.
$$
On $\{\umax-\eta\leq u\leq \umax\}$ we have  
$$
X\,=\,\frac{K}{(\umax-u)^{1-\beta}}\,\geq\,\frac{K}{\eta^{1-\beta}}\,,
$$
Moreover, $\De u$ is continuous and thus bounded in a neighborhood of ${\rm MAX}(u)$. This means that, for any $K$ big enough, we have $(1-\beta)X+\De u\geq 0$ on the whole $\{\umax-\eta\leq u\leq \umax\}$.
For such values of $K$, the right hand side of~\eqref{eq:ellipticine} is nonnegative, that is,
$$
\De w-\beta(1-\beta)K\frac{1}{(\umax-u)^{2-\beta}}w\,\geq\,0\,, \ \,\hbox{ on }\,\{\umax-\eta\leq u\leq \umax\}\,.
$$
Therefore, we can apply the Weak Maximum Principle~\cite[Corollary~3.2]{Gil_Tru} to $w$ in any open set where $w$ is $\mathscr{C}^2$ --that is, on any open set of $\{\umax-\eta\leq u\leq \umax\}$ that does not intersect ${\rm MAX}(u)$.
Up to increasing the value of $K$, if needed, we can suppose
$$
K\,\geq\,\max_{\{u=\umax-\eta\}}\frac{|\D u|^2}{(\umax-u)^\beta}\,=\,\frac{\max_{\{u=\umax-\eta\}}|\D u|^2}{\eta^\beta}\,,
$$
so that $w\leq 0$ on $\{u=\umax-\eta\}$. Now we apply the Weak Maximum Principle to the function $w$ on the open set $\Omega_\epsilon=\{\umax-\eta\leq u\leq \umax-\epsilon\}$, obtaining
$$
w\,\leq\, \max_{\pa \Omega_\ep} (w)\,=\,\max\left\{\max_{\{u=\umax-\ep\}} \!(w)\,,\,\max_{\{u=\umax-\eta\}} \!(w)\right\}\,\leq\,\max\left\{\max_{\{u=\umax-\ep\}} \!(w)\,,\,0\right\}\,.
$$
Taking the limit as $\ep\to 0$, from the continuity of $u$ and the compactness of the level sets, we have  $\lim_{\ep\to 0}\max_{\{u=\umax-\ep\}} \!(w)=0$, hence we obtain $w\leq 0$ on $\{\umax-\eta\leq u\leq \umax\}$. Recalling the definition of $w$, we have proved that the inequality
$$
|\D u|^2\,\leq\, K(\umax-u)^\beta
$$
holds in $\Omega=\{\umax-\eta\leq u< \umax\}$, which is a collar neighborhood of ${\rm MAX}(u)$. 
\end{proof}

\noindent
The above result can actually be improved in the following way. Take $\alpha<\beta<1$. In the neighborhood $\Omega_{\beta}$ given by Proposition~\ref{pro:rev_loj_SD}, we have
$$
\frac{|\D u|^2}{(\umax-u)^\alpha}\,=\,
\frac{|\D u|^2}{(\umax-u)^\beta}\cdot (\umax-u)^{\beta-\alpha}\leq K_{\beta} \,(\umax-u)^{\beta-\alpha}\,,
$$
for some constant $K_{\beta}$.
Since $\beta>\alpha$, the right hand side goes to zero as we approach ${\rm MAX}(u)$ and we obtain the following corollary.

\begin{corollary}
\label{cor:Dpsi_SD}
Let $(M,\go,u)$ be a solution to problem~\eqref{eq:prob_SD} and let $\umax$ be the maximum of $u$. Then, for every $p\in{\rm MAX}(u)$, it holds 
$$
\lim_{x\not\in{\rm MAX}(u),\,x\to p}\frac{|\D u|^2}{(\umax-u)^\a}(x)\,=\,0\,,
$$
for all $0<\a<1$.
\end{corollary}

\noindent 
Of course, we have specified $x\not\in{\rm MAX}(u)$ in the limit above because otherwise the function in the argument is not defined.
Corollary~\ref{cor:Dpsi_SD}, in turn, allows us to prove the following useful estimate near ${\rm MAX}(u)$.

\begin{lemma}
\label{le:estimate_upsi_SD}
Let $(M,\go,u)$ be a solution to problem~\eqref{eq:prob_SD} and let $\Psi=\psi\circ u$ be defined by~\eqref{eq:pr_function_+} or~\eqref{eq:pr_function_-} with respect to a parameter $m\in(0,\mmax)$. Then, for every $p\in{\rm MAX}(u)$, it holds 
$$
\lim_{x\to p}{\dot\psi^{2\a}}|\D u|^2(x)\,=\,0\,,
$$
for every $0<\a<1$.
\end{lemma}

\begin{proof}
First, we compute
\begin{align*}
\frac{\umax-u}{\big[1-\big(r_0(m)/\psi\big)^n\big]^2}\,&=\,\frac{1}{\umax+u}\,\frac{\umax^2-u^2}{\big[1-\big(r_0(m)/\psi\big)^n\big]^2}
\\ 
&=\,
\frac{1}{\umax+u}\,\frac{1-(m/\mmax)^{2/n}-1+\psi^2+2m\psi^{2-n}}{\big[1-\big(r_0(m)/\psi\big)^n\big]^2}
\\
&\!\!\!\!\!\!\!=\,
\frac{1}{(\umax+u)\big[1-\big(r_0(m)/\psi\big)^n\big]}\,\left[\psi^2-\frac{n\,m^{\frac{2}{n}}}{(n-2)^{\frac{n-2}{n}}}\,\frac{1-\big(r_0(m)/\psi\big)^{n-2}}{1-\big(r_0(m)/\psi\big)^n}\right].
\end{align*}
We want to show that the quantity above has a finite nonzero limit as we approach ${\rm MAX}(u)$. If we denote $z=r_0(m)/\psi$, the equation above can be rewritten as
\begin{align*}
\frac{\umax-u}{\big[1-\big(r_0(m)/\psi\big)^n\big]^2}\,&=
\frac{\umax-u}{(1-z^n)^2}
\\
&=\,
\frac{r_0^2(m)}{(\umax+u)(1-z^n)}\,\left[z^{-2}-\frac{n}{n-2}\,\frac{1-z^{n-2}}{1-z^n}\right]
\\
&=\,
\frac{r_0^2(m)}{(\umax+u)(1-z^n)}\,\left[z^{-2}-\frac{n}{n-2}\,\frac{1+z+\cdots+z^{n-3}}{1+z^+\cdots+z^{n-1}}\right]
\\
&=\,
\frac{r_0^2(m)\,z^{-2}}{(\umax+u)(1+z+\cdots+z^{n-1})^2}\,\cdot\,\frac{1+z-\frac{2}{n-2}z^{2}\,\left(1+z+\cdots+z^{n-3}\right)}{1-z}\,.
\end{align*}
It is clear that the first factor above has a finite nonzero limit as we approach ${\rm MAX}(u)$, that is, when $z$ goes to $1$. Concerning the second factor, one easily computes
$$
\frac{1+z-\frac{2}{n-2}z^{2}\,\left(1+z+\cdots+z^{n-3}\right)}{1-z}\,=\,n-\frac{2}{n-2}(1-z)\sum_{k=1}^{n-2}(n-k-1)(1+z+\dots+z^{k-1})\,,
$$
Substituting, we easily obtain
\begin{equation}
\label{eq:lim_DPsi_SD}
\frac{\umax-u}{\big[1-\big(r_0(m)/\psi\big)^n\big]^2}\,=\,\frac{r_0^2(m)}{2n\umax}\left[1\,+\,f(z)\right]\,,
\end{equation}
where $f(z)$ is a function that is analytic near $z=1$ and such that $f(1)=0$. 
In particular, recalling formula~\eqref{eq:dpsideu_SD}, we have proved that
$(\umax-u)\dot\psi^2$ has a finite limit as we approach the set ${\rm MAX}(u)$. Therefore, for any $p\in{\rm MAX}(u)$ and $0<\a<1$, we compute
$$
\lim_{x\to p}{\dot\psi^{2\a}}|\D u|^2(x)\,=\,\lim_{x\to p}\left[({\umax-u}){\dot\psi^{2}}\right]^{\a}\!\frac{|\D u|^2}{(\umax-u)^\a}(x)\,,
$$
and, since ${(\umax-u)}{\dot\psi^{2}}$ has a finite limit on ${\rm MAX}(u)$, from Corollary~\ref{cor:Dpsi_SD} we conclude. 
\end{proof}

\noindent
Lemma~\ref{le:estimate_upsi_SD} will be crucial later in the proof of Proposition~\ref{pro:min_pr_SD}, where a Minimum Principle argument will be used to prove a stronger result, namely, that the quantity $\dot\psi|\D u|$ is bounded near ${\rm MAX}(u)$, see Remark~\ref{rem:bound_Du}. In particular, since $(\umax-u)\dot\psi^2$ is also bounded near ${\rm MAX}(u)$, as shown in the proof of Lemma~\ref{le:estimate_upsi_SD}, it follows that the quantity
$$
4\,\Big|\D\big(\sqrt{\umax-u}\big)\Big|^2\,=\,\frac{|\D u|^2}{\umax-u}\,.
$$
is bounded near ${\rm MAX}(u)$.
In other words, the function $\sqrt{\umax-u}$ is always Lipschitz continuous on $M$.

It is worth remarking that, in the neighborhood of the points of the top stratum of ${\rm MAX}(u)$, we can actually prove a much more precise result about the behavior of the static potential $u$. We recall that with top stratum of ${\rm MAX}(u)$ we mean the open subset $\Sigma\subset{\rm MAX}(u)$ which is a $(n-1)$-dimensional analytic submanifold. In other words, the points $p\in\Sigma$ are the ones such that there exists a neighborhood $\Omega$ of $p$ and an analytic function $f:\Omega\to\R$ such that ${\rm MAX}(u)\cap\Omega=f^{-1}(0)$ and $|df|\neq 0$ in $\Omega$.

\begin{proposition}
\label{pro:expansion_u}
Let $(M,\go,u)$ be a solution to problem~\eqref{eq:prob_SD} and let $p\in{\rm MAX}(u)$ be a point in the top stratum of ${\rm MAX}(u)$. Let $\Omega$ be a small neighborhood of $p$ such that $\Sigma=\Omega\cap{\rm MAX}(u)$ is contained in the top stratum and $\Omega\setminus\Sigma$ has two connected components $\Omega_+,\Omega_-$. We define the signed distance to $\Sigma$ as
$$
r(x)\,=\,
\begin{cases}
+ \, d(x,\Sigma)\,,   & \text{ if } x\in \overline{\Omega}_+\,,
\\
- \, d(x,\Sigma)\,,  & \text{ if } x\in \overline{\Omega}_-\,.
\end{cases}
$$
Then the following expansion holds:
\begin{equation}
\label{eq:expansion_u_final}
u\,=\,\umax\,\left[1-\frac{n}{2}\,r^2\,+\,\frac{n}{6}\,\HHH\,r^3 \,-\,\frac{1}{24}\left(2\,n\,|\mathring{\hhh}|^2\,+\,\frac{n(n+1)}{n-1}\,\HHH^2\,-\,n^2\right)r^4\,+\mathcal{O}(r^5)\right]\,,
\end{equation}
where $\HHH$ is the mean curvature of $\Sigma$ with respect to the normal pointing towards $\Omega_+$.
\end{proposition}

\begin{proof}
Let $(x^1,\dots,x^n)$ be a chart centered at $p$, with respect to which the metric $\go$ and the function $u$ are analytic.
From the fact that $p$ belongs to the top stratum of ${\rm MAX}(u)$, it follows that we can choose an open neighborhood $\Omega$ of $p$ in $M$, where the signed distance $r(x)$
is a well defined analytic function (see for instance~\cite{Kra_Par_smoothdist}, where this result is discussed in full details in the Euclidean space, however the proofs extend with small modifications to the Riemannian setting). More precisely, we have
$$
r=\phi(x^1,\dots,x^n)\,,
$$ 
where $\phi$ is an analytic function. Since $r$ is a signed distance function, we have $|\D r|=1$, which implies in particular that one of the partial derivatives of $\phi$ has to be different from zero. Without loss of generality, let us suppose $\pa\phi/\pa x^1\neq 0$ in a small neighborhood $\Om$ of $p$. As a consequence, we have that the function
$$
H:\,\R^{n+1}\,\to\, \R\,,\qquad H(r,x^1,\dots ,x^n)\,=\,r-\phi(x^1,\dots,x^n)\,.
$$
satisfies $\pa H/\pa x^1=-\pa\phi/\pa x^1\neq 0$ in $\Om$. We can then apply the Real Analytic Implicit Function Theorem (see~\cite[Theorem~2.3.5]{Kra_Par}), from which it follows that there exists an analytic function $h:\R^n\to\R$ such that
$$
H(r,h(r,x^2,\dots,x^n),x^2,\dots, x^n)=0\,.
$$
In other words, the change of coordinates from $(r,x^2,\dots,x^n)$ to $(x^1,\dots,x^n)$, which is obtained setting $x^1=h(r,x^2,\dots,x^n)$, is analytic, and in particular $u$ is an analytic function also with respect to the chart $(r,x^2,\dots,x^n)$. In the following computation, it is convenient to denote this new analytic chart as $y=(y^1,\dots, y^n)$, where $y^1=r$ and $y^i=x^i$ for $i=2,\dots, n$.
In particular, in this new chart, the smooth hypersurface $\Sigma\cap\Omega$ coincides with the points with $y^1=0$. 
Since $u$ is analytic, we can take its Taylor expansion in $p$
\begin{equation}
\label{eq:analytic_expansion_1}
\umax-u(y)\,=\,\sum_{k=2}^\infty\sum_{|I|=k} A_{I}\,y^I\,,
\end{equation}
where $I = (I_1, \dots, I_n)$ is a multi-index and $|I| = I_1 + \dots + I_n$. 
Since $\umax-u\equiv 0$ on $\Sigma\cap\Omega=\{y^1=0\}$, the summand on the right hand side of~\eqref{eq:analytic_expansion_1} must be identically zero when we set $y^1=0$. From this it follows that $A_{I}=0$ whenever $I_1=0$.

We now prove that $A_I=0$ also when $I_1=1$. In fact, suppose that this is not true, and let $k$ be the smallest integer such that there exists a multi-index $I$ with $|I|=k$, $I_1=1$ and $A_I\neq 0$. Consider points of the form
$$
y^1\,=\, \ep^{k} \sigma^1\,,\qquad y^i\,=\, \ep\,\sigma^i\,,
$$
where $\ep\in\R$, $\sigma^1 = 1$ and $\sigma^i\in\R\setminus \{0\}$ for all $i=2,\dots, n$. Recalling that $A_I=0$ whenever $I_1=0$ and whenever $|I|<k$ and $I_1=1$, at these points it holds
\begin{equation*}
\umax-u\,=\,\ep^{2k-1}\!\!\!\!\sum_{|I|=k,\ I_1=1} \!\!\!\!A_{I}\,\sigma^I\,+\,\mathcal{O}(\ep^{2k})\,.
\end{equation*}
We recall that we are supposing that there are some nonzero coefficients in the sum on the right hand side, hence we can choose the values of $\sigma^2,\dots,\sigma^n$ in such a way that $\sum_{|I|=k,\ I_1=1} A_{I}\,\sigma^I < 0$.
Therefore, for small values of $\ep>0$, we would have $\umax-u<0$, against the hypothesis that $\umax$ is the maximum value of $u$.

From these considerations, it follows that we can write
\begin{equation}
\label{eq:analytic_expansion_2}
\umax-u(y)\,=\,(y^1)^2\,\cdot\,(A_{(2,0,\dots,0)}\,+\,y^1\,f)\,,
\end{equation}
where $f$ is an analytic function.
Notice that, at the point $p$, we have $\pa_\a u=0$ for all $\a=1,\dots,n$, and from the second equation in~\eqref{eq:prob_SD} we find 
$$
-n\umax=\De u=\go^{\a\b}[\pa^2_{\a\b}u-\Gamma_{\a\b}^\gamma\pa_\gamma u]=-2\,\go^{11}\,A_{(2,0,\dots,0)}=-2\,A_{(2,0,\dots,0)} \, . $$ 
It follows that $A_{(2,0,\dots,0)}=n\umax/2> 0$.

Now that we have found a good expansion of $u$ around the point $p$, it is convenient to come back to the old notation $(r,x^2,\dots,x^n)$. Namely, we set again $r=y^1$ and $x^i=y^i$ for all $i=2,\dots,n$.
Rewriting~\eqref{eq:analytic_expansion_2} recalling also that $a=n\umax/2$, we obtain the following expansion
\begin{equation}
\label{eq:expansion_u_rewr}
u(r,x)\,=\,\umax\,-\,\frac{n}{2}\,\umax\,r^2\,+r^3\,f\,,
\end{equation}
We now want to gather more information on the analytic function $f$. To this end, set $\Sigma_\rho=\{r=\rho\}$ and observe that all $\Sigma_\rho$ with $\rho$ small enough are smooth, since $(r,x)=(r,x^2,\dots,x^n)$ is an analytic chart and $|\D r|=1\neq 0$. In particular, of course, we have $\Sigma_0=\Sigma\cap\Omega$. On each $\Sigma_\rho$, the laplacian of $u$ satisfies the following well known formula 
\begin{equation}
\label{eq:lapl_exp_Sigmar}
\De u\,=\,\DD u({\rm n}_\rho,{\rm n}_\rho)\,+\,\HHH^{\Sigma_\rho}\,\langle\D u\,|\,{\rm n}_\rho\rangle\,+\,\De^{\Sigma_\rho}u\,,
\end{equation}
where ${\rm n}_\rho=\pa/\pa r$ is the $\go$-unit normal to $\Sigma_\rho$, $\HHH^{\Sigma_\rho}$ is the mean curvature of $\Sigma_\rho$ with respect to ${\rm n}_\rho$ and $\De^{\Sigma_\rho}u$ is the laplacian of the restriction of $u$ to $\Sigma_\rho$ with respect to the metric induced by $\go$ on $\Sigma_\rho$. Evaluating ~\eqref{eq:lapl_exp_Sigmar} at $\rho=0$, since $u\equiv\umax$ and $|\D u|=0$ on $\Sigma_0$, recalling also that $\De u=-nu$, we immediately get
$$
\DD u(\nu,\nu)\,=\,\De u\,=\,-n\umax\,,
$$
in agreement with expansion~\eqref{eq:expansion_u_rewr}.
We now differentiate formula~\eqref{eq:lapl_exp_Sigmar} two times with respect to $r$, obtaining the following
\begin{align*}
-n\frac{\pa u}{\pa r}\,&=\,\frac{\pa^3 u}{\pa r^3}\,+\,\HHH^{\Sigma_r}\frac{\pa^2 u}{\pa r^2}\,+\,\frac{\pa \HHH^{\Sigma_r}}{\pa r}\,\frac{\pa u}{\pa r}\,+\,\frac{\pa}{\pa r}\De^{\Sigma_r}u\,.
\\
-n\frac{\pa^2 u}{\pa r^2}\,&=\,\frac{\pa^4 u}{\pa r^4}\,+\,\HHH^{\Sigma_r}\frac{\pa^3 u}{\pa r^3}\,+\,2\,\frac{\pa \HHH^{\Sigma_r}}{\pa r}\,\frac{\pa^2 u}{\pa r^2}\,+\,\frac{\pa^2 \HHH^{\Sigma_r}}{\pa r^2}\,\frac{\pa u}{\pa r}\,+\,\frac{\pa^2}{\pa r^2}\De^{\Sigma_r}u\,.
\end{align*}
Let us focus first on the terms involving $\De^{\Sigma_r}u$. Calling $g_{(r)}$ the metric induced by $\go$ on $\Sigma_r$ and $\Gamma_{\!(r)}$ the Christoffel symbols of $g_{(r)}$, we have
$$
\De^{\Sigma_r}u\,\,=\,\,
g_{(r)}^{ij}\,\frac{\pa^2 u}{\pa x^i\pa x^j}\,+\,g_{(r)}^{ij}\,\,\Gamma^{\phantom{\!(r)}\,k}_{\!(r)\,ij}\,\,\frac{\pa u}{\pa x^k}\,,
$$
where the indices $i,j,k$ vary between $2$ and $n$.
On the other hand, notice from~\eqref{eq:expansion_u_rewr} that
$$
\frac{\pa^2 u}{\pa x^i\pa x^j}_{|_{r=0}}\,=\,\frac{\pa^2 u}{\pa r\pa x^i}_{|_{r=0}}\,=\,\frac{\pa^3 u}{\pa r^2 \pa x^i}_{|_{r=0}}\,=\,\frac{\pa^4 u}{\pa r^2\pa x^i\pa x^j}_{|_{r=0}}\,=\,0\,,
$$ 
for all $i,j=2,\dots,n$. From this, it easily follows 
$$
\frac{\pa}{\pa r}\De^{\Sigma_r}u_{|_{r=0}}\,=\,\frac{\pa^2}{\pa r^2}\De^{\Sigma_r}u_{|_{r=0}}\,=\,0\,.
$$
Since we also know that $\pa u/\pa r=0$ and $\pa^2 u/\pa r^2=-n\umax$ when $r=0$, from the expansions above we deduce
\begin{align*}
\frac{\pa^3 u}{\pa r^3}_{|_{r=0}}\,&=\,n\,\umax\,\HHH\,.
\\
\frac{\pa^4 u}{\pa r^4}_{|_{r=0}}\,&=\,2\,n\,\umax\,\frac{\pa \HHH^{\Sigma_r}}{\pa r}_{|_{r=0}}\,-\,n\,\umax\,\HHH^2\,+\,n^2\,\umax\,,
\end{align*} 
where we have denoted by $\HHH$ the mean curvature of $\Omega\cap\Sigma=\Sigma_0$ for simplicity.
Furthermore, from~\cite[Lemma~7.6]{Hui_Pol} and the first equation in~\eqref{eq:prob_SD} we get
$$
\frac{\pa \HHH^{\Sigma_r}}{\pa r}_{|_{r=0}}
\,=\,-\,|\hhh|^2\,-\,\Ric(\nu,\nu)
\,=\,-\,\left(|\mathring{\hhh}|^2\,+\,\frac{\HHH^2}{n-1}\right)\,-\,\left[\frac{\DD u(\nu,\nu)}{u}\,+\,n\,\langle\,\nu|\,\nu\rangle\right]
\,=\,-\,|\mathring{\hhh}|^2\,-\,\frac{\HHH^2}{n-1}\,,
$$
where we have used $\DD u(\nu,\nu)=-n\umax$, as proven above. Now that we have computed the third and fourth derivative of $u$, we can use this information to improve~\eqref{eq:expansion_u_rewr} and get the desired expansion of the static potential $u$.
\end{proof}

Proposition~\ref{pro:expansion_u} has some very useful consequences for our analysis. Let us start from the simplest one. From expansion~\eqref{eq:expansion_u_final}, we can compute the explicit formula for the gradient of $u$ as we approach a point $p$ in the top stratum of ${\rm MAX}(u)$ as
\begin{align}
\notag
\lim_{x\not\in{\rm MAX}(u),\,x\to p}\,\frac{|\D u|^2 (x)}{\umax-u(x)}\,&=\,
\lim_{r\to 0}\frac{n^2\,\umax^2\,r^2\,+\,\mathcal{O}(r^3)}{({n}/{2})\,\umax\,r^2\,+\,\mathcal{O}(r^3)}
\\
\label{eq:grad_u_nearSigma_N}
&=\,2\,n\,\umax\,.
\end{align}
In particular, recalling formula~\eqref{eq:lim_DPsi_SD}, at each point of the top stratum we deduce the following identity
\begin{equation}
\label{eq:grad_u_nearSigma}
\lim_{x\not\in{\rm MAX}(u),\,x\to p}\,\frac{|\D u|^2 (x)}{\psi^2(x) \left[1-\big(r_0(m)/\psi(x)\big)^n  \right]^2}\,=\,1\,.
\end{equation}
We will see that the left hand side of formula~\eqref{eq:grad_u_nearSigma} admits an interpretation as the norm of the gradient of a pseudo-affine function (see formula~\eqref{eq:na_ffi_SD}) that will be of extreme importance in the rest of the work.

A second important consequence of Proposition~\ref{pro:expansion_u} is the following regularity result for the pseudo-radial function $\Psi$.

\begin{proposition}
\label{pro:C2}
Let $(M,\go,u)$ be a solution to problem~\eqref{eq:prob_SD} and let $p\in{\rm MAX}(u)$ be a point in the top stratum of ${\rm MAX}(u)$. Let $\Omega$ be a small neighborhood of $p$ such that $\Sigma=\Omega\cap{\rm MAX}(u)$ is contained in the top stratum and $\Omega\setminus\Sigma$ has two connected components $\Omega_+,\Omega_-$. We define the function
$$
\Psi(x)=
\begin{dcases}
\Psi_+(x)\,, & \hbox{ if } x\in \overline{\Omega}_+\,,
\\
\Psi_-(x)\,, & \hbox{ if } x\in \overline{\Omega}_-\,,
\end{dcases}
$$ 
where $\Psi_+$ is the pseudo-radial function defined by~\eqref{eq:pr_function_+} with respect to a parameter $m\in[0,\mmax)$ and $\Psi_-$ is the pseudo-radial function defined by~\eqref{eq:pr_function_-} with respect to the same parameter $m$. Then the function $\Psi$ is $\mathscr{C}^3$ in $\Omega$.
\end{proposition}

\begin{proof}
Let us start from formula~\eqref{eq:lim_DPsi_SD} obtained in the proof of Proposition~\ref{le:estimate_upsi_SD}, where we recall that we had set $z=r_0(m)/\psi$. It is clear that it is possible to refine~\eqref{eq:lim_DPsi_SD} by expanding around $z=1$ the function $f(z)$ appearing in it. Namely, we can write
\begin{equation}
\label{eq:lim_DPsi_enhanced}
\frac{\umax-u}{\big[1-\big(r_0(m)/\psi\big)^n\big]^2}\,=\,\frac{r_0^2(m)}{2n\umax}\left[1\,+\,A\,(1-z)\,+\,B\,(1-z)^2\,+\,\mathcal{O}\left((1-z)^3\right)\right]\,,
\end{equation}
for suitable $A,B\in\R$. The computation of the precise values of the coefficients $A,B$ is tedious and it will not be necessary in our proof, as for our argument it is sufficient that such coefficients exist. If we also expand $1-\big(r_0(m)/\psi\big)^n=1-z^n$ in terms of $1-z$, from~\eqref{eq:lim_DPsi_enhanced} we obtain
\begin{equation}
\label{eq:lim_DPsi_enhanced2}
\frac{\umax-u}{(1-z)^2}\,=\,\frac{r_0^2(m)}{2n\umax}\left[ 1\,+\,C\,(1-z)\,+\,D\,(1-z)^2\,+\,\mathcal{O}\left((1-z)^3\right)\right]\,,
\end{equation}
for suitable coefficients $C,D\in\R$ that, once again, we avoid to compute explicitly. We want to use~\eqref{eq:lim_DPsi_enhanced2} in order to prove that $1-z$ can be expanded in terms of $r$ close to $\Sigma$. We do this by applying~\eqref{eq:lim_DPsi_enhanced2} repeatedly as follows
\begin{align*}
1-z\,&=\,\frac{1-z}{\sqrt{\umax-u}}\sqrt{\umax-u}
\\
&=\,\sqrt{\frac{2n\umax}{r_0^2(m)}}\frac{\sqrt{\umax-u}}{\sqrt{1+C(1-z)+D(1-z)^2+\mathcal{O}\left((1-z)^3\right)}}
\\
&=\,\sqrt{\frac{2n\umax}{r_0^2(m)}}\frac{\sqrt{\umax-u}}{\sqrt{1+C\sqrt{\frac{2n\umax}{r_0^2(m)}}\frac{\sqrt{\umax-u}}{\sqrt{1+C(1-z)+\mathcal{O}\left((1-z)^2\right)}}+D\frac{2n\umax}{r_0^2(m)}\frac{\umax-u}{1+\mathcal{O}(1-z)}+\mathcal{O}\left((1-z)^3\right)}}
\\
&\!\!\!\!\!\!\!\!\!\!\!\!\!\!\!\!=\,\sqrt{\frac{2n\umax}{r_0^2(m)}}\frac{\sqrt{\umax-u}}{\sqrt{1+C\sqrt{\frac{2n\umax}{r_0^2(m)}}\frac{\sqrt{\umax-u}}{\sqrt{1+C\sqrt{\frac{2n\umax}{r_0^2(m)}}\frac{\sqrt{\umax-u}}{\sqrt{1+\mathcal{O}(1-z)}}+\mathcal{O}\left((1-z)^2\right)}}+D\frac{2n\umax}{r_0^2(m)}\frac{\umax-u}{1+\mathcal{O}(1-z)}+\mathcal{O}\left((1-z)^3\right)}}\,.
\end{align*} 
We observe, again from~\eqref{eq:lim_DPsi_enhanced2}, that $\mathcal{O}(1-z)=\mathcal{O}(\umax-u)^{1/2}=\mathcal{O}(r)$. Now we use Proposition~\ref{pro:expansion_u} to substitute $\umax-u$ with its expansion~\eqref{eq:expansion_u_final}. Using the known expansions for square roots and fractions, the cumbersome formula above reduces to
$$
1-z\,=\,E\,+\,F\,r\,+\,G\,r^2\,+\,\mathcal{O}(r^3)\,.
$$
for suitable coefficients $E,F,G$. Notice that these coefficients are not necessarily constant, as they can depend on $\HHH$ and $|\mathring{\hhh}|$ coming from the expansion of $u$. Finally, we can substitute back in~\eqref{eq:lim_DPsi_enhanced} to obtain
\begin{equation}
\label{eq:lim_DPsi_enhanced3}
\frac{\umax-u}{\big[1-\big(r_0(m)/\psi\big)^n\big]^2}\,=\,\frac{r_0^2(m)}{2n\umax}\,+\,P\,r\,+\,Q\,r^2\,+\,\mathcal{O}(r^3)\,,
\end{equation}
for suitable coefficients $P,Q$ (that again, may not be constant).
We are now ready to prove the regularity of $\Psi$. Recalling~\eqref{eq:dpsideu_SD}, and using again the expansion~\eqref{eq:expansion_u_final} of $u$, we compute
\begin{align*}
\frac{\pa(\Psi^2)}{\pa r}\,&=\,2\psi\,\dot\psi\,\frac{\pa u}{\pa r}
\\
&=\,\frac{2\,u}{1-(r_0(m)/\psi)^{n}}\,\,\frac{\pa u}{\pa r}
\\
&=\,\frac{2\,u}{\sqrt{\umax-u}}\,\,\frac{\sqrt{\umax-u}}{1-(r_0(m)/\psi)^{n}}\,\,\frac{\pa u}{\pa r}\,.
\end{align*}
We can now expand the three factors using~\eqref{eq:expansion_u_final} and~\eqref{eq:lim_DPsi_enhanced3}, obtaining
$$
\frac{\pa(\Psi^2)}{\pa r}\,=\,R\,+\,S\,r\,+T\,r^2\,+\,\mathcal{O}(r^3)\,,
$$
where the coefficients $R,S,T$ depend only on $\HHH$ and $|\mathring{\hhh}|$. It is then clear that $\pa(\Psi^2)/\pa r$, and also its first and second derivatives, are continuous along $r=0$. A completely analogous reasoning can be done for $\pa(\Psi^2)/\pa x^i$, for all $i=2,\dots,n$. It follows that $\Psi^2$, and thus $\Psi$, is $\mathscr{C}^3$.

It should be mentioned that it is possible to compute precisely the coefficients of the expansion of $\Psi$.
A sufficiently simple way of doing it is to recall that $\Psi=r_0(m)$ on ${\rm MAX}(u)$ and then write
$$
\Psi\,=\,r_0(m)\,+\,W\,r\,+\,X\,r^2\,+\,Y\,r^3\,+\mathcal{O}(r^4)\,,
$$
where $W,X,Y$ are functions of the coordinates $x^2,\dots,x^n$ only. Now one can compute the expansions of the left and right hand sides of the relation $u^2\,=\,1-\Psi^2-2m\Psi^{2-n}$ to obtain information on the functions $W,X,Y$. With some lenghty (but standard) computations, one obtains
$$
\Psi\,=\,r_0(m)\,+\,\umax\,r\,+\,\frac{\umax}{12}\left[|\mathring{\hhh}|^2
\,+\,(n+1)(n-1)\,\left(\frac{\HHH^2}{(n-1)^2}\,-\,\frac{\umax^2}{r_0^2(m)}\right)
\,-\,2\,n\right]\,r^3\,+\mathcal{O}(r^4)\,.
$$
Anyway, this expansion will not be useful in what follows.
\end{proof}

Finally, we conclude this subsection with another important consequence of Proposition~\ref{pro:expansion_u}, which is the following regularity result on the top stratum of ${\rm MAX}(u)$.

\begin{proposition}
\label{pro:SigmaC1}
Let $(M,\go,u)$ be a solution to problem~\eqref{eq:prob_SD} and let $\Sigma$ be the top stratum of ${\rm MAX}(u)$. Then $\overline{\Sigma}$ is a $\mathscr{C}^1$ hypersurface (possibly with boundary).
\end{proposition}

\begin{proof}
We already know that the top stratum $\Sigma$ is an analytic hypersurface, meaning that each point $p\in\Sigma$ admits a neighborhood $\Om$ such that there exists an analytic function $f:\Om\to\R$ with $\Om\cap \Sigma=f^{-1}(0)$ and $|df|\neq 0$ on the whole $\Om$. 
It remains to prove that $\overline{\Sigma}$ is a $\mathscr{C}^1$ hypersurface also at the points that do not belong to $\Sigma$. Let then $p\in\overline{\Sigma}\setminus\Sigma$ and let $\Omega\ni p$ be a small open neighborhood. From the \L ojasiewicz Structure Theorem~\cite[Theorem~6.3.3]{Kra_Par}, it follows that we can choose $\Om$ small enough so that
\begin{equation}
\label{eq:Sigma1Sigma2..}
\overline{\Sigma}\cap \Om\,=\,\overline{\Sigma}_1\cup\dots\cup\overline{\Sigma}_k
\end{equation}
for some $k\in\N$, where the $\Sigma_i$'s are connected analytic hypersurfaces contained in the top stratum $\Sigma$ and $p\in\overline{\Sigma}_i$ for all $i=1,\dots,k$.
For every $i=1,\dots,k$ and for every $x\in\Sigma_i$, let us denote by ${\rm n}_i(x)$ the unit normal to $\Sigma_i$ at the point $x$.

We now show that the following limit
\begin{equation}
\label{eq:limit_reg_Sigma}
\lim_{x\in\Sigma_i,\,x\to p} {\rm n}_i(x)
\end{equation}
exists for every $i=1,\dots,k$. To this end, suppose that this is not the case, that is, suppose that, for some $i$, there exists a sequence of points $\{x_j\}_{j\in\N}$ on $\Sigma_i$ with $x_j\to p$ as $j\to\infty$ and such that the sequence of normal vectors ${\rm n}_i(x_j)$ does not converge.
Considering an orthonormal basis ${\rm n}_i(x_j),X_2(x_j),\dots,X_n(x_j)$ of $T_{x_j}M$, we easily see from formula~\eqref{eq:analytic_expansion_2} that the hessian $\DD u$ at the point $x_j$ is represented by the following matrix
\begin{equation}
\label{eq:hessian_matrix}
\begin{bmatrix} 
    -n\umax & 0      & \dots  & 0 \\
    0       & 0      & \dots  & 0 \\
    \vdots  & \vdots & \ddots & \vdots \\
    0       & 0      & \dots  & 0
\end{bmatrix}\,\,.
\end{equation}
Since the normal vectors ${\rm n}_i(x_j)$ belong to $\Sph^{n-1}$ (viewed as a subspace of $T_{x_j}M$), which is compact, we can find two subsequences $\{x_{j_k}^{(1)}\}, \{x_{j_k}^{(2)}\}$ of $\{x_j\}_{j\in\N}$ such that the corresponding normal vectors ${\rm n}_i(x^{(1)}_{j_k}),{\rm n}_i(x^{(2)}_{j_k})$ converge to $X^{(1)},X^{(2)}\in\Sph^{n-1}$ (viewed as a subspace of $T_p M$) with $X^{(1)}\neq \pm X^{(2)}$.
Up to pass to subsubsequences, we can also suppose that $X_2(x_{j_k}^{(\ell)}),\dots,X_n(x_{j_k}^{(\ell)})$ converge to some vectors $X_2^{(\ell)},\dots,X_n^{(\ell)}$ in $T_pM$ for $\ell=1,2$. Notice that the continuity of $\go$ grants us that $X^{(1)},X_2^{(1)},\dots,X_n^{(1)}$ and $X^{(2)},X_2^{(2)},\dots,X_n^{(2)}$ are both orthonormal bases of  $T_pM$.
Since $u$ is analytic, its hessian is continuous, hence passing to the limit along the subsequence $\{x_{j_k}^{(1)}\}$ we deduce that $\DD u$ at the point $p$ is represented by the matrix~\eqref{eq:hessian_matrix} with respect to the basis $X^{(1)},X^{(1)}_2,\dots,X^{(1)}_n$.
Analogously, if we take the limit of $\DD u$ along the second subsequence, we get that $\DD u$ at the point $p$ is represented by the matrix~\eqref{eq:hessian_matrix} with respect to $X^{(2)},X^{(2)}_2,\dots,X^{(2)}_n$. On the other hand, since $X^{(1)}\neq\pm X^{(2)}$, it is clear that the change of basis from $X^{(1)},X^{(1)}_2,\dots,X^{(1)}_n$ to $X^{(2)},X^{(2)}_2,\dots,X^{(2)}_n$ must modify the first line and the first row of the hessian matrix, hence we have a contradiction. Therefore necessarily the limit in~\eqref{eq:limit_reg_Sigma} must exist. 
Let us also observe that the limit in~\eqref{eq:limit_reg_Sigma} cannot depend on the index $i \in\{1,\dots,k \}$, otherwise we could repeat the same argument working with two sequences on two different hypersurfaces to obtain the same contradiction. Therefore, up to a suitable choice of the orientation, all of the $\overline{\Sigma}_i$'s share the same normal at the point $p \in \overline{\Sigma} \setminus \Sigma$. In particular, the tangent space $T_p\overline{\Sigma}$ is well defined.

%
 
\begin{figure}
 \centering
 \subfigure[Transversal point~\label{fig:sing1}]
   {\includegraphics[scale=1.4]{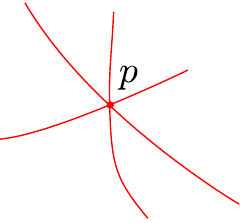}}
 \hspace{15mm}
 \subfigure[Cuspidal point\label{fig:sing2}]
   {\includegraphics[scale=1.4]{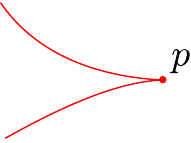}}
   \hspace{15mm}
 \subfigure[Touching point\label{fig:sing3}]{\includegraphics[scale=1.4]{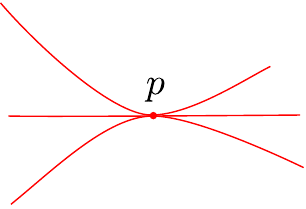}}
 \caption{\small Visual $1$-dimensional representation of the possible singularities of $\overline{\Sigma}$. The first part of the proof of Proposition~\ref{pro:SigmaC1} is concerned with showing that the normal to $\overline{\Sigma}$ is well defined everywhere, thus ruling out transversal singularities like the one pictured in~\ref{fig:sing1}. To exclude folding singularities (cuspidal points~\ref{fig:sing2} or multiple hypersurfaces touching tangentially~\ref{fig:sing3}) one needs a different argument, which is presented in the second part of the proof. 
 \label{fig:sing_Sigma}}
 \end{figure}

We are not finished yet, as it may happen that $p$ is a cuspidal point or that there are multiple hypersurfaces touching at $p$, see~\figurename~\ref{fig:sing_Sigma}. 
We now show that such folding singularities cannot happen by proving that $\Sigma$ is uniformly distant from itself along its normal direction.
Let us start by considering a point $x\in\Sigma$ and a unit speed geodesic $\gamma$ with $\gamma(0)=x$ and $\dot\gamma(0)$ orthogonal to $\Sigma$. Since $M$ is compact and complete, there exists a positive constant $K>0$ such that every such $\gamma(t)$ exists (that is, it does not reach the boundary $\pa M$) and is smooth for $|t|\leq K$. From the above observations on the hessian of $u$, it follows that the restriction of $u$ to $\gamma$ satisfies the following
$$
(u\circ\gamma)(t)\,=\,\umax\,-\,\frac{n}{2}\,\umax\, t^2\,+\,\sigma(t)\,,
$$
where $\sigma(t)$ is an error term such that, for all $t$,
$$
|\sigma(t)|\,\leq\,\frac{(u\circ\gamma)'''(\xi)}{3!}\,|t|^3\,,
$$
for some $\xi\in\R$ with $|\xi|<t$. Since $u$ and $\gamma$ are smooth and $M$ is compact, in particular we have that the derivatives of $u\circ\gamma$ are bounded, hence there exists a constant $C$ such that $|\sigma(t)|< C\,|t|^3$. But this implies that, for every $|t|\leq \min\{n\umax/(2C),K\}$ it holds
$$
(u\circ\gamma)(t)\,=\,\umax\,-\,\frac{n}{2}\,\umax\, t^2\,+\,\sigma(t)\,<\,\umax\,-\,\left(\frac{n\umax}{2C}\,-\,|t|\right)C\,t^2\,\leq\,\umax\,.
$$
This means that every point of $\gamma(t)$ with $|t|\leq \min\{n\umax/(2C),K\}$, $t\neq 0$, does not belong to ${\rm MAX}(u)$. This must hold for every unit speet geodesic starting from a point of $\Sigma$ and directed orthogonally to $\Sigma$. From this property it follows that there cannot be folding pathologies at our point $p$. In fact, if there exist $\Sigma_1,\Sigma_2$ in the decomposition~\eqref{eq:Sigma1Sigma2..} that form a fold, this would mean that orthogonal geodesics starting from points of $\Sigma_1$ arbitrarily close to $p$ would intersect $\Sigma_2$ (which is also contained in ${\rm MAX}(u)$) for arbitrarily small values of $t$. This is in contradiction with what we have just proven, hence such singularities cannot exist. This concludes our proof.
\end{proof}

\begin{remark}
The authors are deeply indebted to P.~T.~Chru\'sciel, who suggested the argument used in Proposition~\ref{pro:SigmaC1} to rule out folding singularities.
\end{remark}

\begin{remark}
We emphasize that in Proposition~\ref{pro:SigmaC1} we are not claiming any regularity of the boundary of $\overline{\Sigma}$, which in principle can be any stratified submanifold of dimension $n-2$. However, we will need Proposition~\ref{pro:SigmaC1} only in the study of the separating hypersurface of a $2$-sided solution, and it is clear that such hypersurface has empty boundary. In other words, Proposition~\ref{pro:SigmaC1} tells us that the separating hypersurface of a $2$-sided solution is $\mathscr{C}^1$. In Section~\ref{sec:ass_SD}, starting from this, we will refine the analysis to show that such hypersurface is actually $\mathscr{C}^\infty$.
\end{remark}

\section{The cylindrical ansatz}
\label{sec:cyl_ans}

In Subsection~\ref{sub:conformal reformulation_SD} we will finally use the pseudo-radial function $\Psi$ to set up our cylindrical ansatz. More precisely, on a region $N$ of our initial manifold, we will consider the new metric
$$
g\,\,=\,\,\frac{\go}{\Psi^2}\,,
$$
and we will also define a {\em pseudo-affine function} $\ffi$. 
The definitions are chosen in such a way that, if $(M,\go,u)$ is isometric to the Schwarzschild--de Sitter solution, then the metric $g$ is just the standard cylindrical metric and $\ffi$ is an affine function, that is, the norm of $\na\ffi$ with respect to the metric $g$ is constant on $M$ (here we have denoted by $\na$ the Levi-Civita connection of $g$). Conversely, the general idea in the future proofs will be to find opportune conditions that force $\ffi$ to be affine and $g$ to be cylindrical, thus proving the isometry with the Schwarzschild--de Sitter solution. The highlight of this subsection is Proposition~\ref{pro:conf_ref_SD}, where we will translate the equations in problem~\eqref{eq:prob_SD} in terms of $g$ and $\ffi$.

In Subsection~\ref{sub:geom_levelsets_SD} we will analyze the level sets of $\ffi$, and in particular we will write down the relations between the mean curvature and second fundamental form of the level sets with respect to $\go$ and $g$. 

In Subsection~\ref{sub:bochner_minpr_SD} we will apply the Bochner formula and the equations of the conformal reformulation of problem~\eqref{eq:prob_SD} written down in Proposition~\ref{pro:conf_ref_SD}, in order to deduce an elliptic inequality for the quantity
$$
w=\beta\left(1-|\na\ffi|_g^2\right)\,,
$$ 
where $\beta$ is a suitably chosen positive function. A Minimum Principle argument, together with an estimate on the behavior of $|\na\ffi|_g$ near ${\rm MAX}(u)$ (which is provided by the reverse {\L}ojasiewicz inequality proved in Subsection~\ref{sub:preparatory_estimates_SD}) will allow us to prove that $w$ is positive on our region $N$. This will give us an important bound from above on the norm of the gradient of $\ffi$, which will be of great importance in the next sections.

As a first consequence of this bound on $|\na\ffi|_g$, in Subsection~\ref{sub:monotonicity_SD} we will prove the monotonicity along the level sets of $u$ of the function $\Phi$ defined in~\eqref{eq:Phi_SD}. From this we will deduce an area lower bound for the boundary of our region $N$.

\subsection{Conformal reformulation of the problem}
\label{sub:conformal reformulation_SD}

Let $(M,\go,u)$ be a solution to problem~\eqref{eq:prob_SD}, and let $N$ be a connected component of $M\setminus{\rm MAX}(u)$.
As already observed, when $(M,\go,u)$ is the Schwarzschild--de Sitter solution, the pseudo-radial function $\Psi=\psi\circ u$, defined by~\eqref{eq:pr_function_+} or by~\eqref{eq:pr_function_-} depending on whether $N$ is outer or inner, coincides with the radial coordinate $|x|$, provided the parameter $m$ in the definition of $\Psi$ coincides with the virtual mass of $N$. As anticipated, we want to proceed via a cylindrical ansatz, that is, on $N$ we consider the following conformal change
\begin{equation}
\label{eq:g_SD}
g=\frac{g_0}{\Psi^2}\,,
\end{equation}
and we rephrase problem \eqref{eq:prob_SD} in terms of $g$. We fix local coordinates in $N$ and we compute the relation between the Christoffel symbols $\Gamma_{\alpha\beta}^{\gamma},G_{\alpha\beta}^{\gamma}$ of $g,g_0$
\begin{equation}
\label{eq:christoffels_SD}
\Gamma_{\alpha\beta}^{\gamma}=G_{\alpha\beta}^{\gamma}-\frac{1}{\psi}\left(\delta_{\alpha}^{\gamma}\partial_{\beta}\Psi+\delta_{\beta}^{\gamma}\partial_{\alpha}\Psi-(g_0)_{\alpha\beta}(g_0)^{\gamma\eta}\partial_{\eta}\Psi\right)\,.
\end{equation}
Denote by $\nabla,\Delta_g$ the Levi-Civita connection and the Laplace-Beltrami operator of $g$. For every $z\in\mathscr{C}^{\infty}$, we compute
\begin{equation}
\label{eq:hessian_SD}
\nabla_{\alpha\beta}^2 z=\D_{\alpha\beta}^2 z+\frac{1}{\psi}\Big(\partial_{\alpha}z\partial_{\beta}\Psi+\partial_{\alpha}\Psi\partial_{\beta}z-\langle\D z\,|\D\Psi\,\rangle\,\cgo_{\alpha\beta}\Big)
\end{equation}
\begin{equation}
\label{eq:laplacian_SD}
\Delta_g z=\psi^2\Delta z-(n-2)\psi\langle\D z\,|\D\Psi\,\rangle
\end{equation}
Substituting $z=\Psi$ in formul\ae~\eqref{eq:hessian_SD} and~\eqref{eq:laplacian_SD}, and using the equations in~\eqref{eq:prob_SD} we compute
\begin{align}
\notag
\nana \Psi&=\DD \Psi+\frac{1}{\psi}\Big(2d\Psi\otimes d\Psi-|\D \Psi|^2\,\go\Big)
\\
\notag
&=\dot\psi\DD u+\left(\frac{1}{u\dot{\psi}}+\frac{n+1}{\psi}+n\frac{\dot{\psi}}{u}\right)d\Psi\otimes d\Psi-\frac{|\D\Psi|^2}{\psi}\go
\\
\label{eq:hessian_psi_SD}
&=u\dot\psi \Ric +\left(\frac{1}{u\dot{\psi}}+\frac{n+1}{\psi}+n\frac{\dot{\psi}}{u}\right)d\Psi\otimes d\Psi-\frac{|\na\Psi|_g^2}{\psi}g-nu\psi^2\dot{\psi}\,g\,,
\\
\notag
\Delta_g \Psi&=\psi^2\Delta \Psi-(n-2)\psi|\D\Psi|^2
\\
\notag
&=\psi^2\dot\psi\De u+\left(\frac{\psi^2}{u\dot\psi}+\psi+n\frac{\psi^2\dot{\psi}}{u}\right)|\D\Psi|^2
\\
\label{eq:laplacian_psi_SD}
&=-nu\psi^2\dot\psi+\left(\frac{1}{u\dot\psi}+\frac{1}{\psi}+n\frac{\dot{\psi}}{u}\right)|\na\Psi|_g^2\,.
\end{align}
On the other hand, we know from~\cite[Theorem~1.159]{Besse} that the Ricci tensors of $\go$ and $g$ are related by the formula
\begin{align}
\notag
\!\!\Ric&=\Ricg-\frac{n-2}{\psi}\nana\Psi+\frac{2(n-2)}{\psi^2}d\Psi\otimes d\Psi-\left(\frac{1}{\psi}\Deg\Psi+\frac{n-3}{\psi^2}|\na\Psi|_g^2\right) g
\\
\label{eq:ricci_SD}
&=\Ricg-\frac{n-2}{\psi}\nana\Psi+\frac{2(n-2)}{\psi^2}d\Psi\otimes d\Psi+\left[nu\psi\dot{\psi}-\bigg(n-2+n\frac{\psi\dot{\psi}}{u}+\frac{\psi}{u\dot{\psi}}	\bigg)\frac{|\na\Psi|_g^2}{\psi^2}\right] g\,.
\end{align}
Substituting~\eqref{eq:ricci_SD} in~\eqref{eq:hessian_psi_SD} we obtain
\begin{multline}
\label{eq:riccig_SD}
\Ricg=\bigg[\frac{n-2}{\psi}+\frac{1}{u\dot{\psi}}\bigg]\nana\Psi-\bigg[\frac{2(n-2)}{\psi^2}+\frac{n}{u^2}+\frac{1}{u^2\dot{\psi}^2}+\frac{n+1}{u\psi\dot\psi}\bigg]d\Psi\otimes d\Psi
\\
+\left[n\psi\Big(\psi-u\dot{\psi}\Big)+\bigg(n-2+n\frac{\psi\dot{\psi}}{u}+\frac{2\psi}{u\dot{\psi}}	\bigg)\frac{|\na\Psi|_g^2}{\psi^2}\right] g\,.
\end{multline}
In order to simplify the above expressions, we notice that, {\em a posteriori}, in the rotationally symmetric case we expect the equality $|\D u|^2=(u/\dot\psi)^2$, or equivalently $|\na\Psi|_g^2=(u\psi)^2$, to hold pointwise on $N$. For this reason, it is convenient to introduce a function $\ffi\in\mathscr{C}^{\infty}(N)$ which satisfies $|\na\ffi|^2_g=|\na\Psi|_g^2/(u\psi)^2$, so that, {\em a posteriori}, we expect $|\na\ffi|_g=1$ pointwise on $N$, that is, we expect $\ffi$ to be an affine function.
Such a function $\ffi$ can be defined in several ways. In fact, if $\ffi$ is such a function, also $c\pm\ffi$, with $c\in\R$, satisfies the same equality $|\na\ffi|_g=1$. However, all these choices are actually equivalent for our analysis, hence we will fix $\ffi$ now, once and for all. We define the {\em pseudo-affine function} $\ffi$ as
\begin{equation}
\label{eq:ffi_SD}
\ffi(p)\,=\,\int_{\Psi(p)}^{r_+(m)} \frac{dt}{t \sqrt{1-t^2-2mt^{2-n}}}\,.
\end{equation}
Despite the integrand has a singularity for $t=r_\pm(m)$, the integral in~\eqref{eq:ffi_SD} is finite. In fact, setting $s=1-t^2-2mt^{2-n}$, fixed $\eta>[(n-2)m]^{1/n}$, we have
\begin{align*}
\int_{\eta}^{r_+(m)} \frac{dt}{t \sqrt{s}}
\,&=\,\int^{0}_{1-\eta^2-2m\eta^{2-n}} \frac{-ds}{2t^2[1-(n-2)mt^{-n}] \sqrt{s}}
\\
&\leq\,\frac{1}{2\eta^2[1-(n-2)m\eta^{-n}]}\int^{1-\eta^2-2m\eta^{2-n}}_{0} \frac{ds}{\sqrt{s}}
\\
&=\frac{\sqrt{1-\eta^2-2m\eta^{2-n}}}{\eta^2[1-(n-2)m\eta^{-n}]
}\,<\,\infty\,.
\end{align*}
The singularity  of the integrand when $t=r_-(m)$ can be handled in the same way. It follows that $\ffi$ is well defined and smooth on $N$. However, {\em a priori} we do not know if the gradient of $\ffi$ is bounded when we approach ${\rm MAX}(u)$, because both the numerator and the denominator of formula~\eqref{eq:na_ffi_SD} below go to zero. This point will be addressed in Proposition~\ref{pro:min_pr_SD}, where we will show that $|\na\ffi|_g$ is bounded above by $1$ on the whole $N$.
Notice that the definition of $\ffi$ is chosen in such a way that, when $N$ is outer and $p\in\pa N$, we have $\ffi=0$ on $\pa N$.
Instead, when $N$ is inner and $p\in\pa N$, that is, $\Psi(p)=r_-(m)$, the function $\ffi$ assumes its maximum value. 

For future convenience, we also write down some formul\ae\ for the gradient and the hessian of $\ffi$
\begin{align}
\label{eq:na_ffi_SD}
|\na\ffi|_g^2\,&=\,\frac{|\na\Psi|_g^2}{u^2\psi^2}\,=\,\frac{\dot\psi^2}{u^2}\,|\D u|^2\,=\,\frac{|\D u|^2}{\psi^2\big[1-\big(r_0(m)/\psi\big)^n\big]^2}\,,
\\
\label{eq:dffi_SD}
\na\ffi\,&=\,-\frac{\dot{\psi}}{u\psi}\D u\,=\,\frac{\D u}{\psi^2\big[1-\big(r_0(m)/\psi\big)^n\big]}
\,,
\\
\notag
\nana\ffi\,&=\,-\,\frac{\nana\Psi}{\psi u}\,+\,\frac{1}{\psi^2 u^2}\bigg(u+\frac{\psi}{\dot\psi}\bigg)d\Psi\otimes d\Psi
\\
\label{eq:nana_ffi_SD}
&=\,-\frac{\dot\psi}{u\psi}\DD u\,+\,\frac{\dot\psi^2}{u\psi^2}|\D u|^2\go\,-\,n\frac{\dot\psi^2}{u^2\psi^2}\left(u+\psi\dot\psi\right)du\otimes du
\end{align}
Combining equations~\eqref{eq:na_ffi_SD},~\eqref{eq:nana_ffi_SD} with~\eqref{eq:laplacian_psi_SD},~\eqref{eq:riccig_SD}, we arrive with some computations to a conformal reformulation of system~\eqref{eq:prob_SD}.

\begin{proposition}
\label{pro:conf_ref_SD}
Let $(M,\go,u)$ be a solution to problem~\eqref{eq:prob_SD}, and let $N$ be an outer or inner region with virtual mass
$$
m\,\,=\,\,\mu(N,\go,u)\,.
$$
Let also $\Psi=\psi\circ u$ be the pseudo-radial function defined by~\eqref{eq:pr_function_+} or~\eqref{eq:pr_function_-}, depending on whether $N$ is an outer or inner region, respectively. Then the metric $g=\go/\Psi^2$ and the pseudo-affine function $\ffi$ defined in~\eqref{eq:ffi_SD} satisfy the following system of differential equations
\begin{equation}
\label{eq:pb_conf_SD}
\!\!\!\!\!\!\!\!\!\!\!\!\begin{dcases}
\Ricg=-\Big[(n-2)u+\frac{\psi}{\dot{\psi}}\Big]\nana\ffi-(n-2)d\ffi\otimes d\ffi+\bigg[(n-2)|\na\ffi|_g^2-\Big(u-\frac{\psi}{\dot{\psi}}\Big)\Deg\ffi\bigg]g, & \mbox{in } N,
\\
\Deg \ffi=n\psi\dot{\psi}\left(1-|\na\ffi|_g^2\right), & \mbox{in } N,
\end{dcases}
\end{equation}
with boundary conditions
\begin{align}
\label{eq:init_cond}
&\begin{dcases}
\ffi=0  & \mbox{on } \partial N,
\\
\ffi=\ffi_0:= \int_{r_0(m)}^{r_+(m)} \frac{dt}{t \sqrt{1-t^2-2mt^{2-n}}} & \mbox{on } \overline{N}\cap {\rm MAX}(u),
\end{dcases}
\ \ \ \hbox{ if $N$ outer\,,}
\\
\smallskip
&\begin{dcases}
\ffi=\ffi_{\rm max}:= \int_{r_-(m)}^{r_+(m)} \frac{dt}{t \sqrt{1-t^2-2mt^{2-n}}}  & \mbox{on } \partial N,
\\
\ffi=\ffi_0:= \int_{r_0(m)}^{r_+(m)} \frac{dt}{t \sqrt{1-t^2-2mt^{2-n}}} & \mbox{on } \overline{N}\cap {\rm MAX}(u),
\end{dcases}
\ \ \ \hbox{ if $N$ inner\,.}
\end{align}
\end{proposition}

\noindent 
Tracing the first equation of~\eqref{eq:pb_conf_SD}, one obtains
\begin{align}
\label{eq:tildeR_SD}
\frac{\RRR_g}{(n-1)(n-2)}&=1-\left(1+\frac{2nu\psi\dot{\psi}-n\psi^2}{n-2}\right)\left(1-|\nabla\varphi|_g^2\right)
\end{align}
where $\Rg$ is the scalar curvature of $g$.
In the cylindrical situation, which is the conformal counterpart of the Schwarzschild--de Sitter solution, $\Rg$ has to be constant. In this case, the above formula implies that also $|\na \ffi|_g$ has to be constant and equal to $1$, as already anticipated. 
For these reasons, also in the  situation, where we do not know a priori if $g$ is cylindrical, it is natural to think of $\na \ffi$ as to a candidate splitting direction and to investigate under which conditions this is actually the case. 
A first important observation is that the splitting is in force when $\ffi$ is an affine function, that is, when its hessian $\nana\ffi$ and the quantity $1-|\na\ffi|_g$ vanish everywhere in our region.

\begin{proposition}
\label{pro:rigidity}
Let $(M,\go,u)$ be a solution to problem~\eqref{eq:prob_SD}, and let $N$ be an outer or inner region with virtual mass
$$
m\,\,=\,\,\mu(N,\go,u)\,.
$$
Let $\Psi=\psi\circ u$ be the pseudo-radial function defined by~\eqref{eq:pr_function_+} or~\eqref{eq:pr_function_-}, depending on whether $N$ is an outer or inner region, respectively. Finally, let $g=\go/\Psi^2$ and let $\ffi$ be the pseudo-affine function defined by~\eqref{eq:ffi_SD}.
If $\nana\ffi\equiv 0$ and $|\na\ffi|_g\equiv 1$ on $N$, then $(M,\go,u)$ is isometric to a generalized Schwarzschild--de Sitter solution~\eqref{eq:gen_SD} with mass $m$.
\end{proposition}

\begin{proof}
Let us suppose that $N$ is an outer region, the inner case being completely equivalent. Proceeding as in the proof of \cite[Theorem~4.1]{Ago_Maz_1}, we obtain that $(\{0\leq\ffi< \ffi_0\},g)$ is isometric to the product
$$
\big([0,\ffi_0)\times \pa N\,,\,d\ffi\otimes d\ffi+g^{\pa N}\big)\,,
$$
where $g^{\pa N}$ is the metric induced by $g$ on $\pa N$. From the first equation in~\eqref{eq:pb_conf_SD} we deduce that $\Ric_{g^{\pa N}}=(n-2)g^{\pa N}$. Recalling the definition of $\ffi$ and the relation between $g$ and $\go$, we deduce that $\go$ is isometric to 
$$
\frac{d\Psi\otimes d\Psi}{u^2}+\psi^2g^{\pa N}\,.
$$
This proves that $(N,\go,u)$ is isometric to the outer region of a generalized Schwarzschild--de Sitter solution $(M^s,\go^s,u^s)$ defined by~\eqref{eq:gen_SD}, where $\Psi$ is the radial coordinate. It remains to prove that this isometry between the outer regions extends to an isometry between the whole $(M,\go,u)$ and the whole $(M^s,\go^s,u^s)$. 
To this end, we distinguish two cases, depending on whether the (possibly stratified) hypersurface $\Sigma_N=\overline{N}\cap\overline{M\setminus \overline{N}}\subseteq{\rm MAX}(u)$ is orientable or not.
\begin{itemize}
\item Let us start by considering the case in which $\Sigma_N$ is an orientable hypersurface. Since $M$ is orientable by hypothesis, the hypersurface $\Sigma_N$ is orientable if and only if it is two sided, meaning that any neighborhood of $\Sigma_N$ contains both points of $N$ and points outside $N$. Considering the corresponding chart in $(M^s,\go^s,u^s)$, by analytic continuation we can extend the isometry between $(N,\go,u)$ and the outer region of $(M^s,\go^s,u^s)$ to all the points in the chart. That way, the isometry pass through $\Sigma$, and we can continue to argue chart by chart until we finally cover all the manifold $M$, thus proving the global isometry of $(M,\go,u)$ and $(M^s,\go^s,u^s)$.
\smallskip
\item If $\Sigma_N$ is not orientable, this means that it is one sided, that is, every point of $\Sigma_N$ has a neighborhood that is entirely contained inside $\overline{N}$.  Therefore, it easily follows that $(M,\go,u)=(\overline{N},\go,u)$ is isometric to $(\overline{M}_+^s,\go^s,u^s)/\sim$, where $\sim$ is a relation on the points of 
$$
{\rm MAX}^s(u)\,=\,\{p\in M^s\,:\,u^s(p)=\umax\}\,\subset\,\pa\overline{M}_+^s\,.
$$
We first observe that this relation is induced by an involution 
$$
\iota_\sim:\,{\rm MAX}^s(u)\,\to\,{\rm MAX}^s(u)\,.
$$
In fact, the neighborhood of a point $x\in {\rm MAX}^s(u)$ inside $(\overline{M}_+^s,\go)$ is isometric to an half space $\R_+^{n}$ endowed with a metric such that the boundary $\pa \R^n_+$ is smooth. In order for the manifold $(\overline{M}_+^s,\go)/\sim$ to be smooth at $x$ it is necessary that there exists exactly one point $\iota_\sim(x)\in{\rm MAX}^s(u)$, $\iota_\sim(x)\neq x$, such that $x\sim \iota_\sim(x)$. Moreover, $\iota_\sim$ has to be continuous and to reverse the orientation. It is also clear that $\iota_\sim^2$ is the identity, so that $\iota_\sim$ is indeed an involution.

We now notice that the mean curvature vector $\vec{\HHH}$ of the hypersurface ${\rm MAX}^s(u)$ has constant nonzero norm and it always points outside $M_+^s$ on the whole ${\rm MAX}^s(u)$, hence the same holds on ${\rm MAX}^s(u)/\sim$. In particular, at the points $x$ and $\iota_\sim(x)$ the vector $\vec{\HHH}$ points outside $M_+^s$. Therefore, in a chart centered at $x=\iota_\sim(x)$ in the quotient manifold we would have that $\vec{\HHH}$ points in one direction according to the measure at $x$, and points in the opposite direction if measured at $\iota_\sim(x)$, which means that the mean curvature of $\Sigma_N$ at $x$ is not well defined. Since the same reasoning can be repeated at each point $x\in\Sigma_N$, we would have that the mean curvature in nowhere defined on $\Sigma_N$, against the fact that we know that $\Sigma_N={\rm MAX}(u)$ is a stratified hypersurface, so that in particular it is smooth $\mathscr{H}^{n-1}$-almost everywhere. We have reached a contradiction, hence $\Sigma_N$ is necessarily oriented and the first case applies.
\end{itemize}
This concludes the proof.
\end{proof}

\subsection{The geometry of the level sets.}
\label{sub:geom_levelsets_SD}
In the forthcoming analysis a crucial role is played by the study of the geometry of the level sets of $\ffi$, which coincide with the level sets of $u$ in $N$, by definition. 
Hence, we pass now to describe the second fundamental form and the mean curvature of the regular level sets of $\ffi$ (or equivalently of $u$) in both the original Riemannian context $(N, \go)$ and the conformally related one $(N, g)$.
To this aim, we fix a regular level set $\{ \ffi = s_0\}$ and we construct a suitable set of coordinates in a neighborhood of it. Note that $\{ \ffi = s_0\}$ must be compact, by the properness of $\ffi$. In particular, there exists a real number $\delta>0$  such that in the tubular neighborhood $\mathcal{U}_\delta = \{s_0 - \delta < \ffi < s_0 + \delta \}$ we have $|\na \ffi |_\g > 0$ so that $\mathcal{U}_\delta$ is foliated by regular level sets of $\ffi$. As a consequence, $\mathcal{U}_\delta$ is diffeomorphic to $(s_0 -\delta , s_0 + \delta) \times \{ \ffi = s_0 \}$ and the function $\ffi$ can be regarded as a coordinate in $\mathcal{U}_\delta$. Thus, one can choose a local system of coordinates $\{\ffi,\vartheta^1\!,\!....,\vartheta^{n-1}\}$, where $\{\vartheta^1\!,\!...., \vartheta^{n-1} \}$ are local coordinates on $\{ \ffi = s_0 \}$. In such a system, the metric $\g$ can be written as
\begin{equation*}
g \, = \,\frac{d\ffi \otimes d\ffi}{|\na \ffi|_\g^2} +\g_{ij}(\ffi,\vartheta^1 \!,\!...., \vartheta^{n-1})\,d\vartheta^i\!\otimes d\vartheta^j \,,
\end{equation*}
where the latin indices vary between $1$ and $n-1$. We now fix in $\mathcal{U}_\delta$ the $\g$-unit vector field $\nu_g=-\na\ffi/|\na\ffi|_g$. We also define $\nu$ as the $\go$-unit vector field that points in the same direction as $\nu_g$ at every point, that is
$$
\nu\,=\,
\begin{cases}
-\D u/|\D u|\,,  &  \mbox{ if $N$ outer},
\\
\D u/|\D u|\,,  &  \mbox{ if $N$ inner}.
\end{cases}
$$
We will denote by $\hhh$ and $\HHH$ the second fundamental form and mean curvature with respect to the metric $\go$ and the normal $\nu$. We will denote by $\hg$ and $\Hg$ the second fundamental form and mean curvature with respect to the metric $g$ and the normal $\nu_g$.
According to these choices,  the second fundamental forms of the level sets of $u$ or $\ffi$ are given by
\begin{equation}
\label{eq:formula_fundamentalform_SD}
\cho_{ij}=
\begin{dcases}
-\frac{\DD_{ij} u}{|\D u|}\,,  &  \mbox{ if $N$ outer},
\\
\frac{\DD_{ij} u}{|\D u|}\,,  &  \mbox{ if $N$ inner},
\end{dcases}
\qquad\qquad
\chg_{ij}=-\frac{\nana_{ij}\ffi}{|\na\ffi|_\g}\, ,  \qquad \mbox{for}\quad i,j = 1,\!...., n-1.
\end{equation}  
Taking the traces of the above expressions with respect to the induced metrics 
we obtain the following expressions for the mean curvatures in  the two ambients
\begin{equation}
\label{eq:formula_curvature_SD}
\Ho=
\begin{dcases}
-\frac{\De u}{|\D u|}+\frac{\DD u(\D u,\D u)}{|\D u|^3}\,,  &  \mbox{ if $N$ outer},
\\
\frac{\De u}{|\D u|}-\frac{\DD u(\D u,\D u)}{|\D u|^3}\,,  &  \mbox{ if $N$ inner},
\end{dcases}
\qquad\qquad
\Hg=-\frac{\Deg \ffi}{|\na \ffi|_g}+\frac{\nana\ffi(\na\ffi,\na\ffi)}{|\na\ffi|_{\g}^3}\,.
\end{equation}
Taking into account expressions~\eqref{eq:na_ffi_SD},~\eqref{eq:nana_ffi_SD}, one can show that 
the second fundamental forms are related by
\begin{equation}
\label{eq:formula_h_h_g_SD}
\chg_{ij}
\,=\,
\frac{1}{\psi}\,\cho_{ij}\,-\,\frac{|\dot{\psi}|}{\psi^2}\,|\D u|\,  \cgo_{ij}\,.
\end{equation}
The analogous formula for the mean curvatures reads
\begin{equation}
\label{eq:formula_H_H_g_SD}
{\Hg}
\,=\,
\psi\,\HHH \,-\, (n-1)\,|\dot{\psi}|\,|\D u|\,.
\end{equation}

Concerning the nonregular level sets of $\ffi$, we first observe that, since $u$ is analytic on $M$ (see~\cite{Chrusciel_1,ZumHagen}), then  $\ffi$ is analytic on the whole $N$. 
As anticipated in Subsection~\ref{sub:prelim}, it follows from the results in~\cite{Lojasiewicz_2} (see also~\cite[Theorem~6.3.3]{Kra_Par}) that there exists an hypersurface $S\subseteq{\rm Crit}(\ffi)$ such that $\mathscr{H}^{n-1}({\rm Crit}(\ffi)\setminus S)=0$.
In particular, the $(n-1)$-dimensional Hausdorff measure of the level sets of $\ffi$ is locally finite. Moreover, the unit normal to a level set is well-defined $\mathscr{H}^{n-1}$-almost everywhere, and so are the second fundamental form $\hg$ and the mean curvature $\Hg$. 
We now compute the relation between $\hg,\Hg$ and $\hhh,\HHH$ at a point $y_0\in S$.
 Let $\nu,\nu_g$ be the unit normal vector fields to $S$ at $y_0$ with respect to $\go,g$ respectively. Since $
|\nu_g|_g^2\,=\,1\,=\,|\nu|^2\,=\,\psi^2\,|\nu|_g^2$, we deduce that
$\nu_g\,=\,\psi\,\nu$.
Let $(\pa/\pa x^1,\dots,\pa/\pa x^{n-1})$ be a basis of $T_{y_0}S$, so that in particular $(\pa/\pa x^1,\dots,\pa/\pa x^{n-1},\nu_g)$ is a basis of $T_{y_0}M$. Recalling~\eqref{eq:christoffels_SD} and observing that the derivatives of $u$ and $\Psi$ in $y_0$ are all zero since $y_0\in{\rm Crit}(\ffi)={\rm Crit}(u)$, we have
$$
\chg_{ij}\,=\,\Big\langle \na_i\frac{\pa}{\pa x^j}\,\Big|\,\nu_g\Big\rangle_g\,=\,\Gamma_{ij}^n\,=\,\Cr_{ij}^n\,=\,\Big\langle \D_i\frac{\pa}{\pa x^j}\,\Big|\,\nu_g\Big\rangle_g\,=\,\frac{1}{\psi}\,\Big\langle \D_i\frac{\pa}{\pa x^j}\,\Big|\,\nu\Big\rangle\,=\,\frac{1}{\psi}\,\cho_{ij}\,.
$$
Taking the trace we obtain $\Hg=\psi\HHH$.
This proves that formul\ae~\eqref{eq:formula_h_h_g_SD} and~\eqref{eq:formula_H_H_g_SD} hold also at any point $y_0\in S$, so that in particular they hold $\mathscr{H}^{n-1}$-almost everywhere on any level set.

\subsection{Consequences of the Bochner formula}
\label{sub:bochner_minpr_SD}

Starting from the Bochner formula and using the equations in~\eqref{eq:pb_conf_SD}, we find
\begin{multline}
\Deg|\na\ffi|_g^2\,=\,2|\nana\ffi|_g^2\,+2\Ricg(\na\ffi,\na\ffi)+2\langle\na\Deg\ffi\,|\,\na\ffi\rangle
\\
\label{eq:bochner_SD}
=\,2|\nana\ffi|_g^2-\,\bigg[(n-2)u+\frac{\psi}{\dot{\psi}}+2n\psi\dot{\psi}\bigg]
\,\langle\na|\na\ffi|_g^2\,|\,\na\ffi\rangle_g\,
-\,2\bigg[(n+1)u+n\psi\dot{\psi}\bigg]\,|\na\ffi|_g^2\,\Deg\ffi\,.
\end{multline}
We will use~\eqref{eq:bochner_SD} to compute the laplacian of the function
$$
w=\beta\left(1-|\na\ffi|_g^2\right)\,,\quad\mbox{where }\beta=\psi^2\left|1-(n-2)m\psi^{-n}\right|=\psi^2\left|\frac{u}{\psi\dot{\psi}}\right|.
$$
The function $\beta$ is smooth in $N$. We will denote by $\beta'$ the derivative of $\beta$ with respect to~$\ffi$, more precisely, $\beta'\in\mathscr{C}^{\infty}(N)$ is the function that satisfies $\na\beta=\beta'\na\ffi$.
One computes
\begin{align}
\label{eq:betaprime_SD}
\frac{\beta'}{\beta}\,&=n\psi\dot\psi+(n-2)u\,,
\\
\label{eq:naw_SD}
\na w\,&=\,-\beta\na|\na\ffi|_g^2+\frac{\beta'}{\beta}w\na\ffi\,.
\end{align}
In order to compute the laplacian of $w$, we take the divergence of~\eqref{eq:naw_SD} 
\begin{equation*}
\Deg w\,=\,\,-\frac{\beta'}{\beta}\big\langle\beta\na|\na\ffi|_g^2\,\big|\,\na\ffi\big\rangle_g-\beta\Deg|\na\ffi|_g^2
+\Big(\frac{\beta'}{\beta}\Big)'w|\na\ffi|_g^2+\frac{\beta'}{\beta}\langle\na w\,|\,\na\ffi\rangle_g+\frac{\beta'}{\beta}w\Deg \ffi
\end{equation*}
and using formula~\eqref{eq:bochner_SD} we obtain
\begin{align}
\notag
\Deg w
&=-2\beta\left[|\nana\ffi|_g^2-\frac{(\Deg\ffi)^2}{n}\right]\!-\!\left[(n-2)u+\frac{\psi}{\dot{\psi}}+2n\psi\dot{\psi}-2\frac{\beta'}{\beta}\right]\!\langle\na w|\na\ffi\rangle_g
+n\psi\dot{\psi}\bigg(\frac{\beta'}{\beta}-2\psi\dot{\psi}\bigg)w
\\
\notag
&\ +\bigg[\Big(\frac{\beta'}{\beta}\Big)'-\Big(\frac{\beta'}{\beta}\Big)^2+\bigg((n-2)u+\frac{\psi}{\dot{\psi}}+n\psi\dot{\psi}\bigg)\frac{\beta'}{\beta}+2n(n+1)u\psi\dot{\psi}+2n(n+1)\psi^2\dot{\psi}^2\bigg]|\na\ffi|_g^2w
\\
\label{eq:Degw_SD}
&\leq\,\left[(n-2)u-\frac{\psi}{\dot{\psi}}\right]\langle\na w\,|\,\na\ffi\rangle_g+n(n-2)m\psi^{2-n}\dot{\psi}^2\big[(n-2)+(n+2)|\na\ffi|_g^2\big]w\,.
\end{align}
In particular, $w$ satisfies an elliptic inequality on our connected component $N$ and, as a consequence of the Minimum Principle, we obtain the following relevant bound on the gradient of the pseudo-affine function $\ffi$.

\begin{proposition}
\label{pro:min_pr_SD}
Let $(M,\go,u)$ be a solution to problem~\eqref{eq:prob_SD} and let $N\subseteq M\setminus{\rm MAX}(u)$ be  an outer or inner region with virtual mass $m=\mu(N,\go,u)$. Let also $g,\ffi$ be defined by~\eqref{eq:g_SD},~\eqref{eq:ffi_SD}. 
Then it holds
\begin{equation}
\label{eq:min_pr_SD}
|\na\ffi|_g\leq 1\,,
\end{equation}
on the whole $N$. Moreover, if $|\na\ffi|_g=1$ at a point in the interior of $N$, then $|\na\ffi|_g\equiv 1$ on the whole $N$ and $(M,\go,u)$ is isometric to a generalized Schwarzschild--de Sitter solution~\eqref{eq:gen_SD} with mass $m$.
\end{proposition}

\begin{proof}
We recall from~\eqref{eq:na_ffi_SD} that it holds
$$
|\na\ffi|_g\,=\,\frac{|\D u|}{\psi\left|1-\big(r_0(m)/\psi\big)^n\right|}\,,
$$ 
therefore, from Lemma~\ref{le:bound_psi_SD} we deduce that $w\geq 0$ on $\pa N$. On the other hand, from the definition of $w$, we compute 
$$
w\,=\,\beta(1-|\na\ffi|_g^2)\,=\,\psi^2\left|\frac{u}{\psi\dot{\psi}}\right|-\bigg|\frac{\psi \dot \psi}{u}\bigg|\,|\D u|^2\,.
$$ 
Since $\beta$ goes to $0$ as we approach ${\rm MAX}(u)$, and also $\dot\psi|\D u|^2$ goes to $0$ by Lemma~\ref{le:estimate_upsi_SD}, we have $w\to 0$ as we approach ${\rm MAX}(u)$. 
We also recall that $\ffi, g$, thus also $w$, are analytic in the interior of $N$. As observed in Subsection~\ref{sub:prelim}, this implies that the critical level sets of $w$ are discrete. Therefore there exists $\eta>0$ such that any $0<\ep<\eta$ is a regular values for $w$. In particular, the set 
$$
N_\ep\,=\,\{|w|\geq\ep\}
$$
has a smooth boundary. Since we have already observed that $w\to 0$ as we approach $\pa N$ and ${\rm MAX}(u)$, we have that $N_\ep$ is a compact domain contained in the interior of $N$.
In particular, the coefficients of the elliptic inequality~\eqref{eq:Degw_SD} are bounded in $N_\ep$ and we can apply the Weak Minimum Principle (see for instance~\cite[Corollary~3.2]{Gil_Tru}) to deduce that 
\begin{equation}
\label{eq:min_pr_aux_SD}
\min_{N_\ep} w\,\,\geq\,\, \min_{\pa N_\ep} w\,\,\geq\,\, -\ep\,,
\end{equation}
where in the latter inequality we have used the fact that the boundary of $N_\ep$ is contained in $\{w=\ep\}\cup\{w=-\ep\}$.
Since inequality~\eqref{eq:min_pr_aux_SD} holds for all $0<\ep<\eta$, taking the limit as $\ep\to 0$ we obtain $w\geq 0$ on the whole $N$, and~\eqref{eq:min_pr_SD} follows.

Now we pass to the proof of the second part of the statement. Let $x$ be a point in the interior of $N$ such that $|\na\ffi|_g(x)=1$. In particular it holds $w(x)=0$ and we have proved above that $w\geq 0$ on the whole $N$. Applying the Strong Minimum Principle on an open set $\Omega$ containing $x$, we obtain $w\equiv 0$, or equivalently $|\na\ffi|_g\equiv 1$, on $\Omega$. From the arbitrariness of $\Omega$ we deduce $|\na\ffi|_g\equiv 1$ on $N$, and plugging this information inside the Bochner formula~\eqref{eq:bochner_SD}, we obtain $|\nana\ffi|_g\equiv 0$. We can now invoke Proposition~\ref{pro:rigidity} to conclude.
\end{proof}

\begin{remark}
Proposition~\ref{pro:min_pr_SD} should be compared with~\cite[Proposition~1]{Bei_Sim}, where an analogous result is obtained for a vast class of static perfect fluid solutions. The proof in~\cite{Bei_Sim} is based on the application of the Maximum Principle to an elliptic inequality which seems to be closely related to the one used in our proof.
\end{remark}

\begin{remark}
\label{rem:bound_Du}
Translating the thesis of Proposition~\ref{pro:min_pr_SD} back in terms of $u,\go$, we have that $\dot\psi|\D u|$ is bounded in $N$, 
and recalling formula~\eqref{eq:lim_DPsi_SD}, we deduce that the quantity $|\D u|^2/(\umax-u)$ is bounded on $N$. 
In particular, for any $p\in\overline{N}\cap{\rm MAX}(u)$ there exists a collar neighborhood $p\in \Omega_p\subset \overline{N}$ and a constant $K_p$ such that
\begin{equation}
\label{eq:bounded_quantity_SD}
|\D u|^2(x)\,\,\leq\,\,K_p\,[\umax-u(x)]
\end{equation}
for any $x\in \Omega_p$. The same proof can be repeated on each inner and outer region, and a similar result will also be shown for cylindrical regions, see Proposition~\ref{pro:min_pr_N_SD}. It follows from these considerations that inequality~\eqref{eq:bounded_quantity_SD} is always in force in a neighborhood of any $p\in{\rm MAX}(u)$. This is an improvement of Proposition~\ref{pro:rev_loj_SD} proved above.
\end{remark}


\subsection{Area lower bound.}
\label{sub:monotonicity_SD}

In this subsection, we will study the function
\begin{equation}
\label{eq:Phi_SD}
\Phi(s)=\int_{\{\ffi=s\}}|\na\ffi|_g\,\rmd\sigma_g\,.
\end{equation}
which is defined on $s\in[0,\ffi_0)$ or $s\in(\ffi_0,\ffi_{\rm max}]$ depending on whether $\max_{\pa N}|\D u|/\umax$ is less or greater than $\sqrt{n}$, respectively.
As an application of Proposition~\ref{pro:min_pr_SD}, one can prove the following monotonicity result for $\Phi$.
\begin{proposition}
\label{pro:mon_Phi1_SD}
Let $(M,\go,u)$ be a solution to problem~\eqref{eq:prob_SD}, let $N$ be  a connected component of $M\setminus{\rm MAX}(u)$ with virtual mass $m<\mmax$, and let $\Phi(s)$ be the function defined by~\eqref{eq:Phi_SD}, with respect to the metric $g$ and the pseudo-affine function $\ffi$ defined by~\eqref{eq:g_SD},~\eqref{eq:ffi_SD}.
\begin{itemize}
\smallskip
\item[(i)]
If $N$ is an outer region, then the function $\Phi(s)$, defined for $s\in[0,\ffi_0)$, is monotonically nonincreasing.
Moreover, if $\Phi(s_1)=\Phi(s_2)$ for two different values $0\leq s_1<s_2<\ffi_0$, then the triple $(M,\go,u)$ is isometric to a generalized Schwarzschild--de Sitter triple~\eqref{eq:gen_SD} with mass $m$.
\smallskip
\item[(ii)] If $N$ is an inner region, then the function $\Phi(s)$, defined for $s\in(\ffi_0,\ffi_{\rm max}]$, is monotonically nondecreasing. 
Moreover, if $\Phi(s_1)=\Phi(s_2)$ for two different values $\ffi_0< s_1<s_2\leq\ffi_{\rm max}$, then the solution $(M,\go,u)$ is isometric a generalized Schwarzschild--de Sitter triple~\eqref{eq:gen_SD} with mass $m$.
\end{itemize}
\end{proposition}

\begin{proof}
Consider the case $N$ outer, that is, $\max_{\pa N}|\D u|/\umax<\sqrt{n}$. In particular, the determination of $\psi$ is~\eqref{eq:pr_function_+}, $\ffi\in[0,\ffi_0)$, $\dot{\psi}\leq 0$ and $\Deg \ffi\leq 0$ (this last inequality is a consequence of the second equation of system~\eqref{eq:pb_conf_SD} and of Proposition~\ref{pro:min_pr_SD}). Integrating $\Deg \ffi\leq 0$  in $\{s_1\leq\ffi\leq s_2\}$ for any $0\leq s_1<s_2<\ffi_0$, we get
\begin{equation}
\label{eq:int_lapl_SD}
\int\limits_{\{s_1\leq\ffi\leq s_2\}}\!\!\!\Deg\ffi\,\rmd\sigma_g\,\leq\, 0\,.
\end{equation}
Applying the Divergence Theorem to inequality~\eqref{eq:int_lapl_SD}, we easily obtain $\Phi(s_2)\leq\Phi(s_1)$, therefore $\Phi$ is nonincreasing. 
To prove the rigidity statement, we observe that, if the equality $\Phi(s_1)=\Phi(s_2)$ holds for some $0\leq s_1<s_2<\ffi_0$, then $\Deg\ffi\equiv 0$ on $\{s_1\leq\ffi\leq s_2\}$, hence by the analyticity of $\ffi$ we deduce $\Deg\ffi\equiv 0$ on $N$. Recalling the definition of $\Deg\ffi$, this in turn implies $|\na\ffi|_g\equiv 1$ on $N$. Substituting this information in the Bochner formula~\eqref{eq:bochner_SD} we obtain $|\nana\ffi|_g\equiv 0$, hence we can apply Proposition~\ref{pro:rigidity} to conclude.

If instead $N$ is an inner region, that is, $\max_{\pa N}|\D u|/\umax>\sqrt{n}$, then $\psi$ is as in~\eqref{eq:pr_function_-}, $\ffi\in(\ffi_0,\ffi_{\rm max}]$ and $\Deg \ffi\geq 0$. Proceeding in the same way as above, we obtain the opposite monotonicity for $\Phi$. The rigidity statement is proved in the same way as in the preceding case.
\end{proof}

If the limit of $\Phi(s)$ exists as $s\to\ffi_0$, then this limit is finite since $|\na\ffi|_g$ is bounded (this is a consequence of Proposition~\ref{pro:min_pr_SD}) and the level sets are compact.
Therefore, from the monotonicity of $\Phi$ we can deduce the following global monotonicity property.

\begin{corollary}
	\label{le:in_mon_Phi1_SD}
Let $(M,\go,u)$ be a solution to problem~\eqref{eq:prob_SD}, let $N$ be  a connected component of $M\setminus{\rm MAX}(u)$ with virtual mass $m<\mmax$, and let $g$ and $\ffi$ be defined by~\eqref{eq:g_SD},~\eqref{eq:ffi_SD}. Let $\Sigma_N=\overline{N}\cap\overline{M\setminus\overline{N}}$ be the possibly stratified hypersurface separating $N$ from the rest of the manifold $M$. Then
\begin{equation}
\label{eq:lower_bound_g}
|\Sigma_N|_g\,\leq\,|\pa N|_g \,.
\end{equation}
Moreover, if the equality holds, 
then $(M,\go,u)$ is isometric to a generalized Schwarzschild--de Sitter triple~\eqref{eq:gen_SD} with mass $m$.
\end{corollary}

\begin{proof}
If $N$ is outer, then we know that the function $\Phi(s)$, defined for $0\leq \ffi<\ffi_0$, is monotonically nonincreasing by Proposition~\ref{pro:mon_Phi1_SD}, hence $\lim_{s\to\ffi_0}\Phi(s)\leq\Phi(0)$. Moreover, from Lemma~\ref{le:bound_psi_SD}, we know that $|\na \ffi|_g\leq 1$ on $\pa N$, thus $\Phi(0)\leq\int_{\pa N}\rmd\sigma_g=|\pa N|_g$. This proves the following
\begin{equation*}
|\pa N|_g\,\geq\,\lim_{s\to\ffi_0}\int_{\{\ffi=s\}} |\na\ffi|_g\,\rmd\sigma_g \,.
\end{equation*}
It remains to show that the right hand side is greater than or equal to $|\Sigma_N|_g$. To this end, for a small value $\ep>0$, consider a set $S_\ep\subset\Sigma_N$ such that $|S_\ep|_g<\ep$ and $\Sigma_N\setminus S_\ep$ is contained in the top stratum of $\Sigma_N$.
From Proposition~\ref{pro:C2} and formula~\eqref{eq:na_ffi_SD} we know that $g,\ffi$ are $\mathscr{C}^2$ and
$$
|\na\ffi|_g^2\,=\,\frac{|\D u|^2}{\psi^2\left[1-\big(r_0(m)/\psi\big)^n\right]^2}\,\to\,1
$$
as we approach the top stratum of $\Sigma_N$. 
In particular, for every value $s$ close enough to $\ffi_0$, the flow of $\na\ffi$ gives a diffeomorphism between $\Sigma_N\setminus S_\ep$ and an open subset $V_s\subset\{\ffi=s\}$. 
Therefore, from the continuity of $g$ and $|\na\ffi|_g$, we get
\begin{equation*}
\lim_{s\to\ffi_0}\int_{\{\ffi=s\}} |\na\ffi|_g\,\rmd\sigma_g\,\geq\,\lim_{s\to\ffi_0}\int_{V_s} |\na\ffi|_g\,\rmd\sigma_g\,=\,|\Sigma_N\setminus S_\ep|_g \,.
\end{equation*}
Taking the limit as $\ep\to 0$ we obtain the wished inequality.
The case $N$ inner is proved in the exact same way.
\end{proof}

\noindent
Recalling the definition of $g$, if we rewrite formula~\eqref{eq:lower_bound_g} in terms of $\go$, we obtain
$$
|\pa N|\,\geq\,\left[\frac{r_{\pm}(m)}{r_0(m)}\right]^{n-1}|\Sigma_N|\,,
$$
where the sign $\pm$ depends on whether $N$ is outer or inner. This proves Theorem~\ref{thm:lower_bound}, stated in the introduction, except for the case where $N$ is a cylindrical region, which will be studied in Section~\ref{sec:nariai_SD}.

\section{Integral identities}
\label{sec:int_id}

In this section we use the pseudo-affine function $\ffi$ in order to construct a vector field with a nonnegative divergence. As an application of the Divergence Theorem, we will then deduce a couple of important integral identities. In particular, in Propositions~\ref{pro:int_id_g_SD} and~\ref{pro:int_id_g_2_SD} we show two functions that give a nonnegative value when integrated along any level set of $\ffi$. Moreover, the integral on a level set is zero only in the case of the Schwarzschild--de Sitter solution.

The analysis of the case where $N$ is an outer region and the case where $N$ is an inner region are slightly different. The outer case will be studied in Subsection~\ref{sub:integral_identities_+_SD} and the inner case will be studied in Subsection~\ref{sub:integral_identities_-_SD}.

\subsection{Integral identities in the outer regions.}
\label{sub:integral_identities_+_SD}

We start by considering the case where $N$ is an outer region, that is, in this subsection we will suppose
$$
\max_{S\in\pi_0(\pa N)}\kappa(S)\,=\,\max_{\pa N}\frac{|\D u|}{\umax}\,<\,\sqrt{n}\,,
$$ 
and the pseudo-radial function $\Psi=\psi\circ u$ is chosen as in~\eqref{eq:pr_function_+}.
Consider the vector field 
$$
Y\,=\,\na|\na\ffi|_g^2\,+\,\Deg\ffi\na\ffi\,.
$$
From the Bochner formula~\eqref{eq:bochner_SD} and the equations in~\eqref{eq:pb_conf_SD} we compute
\begin{multline*}
\!\!\!\!{\rm div}_g(Y)\,=\,\Deg|\na\ffi|_g^2+{\rm div}_g(\Deg\ffi\na\ffi)
\\
=\,-\left[(n-2)u+\frac{\psi}{\dot{\psi}}+3n\psi\dot{\psi}\right]\langle\na\ffi\,|Y\,\rangle_g+2|\nana\ffi|_g^2+(\Deg\ffi)^2-2n(n+2)u\psi\dot{\psi}|\na\ffi|_g^2(1-|\na\ffi|_g^2)\,.
\end{multline*} 
Since 
$$
u\psi\dot{\psi}=\psi^2\dot{\psi}^2\left(\frac{u}{\psi\dot{\psi}}\right)=-\psi^2\dot{\psi}^2[1-(n-2)m\psi^{-n}]
$$
is negative when the chosen determination of $\psi$ is~\eqref{eq:pr_function_+}, we have
\begin{multline}
\label{eq:divY_SD}
{\rm div}_g(Y)\,+\,\left[(n-2)u+\frac{\psi}{\dot{\psi}}+3n\psi\dot{\psi}\right]\langle\na\ffi\,|Y\,\rangle_g\,=
\\
=\,2|\nana\ffi|_g^2+(\Deg\ffi)^2-2n(n+2)u\psi\dot{\psi}|\na\ffi|_g^2(1-|\na\ffi|_g^2)
\geq \,0\,.
\end{multline}
Now consider the function 
\begin{equation}
\label{eq:gamma_+_SD}
\gamma=-\frac{u^2\psi^{2n-1}}{\dot{\psi}^{3}}\,=\,\frac{\psi^{2(n+1)}}{u}\,\big[1-(n-2)m\psi^{-n}\big]^{3}
\end{equation}
(notice that $\gamma\geq 0$ when $\Psi=\psi\circ u$ is as in~\eqref{eq:pr_function_+}).
We compute
\begin{align*}
\frac{\gamma'}{\gamma}\,&=\,\frac{\dot{\psi}^{3}}{u^2\psi^{2n-1}}\cdot\frac{d u}{d \ffi}\cdot\left[2\frac{u\psi^{2n-1}}{\dot{\psi}^{3}}+(2n-1)\frac{u^2\psi^{2n-2}\dot{\psi}}{\dot{\psi}^{3}}-3\frac{u^2\psi^{2n-1}\ddot{\psi}}{\dot{\psi}^{4}}\right]
\\
&=\,-\frac{\dot{\psi}^{2}}{u\psi^{2n-2}}\left[2\frac{u\psi^{2n-1}}{\dot{\psi}^{3}}+(2n-1)\frac{u^2\psi^{2n-2}\dot{\psi}}{\dot{\psi}^{3}}-3\frac{u^2\psi^{2n-1}\ddot{\psi}}{\dot{\psi}^{4}}\right]
\\
&=\,-2\frac{\psi}{\dot{\psi}}-(2n-1)u+3\frac{u\psi\ddot{\psi}}{\dot{\psi}^{2}}
\\
&=\,-2\frac{\psi}{\dot{\psi}}-(2n-1)u+3n\psi\dot{\psi}+3(n-1)u+3\frac{\psi}{\dot{\psi}}
\\
&=\,(n-2)u+\frac{\psi}{\dot{\psi}}+3n\psi\dot{\psi}\,.
\end{align*}
Therefore, formula~\eqref{eq:divY_SD} rewrites as
\begin{equation}
\label{eq:divY2_SD}
{\rm div}_g(\gamma Y)\,=\,\gamma\left[2|\nana\ffi|_g^2+(\Deg\ffi)^2-2n(n+2)u\psi\dot{\psi}|\na\ffi|_g^2\left(1-|\na\ffi|_g^2\right)\right]\,\geq\,0\,.
\end{equation}
Integrating~\eqref{eq:divY2_SD} in $N$, we obtain the following proposition.

\begin{proposition}
\label{pro:int_id_g_SD}
Let $(M,\go,u)$ be a solution to problem~\eqref{eq:prob_SD}, let $N\subseteq M\setminus{\rm MAX}(u)$ be an outer region with virtual mass $m=\mu(N,\go,u)$, and let $\Psi$, $g$ and $\ffi$ be defined by~\eqref{eq:pr_function_+},~\eqref{eq:g_SD} and~\eqref{eq:ffi_SD}.
For any $0\leq s<\ffi_0$ it holds
\begin{multline}
\label{eq:int_id_g_SD}
\int_{\pa N}|\na\ffi|_g\left[\Ricg(\nu_g,\nu_g)-\frac{3}{2}n r_+^2(m)(1-|\na\ffi|_g^2)\right]\rmd\sigma_g\,=
\\
=\,-\,\frac{1}{C}\int_{N}\gamma\,\bigg[\,|\nana\ffi|_g^2\,+\,\frac{1}{2}(\Deg\ffi)^2\,-
\,n(n+2)u\psi\dot{\psi}|\na\ffi|_g^2(1-|\na\ffi|_g^2)\bigg]\rmd\sigma_g
\,\leq\,0\,.
\end{multline}
where $C=C(m,n)=r_+^{2n+1}(m)[1-(n-2)mr_+^{-n}(m)]^2$ and $\gamma$ is the function defined by~\eqref{eq:gamma_+_SD}.
Moreover, if the equality
\begin{equation}
\label{eq:int_id_rig_g_SD}
\int_{\pa N}|\na\ffi|_g\left[\Ricg(\nu_g,\nu_g)-\frac{3}{2}n r_+^2(m)(1-|\na\ffi|_g^2)\right]\rmd\sigma_g\,=\,0\,,
\end{equation}
holds, then the solution $(M,\go,u)$ is isometric to a generalized Schwarzschild--de Sitter triple~\eqref{eq:gen_SD} with mass $m$.
\end{proposition}

\begin{proof}
Let us recall from Subsection~\ref{sub:prelim}
 that $u$ is an analytic function. In particular, also $\ffi$ is analytic in the interior of $N$, hence its critical level sets are discrete. It follows that we can choose $0<s<S<\ffi_0$, with $s$ arbitrarily close to zero and $S$ arbitrarily close to $\ffi_0$ such that both $s$ and $S$ are regular values for $\ffi$. 
Integrating ${\rm div}_g(\gamma Y)$ on $\{s\leq\ffi\leq S\}$ and using the Divergence Theorem we obtain
\begin{equation}
\label{eq:int_in_+_aux_SD}
\int\limits_{\{s\leq\ffi\leq S\}}\!\!\!{\rm div}_g(\gamma Y)\rmd\sigma_g
\,=\,
\int\limits_{\{\ffi= S\}}\!\!\!\gamma(S)\Big\langle Y\,\Big|\,\frac{\na\ffi}{|\na\ffi|_g}\Big\rangle_{\!g}\rmd\sigma_g
-\int\limits_{\{\ffi= s\}}\!\!\!\gamma(s)\Big\langle Y\,\Big|\,\frac{\na\ffi}{|\na\ffi|_g}\Big\rangle_{\!g}\rmd\sigma_g\,,
\end{equation}
First of all, we notice that it holds
\begin{equation}
\label{eq:limit_int_SD}
\lim_{S\to\ffi_0}\gamma(S)\int_{\{\ffi= S\}}\Big\langle Y\,\Big|\,\frac{\na\ffi}{|\na\ffi|_g}\Big\rangle_{\!g}\rmd\sigma_g=0\,.
\end{equation}
In fact, using formul\ae~\eqref{eq:na_ffi_SD},~\eqref{eq:nana_ffi_SD} and~\eqref{eq:pb_conf_SD} to translate the integrand in terms of $u,\go$, we find
\begin{align*}
\gamma\,\Big\langle Y\,\Big|\,\frac{\na\ffi}{|\na\ffi|_g}\Big\rangle_{\!g}\,&=\gamma\,\left(\frac{\langle\na|\na\ffi|_g^2\,|\,\na\ffi\rangle_g}{|\na\ffi|_g}\,+\,\Deg\ffi\,|\na\ffi|_g\right)
\\
&=\,\gamma\,|\na\ffi|_g\left[2\,\nana\ffi(\nu_g,\nu_g)+\Deg\ffi\right]
\\
&=\,\frac{\psi^{2n}}{\dot\psi}|\D u|\bigg[-2\,\DD u(\nu,\nu)-2\frac{\dot\psi}{u\psi}\bigg(n-1+n\frac{\psi\dot\psi}{u}\bigg)|\D u|^2 \,+\,n\,\bigg(1-\frac{\dot\psi^2}{u^2}\,|\D u|^2\bigg)\,\bigg]\,,
\end{align*}
where $\nu=\D u/|\D u|,\nu_g=\na\ffi/|\na\ffi|_g=\psi\,\nu$ are the unit normals to $\{\ffi=S\}$, which exist everywhere because $\{\ffi=S\}$ is a regular level set.
Since $|\na\ffi|_g^2=(\dot\psi^2/u^2)|\D u|^2\leq 1$ by Proposition~\ref{pro:min_pr_SD}, we deduce that the limit of the term in square bracket as $S\to\ffi_0$ (or equivalently $u\to \umax$) is bounded from above. Therefore, in order to prove~\eqref{eq:limit_int_SD}, it is enough to show that
$$
\lim_{t\to 1^-}\int_{\{u=t\}\cap N}\frac{1}{\dot\psi}\,|\D u|\,\rmd\sigma\,=\,0\,.
$$
But this can be done proceeding exactly as in the proof of~\cite[Theorem~4.4]{Bor_Maz_2-I}, via a simple argument using the coarea formula and the facts that $(\dot\psi^2/u^2)|\D u|^2\leq 1$ and $\dot\psi\to +\infty$ as $u\to \umax$.
Therefore, taking the limit as $S\to\ffi_0$ of~\eqref{eq:int_in_+_aux_SD}, we deduce
\begin{equation}
\label{eq:int_in_aux2_+_SD}
\int_{\{\ffi= s\}}\gamma(s)\Big\langle Y\,\Big|\,\frac{\na\ffi}{|\na\ffi|_g}\Big\rangle_{\!g}\rmd\sigma_g
\,=\,
-\int_{\{s\leq\ffi<\ffi_0\}}{\rm div}_g(\gamma Y)\rmd\sigma_g\,\leq\,0\,,
\end{equation}
where in the last inequality we have used~\eqref{eq:divY2_SD}.
Now we compute the integral on the left hand side. Using the equations in~\eqref{eq:pb_conf_SD}, we obtain
\begin{align*}
\frac{1}{\dot\psi}\Big\langle Y\,\Big|\,\frac{\na\ffi}{|\na\ffi|_g}\Big\rangle_{\!g}\,&=\,2\,\frac{\nana\ffi(\na\ffi,\na\ffi)}{|\na\ffi|_g}+\Deg\ffi|\na\ffi|_g
\\
&=\,|\na\ffi|_g\,\bigg[-\frac{2}{(n-2)u+({\psi}/{\dot\psi})}\,\Ricg(\nu_g,\nu_g)\,+\,\bigg(1-2\,\frac{u-({\psi}/{\dot\psi})}{(n-2)u+({\psi}/{\dot\psi})}\bigg)\Deg\ffi\bigg]\,, 
\end{align*}
and taking the limit as $s\to 0$, since $u\to 0$, $\psi\to r_+(m)$ and $\dot\psi\to 0$, we get
\begin{equation}
\label{eq:integrand_aux_+_SD}
\!\!\!\!\lim_{s\to 0}\left[\frac{1}{\dot\psi}\,\Big\langle Y\,\Big|\,\frac{\na\ffi}{|\na\ffi|_g}\Big\rangle_{\!g}\right]_{|_{\{\ffi=s\}}}\!\!\!\!\!\!\!\!
=\,\left\{|\na\ffi|_g
\left[-\frac{2}{r_+(m)}\,\Ricg(\nu_g,\nu_g)+3n r_+(m)\left(1-|\na\ffi|_g^2\right)\right]\right\}_{|_{\pa N}}\!\!\!.
\end{equation}
Moreover, recalling from~\eqref{eq:dpsideu_SD} the relation between $\psi$, $\dot\psi$ and $u$, we find
\begin{align}
\notag
\lim_{s\to 0}(\dot\psi\gamma)_{|_{\{\ffi=s\}}}\,&=\,-\lim_{s\to 0}\left(\frac{u^2}{\dot\psi^2}\psi^{2n-1}\right)_{|_{\{\ffi=s\}}}
\\
\notag
\,&=\,-\lim_{s\to 0}\left\{\psi^{2n+1}[1-(n-2)m\psi^{-n}]^2\right\}_{|_{\{\ffi=s\}}}
\\
\label{eq:limit_aux_+_SD}
\,&=\,-\,r_+^{2n+1}(m)[1-(n-2)mr_+^{-n}(m)]^2\,.
\end{align}
Taking the limit of~\eqref{eq:int_in_aux2_+_SD} as $s\to 0$ and using the information given by~\eqref{eq:integrand_aux_+_SD} and~\eqref{eq:limit_aux_+_SD}, we obtain te desired inequality~\eqref{eq:int_id_g_SD}.

To prove the rigidity statement, we start by observing that, if the equality~\eqref{eq:int_id_rig_g_SD} holds, then necessarily the right-hand side of~\eqref{eq:int_id_g_SD} is null. In particular, $|\na\ffi|_g\equiv 1$ on $N$.
Substituting this information in the Bochner formula~\eqref{eq:bochner_SD} we obtain $|\nana\ffi|_g\equiv 0$, hence we can apply Proposition~\ref{pro:rigidity} to conclude.
\end{proof}

\subsection{Integral identities in the inner regions.}
\label{sub:integral_identities_-_SD}
In this subsection, we deal with the case in which $N$ is an inner region, that is, 
$$
\max_{S\in\pi_0(\pa N)}\kappa(S)\,=\,\max_{\pa N}\frac{|\D u|}{\umax}\,>\,\sqrt{n}\,,
$$ 
and the pseudo-radial function $\Psi=\psi\circ u$ is defined by~\eqref{eq:pr_function_-}. This case is slightly more complicated than the outer one, and requires a generalization of the computations of the previous subsection. Let 
$$
Y_\a=\na|\na\ffi|_g^2+\a\Deg\ffi\na\ffi\,,
$$
where $\a\in\R$.
From the Bochner formula~\eqref{eq:bochner_SD} and the equations in~\eqref{eq:pb_conf_SD} we compute
\begin{multline*}
{\rm div}_g(Y_\a)+\left[(n-2)u+\frac{\psi}{\dot{\psi}}+3n\psi\dot{\psi}\right]\langle\na\ffi\,|Y_\a\,\rangle_g\,=
\\
=\,2|\nana\ffi|_g^2+\a(\Deg\ffi)^2+n\psi^2\dot\psi^2\left[n(\a-1)(\a+2)-2(n+\a+1)\frac{u}{\psi\dot{\psi}} \right]|\na\ffi|_g^2(1-|\na\ffi|_g^2)\,.
\end{multline*} 
In order for the term $2|\nana\ffi|_g^2+\a(\Deg\ffi)^2$ to be positive, we want $\a\geq -2/n$. Recalling 
$$
\frac{u}{\psi\dot{\psi}}\,=\,-[1-(n-2)m\psi^{-n}]\,,
$$
we have
\begin{multline}
\label{eq:divg_Ya_SD}
{\rm div}_g(Y_\a)\,+\,\left[(n-2)u+\frac{\psi}{\dot{\psi}}+(\a+2)n\psi\dot{\psi}\right]\langle\na\ffi\,|Y_\a\,\rangle_g\,=
\\
=\,2|\nana\ffi|_g^2+\a(\Deg\ffi)^2+n\psi^2\dot\psi^2\left[n(n\a+2)(\a+1)-2(n+\a+1)(n-2)\frac{m}{\psi^{n}} \right]|\na\ffi|_g^2(1-|\na\ffi|_g^2)\,.
\end{multline}
The term in square brackets is positive if and only if
\begin{equation}
\label{eq:alpha_aux_SD}
\frac{\psi^n}{m}\geq 2(n-2)\frac{n+\a+1}{(n\a+2)(\a+1)}\,.
\end{equation}
Since the term on the right hand side goes to zero as $\a\to\infty$, there exists an $\a$ big enough so that
\begin{equation}
\label{eq:alpha_SD}
\frac{r_-^n(m)}{m}= 2(n-2)\frac{n+\a+1}{(n\a+2)(\a+1)}\,.
\end{equation}
Notice that the value of $\a$ that satisfies~\eqref{eq:alpha_SD} is greater than or equal to $1$ (in fact, if we set $\a=1$ in~\eqref{eq:alpha_aux_SD} we have $\psi^n\geq (n-2)m$, which is never satisfied on $N$). If we choose $\alpha$ as in~\eqref{eq:alpha_SD}, we have that the square bracket above is positive for any $\psi\in[r_-(m),\big((n-2)m\big)^{1/n}]$. In particular, for that $\a$ we have
\begin{equation}
\label{eq:divY_2_SD}
{\rm div}_g(Y_\a)\,+\,\left[(n-2)u+\frac{\psi}{\dot{\psi}}+(\a+2)n\psi\dot{\psi}\right]\langle\na\ffi\,|Y_\a\,\rangle_g\,\geq\, 0\,.
\end{equation}
on the whole $N$.
Now we choose 
\begin{equation}
\label{eq:gamma_-_SD}
\gamma\,=\,\frac{u^{\a+1}\psi^{n\a+n-\a}}{\dot{\psi}^{\a+2}}\,=\,\frac{\psi^{n\a+n+2}}{u}\,\big[1-(n-2)m\psi^{-n}\big]^{\a+2}\geq 0\,,
\end{equation} 
(notice that $\gamma\geq 0$ when $\Psi=\psi\circ u$ is as in~\eqref{eq:pr_function_-}). We compute
\begin{align*}
\frac{\gamma'}{\gamma}\,&=\,\frac{\dot{\psi}^{\a+2}}{u^{\a+1}\psi^{n\a+n-\a}}\cdot\frac{d u}{d \ffi}\cdot
\bigg[(\a+1)\frac{u^\a\psi^{n\a+n-\a}}{\dot{\psi}^{\a+2}}+(n\a+n-\a)\frac{u^{\a+1}\psi^{n\a+n-\a-1}\dot{\psi}}{\dot{\psi}^{\a+2}}
\\
&\qquad\qquad\qquad\qquad\qquad\qquad\qquad\qquad\qquad\qquad\qquad\qquad\qquad\qquad\qquad\ \ -(\a+2)\frac{u^2\psi^{n\a+n-\a}\ddot{\psi}}{\dot{\psi}^{\a+3}}\bigg]
\\
&=\,-\frac{\dot{\psi}^{\a+1}}{u^\a\psi^{n\a+n-\a-1}}
\bigg[(\a+1)\frac{u^\a\psi^{n\a+n-\a}}{\dot{\psi}^{\a+2}}+(n\a+n-\a)\frac{u^{\a+1}\psi^{n\a+n-\a-1}\dot{\psi}}{\dot{\psi}^{\a+2}}
\\
&\qquad\qquad\qquad\qquad\qquad\qquad\qquad\qquad\qquad\qquad\qquad\qquad\qquad\qquad\qquad\ \ -(\a+2)\frac{u^2\psi^{n\a+n-\a}\ddot{\psi}}{\dot{\psi}^{\a+3}}\bigg]
\\
&=\,-(\a+1)\frac{\psi}{\dot{\psi}}-(n\a+n-\a)u+(\a+2)\frac{u\psi\ddot{\psi}}{\dot{\psi}^{2}}
\\
&=\,-(\a+1)\frac{\psi}{\dot{\psi}}-(n\a+n-\a)u+(\a+2)n\psi\dot{\psi}+(\a+2)(n-1)u+(\a+2)\frac{\psi}{\dot{\psi}}
\\
&=\,(n-2)u+\frac{\psi}{\dot{\psi}}+(\a+2)n\psi\dot{\psi}\,.
\end{align*}
From formul\ae~\eqref{eq:divg_Ya_SD},~\eqref{eq:divY_2_SD} we deduce
\begin{multline}
\label{eq:divY2_2_SD}
{\rm div}_g(\gamma Y_\a)\,
=\,2|\nana\ffi|_g^2+\a(\Deg\ffi)^2+
\\
+n\psi^2\dot\psi^2\left[n(n\a+2)(\a+1)-2(n+\a+1)(n-2)\frac{m}{\psi^{n}} \right]|\na\ffi|_g^2(1-|\na\ffi|_g^2)
\,
\geq \,0\,.
\end{multline}
Integrating~\eqref{eq:divY2_2_SD} on $N$, we obtain the following statement.

\begin{proposition}
\label{pro:int_id_g_2_SD}
Let $(M,\go,u)$ be a solution to problem~\eqref{eq:prob_SD}, let $N\subseteq M\setminus{\rm MAX}(u)$ be an inner region with virtual mass $m=\mu(N,\go,u)$, and let $\Psi$, $g$ and $\ffi$ be defined by~\eqref{eq:pr_function_+},~\eqref{eq:g_SD} and~\eqref{eq:ffi_SD}.
For any $0\leq s<\ffi_0$ it holds
\begin{multline}
\label{eq:int_id_g_2_SD}
	\int_{\pa N}|\na\ffi|_g\left[\Ricg(\nu_g,\nu_g)-\frac{\a+2}{2}n r_-^2(m)(1-|\na\ffi|_g^2)\right]\rmd\sigma_g\,=\,-\,\frac{1}{C}\int_{N}\gamma\,\bigg[|\nana\ffi|_g^2+\frac{\a}{2}(\Deg\ffi)^2+
\\
+\psi\dot\psi\left(\frac{1}{2}n(n\a+2)(\a+1)-(n+\a+1)(n-2)\frac{m}{\psi^{n}} \right)|\na\ffi|_g^2\Deg\ffi\bigg]\rmd \sigma_g
\,\leq\,0\,,
\end{multline}
where $\a\geq 1$ is the solution of equation~\eqref{eq:alpha_SD}, $\gamma$ is the function defined by~\eqref{eq:gamma_+_SD} and $C=C(\a,m,n)=r_-^{(\a+1)n+1}(m)[1-(n-2)mr_-^{-n}(m)]^{\a+1}$.
Moreover, if the equality
\begin{equation}
\label{eq:int_id_rig_g2_SD}
\int_{\pa N}|\na\ffi|_g\left[\Ricg(\nu_g,\nu_g)-\frac{\a+2}{2}n r_-^2(m)(1-|\na\ffi|_g^2)\right]\rmd\sigma_g\,=\,0\,,
\end{equation}
holds, then the solution $(M,\go,u)$ is isometric to a generalized Schwarzschild--de Sitter triple~\eqref{eq:gen_SD} with mass $m$.
\end{proposition}

\begin{proof}
Let us recall from Subsection~\ref{sub:prelim}
 that $u$ is an analytic function. In particular, also $\ffi$ is analytic in the interior of $N$, hence its critical level sets are discrete. It follows that we can choose $\ffi_0<S<s<\ffi_{\rm max}$, with $S$ arbitrarily close to $\ffi_0$ and $s$ arbitrarily close to $\ffi_{\rm max}$ such that both $s$ and $S$ are regular values for $\ffi$. 
Integrating ${\rm div}_g(\gamma Y_\a)$ on $\{S\leq\ffi\leq s\}$ and using the Divergence Theorem we obtain
\begin{equation}
\label{eq:int_in_-_aux_SD}
\int\limits_{\{S\leq\ffi\leq s\}}\!\!\!{\rm div}_g(\gamma Y_\a)\rmd\sigma_g
\,=\,
\int\limits_{\{\ffi= s\}}\!\!\!\gamma(s)\Big\langle Y_\a\,\Big|\,\frac{\na\ffi}{|\na\ffi|_g}\Big\rangle_{\!g}\rmd\sigma_g
-\int\limits_{\{\ffi= S\}}\!\!\!\gamma(S)\Big\langle Y_\a\,\Big|\,\frac{\na\ffi}{|\na\ffi|_g}\Big\rangle_{\!g}\rmd\sigma_g\,,
\end{equation}
First of all, with analogous computations to the ones employed in the proof of Proposition~\ref{pro:int_id_g_SD}, we obtain
$$
\lim_{S\to\ffi_0}\gamma(S)\int_{\{\ffi= S\}}\Big\langle Y_\a\,\Big|\,\frac{\na\ffi}{|\na\ffi|_g}\Big\rangle_{\!g}\rmd\sigma_g=0\,.
$$
Therefore, taking the limit as $S\to\ffi_0$ of~\eqref{eq:int_in_-_aux_SD}, we deduce
\begin{equation}
\label{eq:int_in_aux2_-_SD}
\int_{\{\ffi= s\}}\gamma(s)\Big\langle Y_\a\,\Big|\,\frac{\na\ffi}{|\na\ffi|_g}\Big\rangle_{\!g}\rmd\sigma_g
\,=\,
\int_{\{\ffi_0<\ffi\leq s\}}{\rm div}_g(\gamma Y_\a)\rmd\sigma_g\,\geq\,0\,,
\end{equation}
where in the last inequality we have used~\eqref{eq:divY2_2_SD}.
Now we compute the integral on the left hand side. Using the equations in~\eqref{eq:pb_conf_SD}, we obtain
\begin{align*}
\frac{1}{\dot\psi}\Big\langle Y_\a\,\Big|\,\frac{\na\ffi}{|\na\ffi|_g}\Big\rangle_{\!g}\,&=\,2\,\frac{\nana\ffi(\na\ffi,\na\ffi)}{|\na\ffi|_g}+\Deg\ffi|\na\ffi|_g
\\
&=\,|\na\ffi|_g\,\bigg[-\frac{2}{(n-2)u+({\psi}/{\dot\psi})}\,\Ricg(\nu_g,\nu_g)\,+\,\bigg(\a-2\,\frac{u-({\psi}/{\dot\psi})}{(n-2)u+({\psi}/{\dot\psi})}\bigg)\Deg\ffi\bigg]\,, 
\end{align*}
and taking the limit as $s\to \ffi_{\rm max}$, since $u\to 0$, $\psi\to r_-(m)$ and $\dot\psi\to 0$, we get
\begin{multline}
\label{eq:integrand_aux_-_SD}
\lim_{s\to \ffi_{\rm max}}\left[\frac{1}{\dot\psi}\,\Big\langle Y_\a\,\Big|\,\frac{\na\ffi}{|\na\ffi|_g}\Big\rangle_{\!g}\right]_{|_{\{\ffi=s\}}}\!\!\!\!\!\!
=
\\
=\,\left\{|\na\ffi|_g
\left[-\frac{2}{r_-(m)}\,\Ricg(\nu_g,\nu_g)+(\a+2)n r_-(m)\left(1-|\na\ffi|_g^2\right)\right]\right\}_{|_{\pa N}}\!\!\!\!.
\end{multline}
Moreover, recalling from~\eqref{eq:dpsideu_SD} the relation between $\psi$, $\dot\psi$ and $u$, we find
\begin{align}
\notag
\lim_{s\to \ffi_{\rm max}}(\dot\psi\gamma)_{|_{\{\ffi=s\}}}\,&=\,\lim_{s\to \ffi_{\rm max}}\left(\frac{u^{\a+1}}{\dot\psi^{\a+1}\psi^{\a+1}}\psi^{(\a+1)n+1}\right)_{|_{\{\ffi=s\}}}
\\
\notag
\,&=\,\lim_{s\to \ffi_{\rm max}}\left\{\psi^{(\a+1)n+1}[1-(n-2)m\psi^{-n}]^{\a+1}\right\}_{|_{\{\ffi=s\}}}
\\
\label{eq:limit_aux_-_SD}
\,&=\,r_-^{(\a+1)n+1}(m)[1-(n-2)mr_-^{-n}(m)]^{\a+1}\,.
\end{align}
Taking the limit of~\eqref{eq:int_in_aux2_-_SD} as $s\to \ffi_{\rm max}$ and using~\eqref{eq:integrand_aux_-_SD} and~\eqref{eq:limit_aux_-_SD}, we obtain the desired inequality~\eqref{eq:int_id_g_2_SD}.

The rigidity statement is proved in the same way as in Proposition~\ref{pro:int_id_g_SD}. If the equality in~\eqref{eq:int_id_rig_g2_SD} holds, then necessarily the right hand side of~\eqref{eq:int_id_g_2_SD} is null. In particular, $|\na\ffi|_g\equiv 1$ on $N$.
Substituting this information in the Bochner formula~\eqref{eq:bochner_SD} we obtain $|\nana\ffi|_g\equiv 0$, hence we can apply Proposition~\ref{pro:rigidity} to conclude.
\end{proof}

\section{Area bounds}
\label{sec:area_bounds}

Subsection~\ref{sub:pointwise_bounds} is devoted to the proof of the inequalities in Theorems~\ref{thm:area_bound_SD} and~\ref{thm:scalarcurvature_bound_SD} for outer and inner regions. The proof of the corresponding rigidity statements will be discussed in Section~\ref{sec:consequences_SD}, whereas the cylindrical case will be addressed in Section~\ref{sec:nariai_SD}.
In Subsection~\ref{sub:area_bounds_Sigma} we will discuss some area bounds for the hypersurface separating our region from the rest of the manifold. In particular, we will recover Corollary~\ref{cor:lower_bound}.

\subsection{Area bounds for the horizons}
\label{sub:pointwise_bounds}

Let $N$ be a connected component of $M\setminus{\rm MAX}(u)$, let $\mu=\mu(N,\go,u)$ be its virtual mass and let $S\subseteq\pa N$ be an horizon with maximum surface gravity. We now follow~\cite{Chr_Sim} (see also~\cite[Section~4]{Lee_Nev}) to prove area bounds for the horizon $S$.
We start by noticing that, if we define $g,\Psi,\ffi$ as usual with respect to the mass $m$, the definitions are chosen in such a way that
$$
|\na\ffi|_g^2(p)= \left|\frac{\D u}{\psi\big[1-(n-2)m\psi^{-n}\big]}\right|^2(p)\,=\,\frac{|\D u|^2(p)}{W(u(p))}\,,
$$
where $W(t)$ is the constant value of $|\D u|^2$ on the level set $\{u=t\}\cap M_\pm$ of the Schwarzschild--de Sitter solution, where of course the sign $\pm$ depends on whether we are on an outer or inner region.

From Lemma~\ref{le:bound_psi_SD} we have $|\na\ffi|_g=1$ on $S$, whereas Proposition~\ref{pro:min_pr_SD} tells us that $|\na\ffi|_g\leq 1$ on the whole $N$. In other words, we have $|\D u|^2\leq W(u)$ on the whole $N$ and the equality holds on $S$. Let now $p\in S$ and $\gamma:[0,\ep)\to\R$ be a geodesic such that $\gamma(0)=p$ and $\gamma'(0)=\nu$, where $\nu$ is the unit normal to $S$ pointing inside $N$. Applying~\cite[(11)]{Ambrozio} in $N$ we have
\begin{align}
\notag
|\D u|^2\circ \gamma(s)\,&=\,W(0)\left[1+(\Ric(\nu,\nu)-n)s^2+\mathcal{O}(s^4)\right]
\\
\label{eq:Ambrozio_nearboundary}
&=\,W(0)\left[1+\left(\frac{n(n-3)}{2}-\frac{\RRR^S}{2}\right)s^2+\mathcal{O}(s^4)\right]\,,
\end{align}
where in the second identity we have used the Gauss-Codazzi equation and the fact that the horizon $S$ is totally geodesic.
Of course we can apply the same formula on the model solution, obtaining the same expansion with $\RRR^{\pa M_\pm}$ in place of $\RRR^S$.
Since $|\D u|^2\leq W$ as observed above, necessarily we have $\RRR^S\geq\RRR^{\pa M_\pm}$. In other words, we have proven the following:
\begin{theorem}
	\label{thm:intid_rewr_SD}
	Let $(M,\go,u)$ be a solution to problem~\eqref{eq:prob_SD}, let $N\subseteq M\setminus{\rm MAX}(u)$ be a region with virtual mass $m=\mu(N,\go,u)$. Let also $S\subset \pa N$ be an horizon with maximum surface gravity. 
Then 
\begin{itemize}
\item If $N$ is outer, it holds
\begin{equation*}
\RRR^{S}\,\geq\,(n-1)(n-2)r_+^{-2}(m)\,.	
\end{equation*}

\item If $N$ is inner, it holds
\begin{equation*}
\RRR^{S}\,\geq\,(n-1)(n-2)r_-^{-2}(m)\,.	
\end{equation*}
\end{itemize}
\end{theorem}

\noindent
Integrating the inequalities in Theorem~\ref{thm:intid_rewr_SD} on $S$, we obtain the following formula
	\begin{equation}
	\label{eq:Gauss-Codazzi_SD}
	\int_{S}
	\frac{\RRR^{S}}{(n-1)(n-2)}\,\rmd\sigma\,\geq\, r_\pm^{-2}(m)|S|\,.
	\end{equation}
This gives a particularly nice result in the case $n=3$. 
	
\begin{theorem}
	\label{thm:Gauss-Bonnet_SD}
	Let $(M,\go,u)$ be a $3$-dimensional solution to problem~\eqref{eq:prob_SD}, and let $N\subseteq M\setminus{\rm MAX}(u)$ be a region with virtual mass $m=\mu(N,\go,u)$. Let also $S\subset\pa N$ be an horizon with maximum surface gravity in $N$.
Then $S$ is diffeomorphic to $\Sph^2$ and it holds
	\begin{equation*}
	|S|\,\leq\,4\pi r_\pm^{2}(m)\,,
	\end{equation*}
where the sign $\pm$ depends on whether $N$ is an outer or inner region.
\end{theorem}

\begin{proof}
	Substituting $n=3$ in inequality~\eqref{eq:Gauss-Codazzi_SD} and using the Gauss-Bonnet formula, we immediately obtain
	\begin{equation*}
	4\pi\chi(S)\,\geq\, 2r_\pm^{-2}(m)
	\,|S|\,.
	\end{equation*}
	In particular, $\chi(S)$ has to be positive, hence $S$ is necessarily a sphere and we obtain the thesis.
\end{proof}

\subsection{Area bounds for the disconnecting hypersurface}
\label{sub:area_bounds_Sigma}

Combining the results of this section with Corollary~\ref{le:in_mon_Phi1_SD}, it is straightforward to obtain an area bound on the hypersurface $\Sigma_N$ that separates $N$ from the rest of the manifold.

\begin{theorem}
\label{thm:area_bound_dischyp_gen}
Let $(M,\go,u)$ be a solution to problem~\eqref{eq:prob_SD}, let $N\subseteq M\setminus{\rm MAX}(u)$ be a region of $M$ with connected boundary $\pa N$ and with virtual mass $m<\mmax$. Let $\Sigma_N=\overline{N}\cap\overline{M\setminus\overline{N}}$ be the hypersurface separating $N$ from the rest of the manifold $M$. 
\begin{itemize}
\item If $N$ is an outer region, then
\begin{equation}
|\Sigma_N|\,\leq\,\left(\int_{\pa N}
\frac{\RRR^{\pa N}}{(n-1)(n-2)}\,\rmd\sigma\right)\, \frac{r_0^{n-1}(m)}{r_+^{n-3}(m)}\,,
\end{equation}
and, if the equality holds, 
then $(M,\go,u)$ is isometric to a generalized Schwarzschild--de Sitter triple~\eqref{eq:gen_SD} with mass $m$.
\smallskip
\item If $N$ is an inner region, then
\begin{equation}
|\Sigma_N|\,\leq\,\left(\int_{\pa N}
\frac{\RRR^{\pa N}}{(n-1)(n-2)}\,\rmd\sigma\right)\, \frac{r_0^{n-1}(m)}{r_-^{n-3}(m)}\,,
\end{equation}
and, if the equality holds, 
then $(M,\go,u)$ is isometric to a generalized Schwarzschild--de Sitter triple~\eqref{eq:gen_SD} with mass $m$.
\end{itemize}
\end{theorem}

\begin{proof}
Let us study the case where $N$ is outer, the inner case being completely analogous.
From Corollary~\ref{le:in_mon_Phi1_SD}, recalling the definitions of $g,\ffi$, we get
$$
r_0^{1-n}(m)|\Sigma_N|\,=\,|\Sigma_N|_g\,\leq\, |\pa N|_g\,=\,r_+^{1-n}(m)|\pa N|\,.
$$
Now we conclude using formula~\eqref{eq:Gauss-Codazzi_SD}.
\end{proof}

This result becomes particularly nice in dimension $n=3$. Combining Theorem~\ref{thm:Gauss-Bonnet_SD} with Corollary~\ref{le:in_mon_Phi1_SD} we immediately obtain Corollary~\ref{cor:lower_bound}, which we recall here for the reader's convenience.

\begin{corollary}
\label{cor:Sigma_areabound}
Let $(M,\go,u)$ be a $3$-dimensional solution to problem~\eqref{eq:prob_SD}, let $N\subseteq M\setminus{\rm MAX}(u)$ be a region with connected boundary $\pa N$ and with virtual mass $m<\mmax$. Let $\Sigma_N=\overline{N}\cap\overline{M\setminus\overline{N}}$ be the hypersurface separating $N$ from the rest of the manifold $M$. 
Then
\begin{equation}
|\Sigma_N|\,\leq\,4\,\pi \,r_0^2(m)\,.
\end{equation}
Moreover, if the equality holds, 
then $(M,\go,u)$ is isometric to the Schwarzschild--de Sitter triple~\eqref{eq:SD} with mass $m$.
\end{corollary}

\section{Balancing inequalities and rigidity of area bounds}
\label{sec:consequences_SD}

Here we translate the integral identities obtained in Section~\ref{sec:int_id} in terms of $u$ and $\go$. Some computations will lead to the proof of the rigidity statements in Theorem~\ref{thm:scalarcurvature_bound_SD} in the case where $N$ is outer (Theorem~\ref{cor:Gauss-Codazzi_+_SD}) and inner (Theorem~\ref{cor:Gauss-Codazzi_-_SD}). As a consequence of the Gauss-Bonnet formula we will then deduce Theorem~\ref{thm:area_bound_SD} (see Theorems~\ref{thm:3dareabounds_+_SD} and~\ref{thm:3dareabounds_-_SD}). We will also prove some more general statements, in the cases where $N$ has more than one horizon. 

\subsection{Area bounds for outer regions.}
\label{sub:cons_integral_identities_+_SD}

Here we focus on the case where our region $N$ is outer and translate Proposition~\ref{pro:int_id_g_SD}, proved in Subsection~\ref{sub:integral_identities_+_SD}, in terms of $u, g_0$.
To do that, it is useful to let $A :  \pa N \rightarrow (0,1]$ be the locally constant function defined for every $x \in \pa N$ by 
\begin{equation}
\label{eq:step_function_SD}
A(x) \, = \, \frac{|\D u|}{\max_{\pa N}| \D u|}(x) \, . 
\end{equation}

\begin{theorem}
	\label{cor:intid_rewr_+_SD}
	Let $(M,\go,u)$ be a solution to problem~\eqref{eq:prob_SD} and let $N\subseteq M\setminus{\rm MAX}(u)$ be an outer region with virtual mass $m=\mu(N,\go,u)$.
Then it holds
\begin{equation*}
\int_{\pa N}\!\! A^3 
\,\rmd\sigma\,\leq\, 
\left(\int_{\pa N} 
\!\!A\, \frac{\RRR^{\pa N}}{(n-1)(n-2)} \,\rmd\sigma \right)  \, r_+^{2}(m)\,
\, - \, \frac{n(n-4)}{(n-1)(n-2)} \left(\int_{\pa N} \!\!A\,\left( 1- A^2 \right)  
\,\rmd\sigma \right)  \, r_+^{2}(m)\,.
\end{equation*}
where $\RRR^{\pa N}$ is the scalar curvature of the metric induced by $\go$ on $\pa N$ and $A$ is the step function defined in~\eqref{eq:step_function_SD}.
Moreover, if the equality holds, then the solution $(M,\go,u)$ is isometric to a generalized Schwarzschild--de Sitter triple~\eqref{eq:gen_SD} with mass $m$.
\end{theorem}

\begin{proof}
It is enough to translate formula~\eqref{eq:int_id_g_SD} in terms of $u$ and $\go$, using the relations developed in Subsection~\ref{sub:conformal reformulation_SD}. In particular, let us notice that
$$
|\na\ffi|_g^2\,=\,\frac{\dot\psi^2}{u^2}\,|\D u|^2\,,\qquad\hbox{and}\qquad \max_{\pa N}|\na \ffi|_g\,=\,1\,,
$$
where the second identity follows from Lemma~\ref{le:bound_psi_SD}. Therefore
$$
\Big(\frac{\dot{\psi}^2}{u^2}\Big)_{|_{\pa N}} \max_{\pa N}|\D u|^2\,=\,1\,,
$$
which in turn implies $|\na\ffi|^2_g=|\D u|^2/\max_{\pa N}|\D u|^2$.
Now we translate $\Ricg(\nu_g,\nu_g)$ in terms of $\Ric(\nu,\nu)$, where $\nu=\D u/|\D u|$ and $\nu_g=\na\ffi/|\na\ffi|_g=\psi\nu$ are the unit normals to the level sets of $u$ with respect to $\go$ and $g$. Recalling the equations in systems~\eqref{eq:prob_SD},~\eqref{eq:pb_conf_SD}, using also formula~\eqref{eq:nana_ffi_SD} and the fact that $\dot\psi\to 0$ as $u\to 0$, we obtain that on $\pa N=\{u=0\}\cap N$ it holds
\begin{align*}
\Ricg(\nu_g,\nu_g)\,&=\,-\left[(n-2)u+\frac{\psi}{\dot\psi}\right]\nana\ffi(\nu_g,\nu_g)\,-\,\left(u-\frac{\psi}{\dot\psi}\right)\Deg\ffi
\\
&=\,\psi^2\frac{\DD u(\nu,\nu)}{u}\,+\,\left[(n-1)u\frac{\psi}{\dot\psi}+n\psi^2\right]\frac{\dot\psi^2}{u^2}|\D u|^2\,+\,n\psi^2\left(1-\frac{\dot\psi^2}{u^2}|\D u|^2\right)
\\
&=\,\psi^2\left[\Ric(\nu,\nu)-n+n\frac{\dot\psi^2}{u^2}|\D u|^2+(n-1)\frac{u}{\psi\dot\psi}\frac{\dot\psi^2}{u^2}|\D u|^2\,+\,n\left(1-\frac{\dot\psi^2}{u^2}|\D u|^2\right)\right]
\\
&=\,\psi^2\left[\Ric(\nu,\nu)+(n-1)\,\frac{u}{\psi\dot\psi}\,\frac{|\D u|^2}{\max_{\pa N}|\D u|^2}\right]
\end{align*}
Substituting the above computations inside formula~\eqref{eq:int_id_g_SD}, and recalling that $\psi=r_+(m)$ on $\pa N$ (because $N$ is an outer region and $\psi$ is defined as specified in~\eqref{eq:pr_function_+}), we obtain
$$
\int\limits_{\pa N}|\D u|\left[\Ric(\nu,\nu)
+(n-1)\frac{u}{\psi\dot{\psi}}
-
\left(\frac{3}{2}n+(n-1)\frac{u}{\psi\dot{\psi}}\right)
\left(1-\frac{|\D u|^2}{\max_{\pa N}|\D u|^2}\right)\right]\,\rmd\sigma\,\geq\, 0\,,
$$
where we recall that the equality holds if and only if the solution is isometric to the Schwarzschild--de Sitter solution. Notice that the above formula is slightly imprecise, as, rigorously, the quantity $u/(\psi\dot\psi)$ is not defined on $\pa N$, because $\dot\psi\to 0$ as $u\to 0$. However, from formula~\eqref{eq:dpsideu_SD} that quantity can be explicitated as
$$
\frac{u}{\psi\dot\psi}\,=\,-\left[1-(n-2)m\psi^{-n}\right]\,,
$$
which has a finite value on the boundary, as $\psi=r_+(m)$ on $\pa N$.
Moreover, using the Gauss-Codazzi equation we have $2\Ric(\nu,\nu)=\RRR-\RRR^{\pa N}=n(n-1)-\RRR^{\pa N}$, and substituting in the inequality above we get
\begin{multline*}
\int_{\pa N}|\D u|\bigg[\RRR^{\pa N}-(n-1)(n-2)\big(1+2mr_+^{-n}(m)\big)\bigg]\,\rmd\sigma\,\geq
\\
\geq\,-\int_{\pa N}|\D u|\bigg[\Big(n+2+2(n-1)(n-2)m r_+^{-n}(m)\Big)\left(1-\frac{|\D u|^2}{\max_{\pa N}|\D u|^2}\right)\bigg]\,\rmd\sigma\,.
\end{multline*}
Moreover, since $r_+(m)$ satisfies $1-r_+^2(m)-2mr_+^{2-n}(m)=0$, we have
$$
1\,+\,2\,m\, r_+^{-n}(m)\,\,=\,\,r_+^{-2}(m)\,,
$$
hence the integral inequality above becomes
\begin{multline*}
\int_{\pa N}A\bigg[\RRR^{\pa N}-(n-1)(n-2)r_+^{-2}(m)\bigg]\,\rmd\sigma\,\geq
\\
\geq\,-\int_{\pa N}A\bigg[\Big((n+2)+2(n-1)(n-2)mr_+^{-n}(m)\Big)\left(1-A^2\right)\bigg]\,\rmd\sigma\,.
\end{multline*}
where $A$ is the function defined in~\eqref{eq:step_function_SD}. The thesis follows from this inequality with some straightforward algebra.
\end{proof}

\noindent 
Notice that $A\leq 1$ on $\pa N$ by definition, hence for $n\geq 4$ Theorem~\ref{cor:intid_rewr_+_SD} gives the following formula
\begin{equation*}
\int_{\pa N}\!\! A^3 
\,\rmd\sigma\,\leq\, 
\left(\int_{\pa N} 
\!\!A\, \frac{\RRR^{\pa N}}{(n-1)(n-2)} \,\rmd\sigma \right)  \, r_+^{2}(m)\,.
\end{equation*}

Instead, in dimension $n=3$, we can make Theorem~\ref{cor:intid_rewr_+_SD} more explicit by means of the Gauss-Bonnet formula.

\begin{theorem}
\label{thm:3dareabounds_+_SD}
Let $(M,\go,u)$ be a $3$-dimensional solution to problem~\eqref{eq:prob_SD} and let $N\subseteq M\setminus{\rm MAX}(u)$ be an outer region with virtual mass $m=\mu(N,\go,u)$. Then
$$
\frac{\sum_{i=0}^{p}\bigg[\Big(\frac{\kappa_i}{\kappa_0}\Big)^2-\frac{3}{2} r_+^2(m)\Big(1-\Big(\frac{\kappa_i}{\kappa_0}\Big)^2\Big)\bigg]\kappa_i |S_i|}{\sum_{i=0}^{p} \kappa_i}\,\leq 4\pi r_+^2(m)
$$
where $\pa N=S_0\sqcup\cdots\sqcup S_p$ and $\kappa_0\geq \cdots\geq \kappa_p$ are the surface gravities of $S_0,\dots, S_p$.
Moreover, if the equality holds then $\pa N$ is connected and $(M,\go,u)$ is isometric to the Schwarzschild--de Sitter solution with mass $m$.
\end{theorem}

\begin{proof}
For $n=3$, the formula in Corollary~\ref{cor:intid_rewr_+_SD} rewrites as
\begin{equation*}
\sum_{i=0}^p\int_{S_i}\kappa_i\left[\RRR^{S_i}-2r_+^{-2}(m)+\left[5+4mr_+^{-3}(m)\right]\left(1-\frac{\kappa_i^2}{\kappa_0^2}\right)\right]\,\rmd\sigma\,\geq \,0\,.
\end{equation*}
Since $1-r_+^2(m)-2mr_+^{-1}(m)=0$ by definition, we compute $5+4mr_+^{-3}(m)=3+2r_+^{-2}(m)$. Moreover, from the Gauss-Bonnet formula, we have $\int_{S_i}\RRR^{S_i}\rmd\sigma=4\pi\chi(S_i)$ for all $i=0,\dots,p$. From~\cite[Theorem~B]{Ambrozio}, we also know that each $S_i$ is diffeomorphic to a sphere, hence $\chi(S_i)=2$. Substituting these pieces of information inside the above formula, with some manipulations we arrive to the thesis.
\end{proof}

The local formula proven in Theorem~\ref{thm:3dareabounds_+_SD} may be compared with Theorem~\ref{thm:area-bound_Ambrozio} by Ambrozio~\cite{Ambrozio}. Although our result has the virtue of being sharp for the Schwarzschild--de Sitter solutions, the formula that we obtain is much more cumbersome. On the other hand, our results become particularly nice when the boundary $\pa N$ is connected. In fact, in this case, 
the constancy of the quantity $|\D u|$ on the whole boundary allows to obtain the following stronger results.

\begin{corollary}
\label{cor:Gauss-Codazzi_+_SD}
Let $(M,\go,u)$ be a solution to problem~\eqref{eq:prob_SD} and let $N\subseteq M\setminus{\rm MAX}(u)$ be an outer region with virtual mass $m$.
If $\pa N$ is connected, then it holds
\begin{equation*}
|\pa N|\,\leq\,\left(\int_{\pa N}
\frac{\RRR^{\pa N}}{(n-1)(n-2)}\,\rmd\sigma\right)\, r_+^{2}(m)\,.
\end{equation*}
Moreover, if the equality is fulfilled, then the solution $(M,\go,u)$ is isometric to a generalized Schwarzschild--de Sitter triple~\eqref{eq:gen_SD} with mass $m$.
\end{corollary}
	
\begin{proof}
The result is an immediate consequence of Corollary~\ref{cor:intid_rewr_+_SD} and the fact that $|\D u|$ is constant on $\pa N$.
\end{proof}	
	
\begin{theorem}
\label{thm:Gauss-Bonnet_+_SD}
Let $(M,\go,u)$ be a $3$-dimensional solution to problem~\eqref{eq:prob_SD} and let $N\subseteq M\setminus{\rm MAX}(u)$ be an outer region with virtual mass $m$.
If $\pa N$ is connected, then $\pa N$ is diffeomorphic to $\Sph^2$ and it holds
\begin{equation*}
|\pa N|\,\leq\,4\pi r_+^{2}(m)\,.
\end{equation*}
Moreover, if the equality holds, then the solution $(M,\go,u)$ is isometric to the Schwarzschild--de Sitter triple~\eqref{eq:SD} with mass $m$.
\end{theorem}

\begin{proof}
	Substituting $n=3$ in Corollary~\ref{cor:Gauss-Codazzi_+_SD} and using the Gauss-Bonnet formula, we immediately obtain
	\begin{equation*}
	4\pi\chi(\pa N)\,\geq\, 2r_+^{-2}(m)
	\,|\pa N|\,.
	\end{equation*}
	In particular, $\chi(\pa N)$ has to be positive, hence $\pa N$ is necessarily a sphere and we obtain the thesis.
\end{proof}

\subsection{Area bounds for inner regions.}
\label{sub:cons_integral_identities_-_SD}
Here we proceed as in Subsection~\ref{sub:cons_integral_identities_+_SD} to prove analogous integral identities when $N$ is an inner region.

\begin{theorem}
\label{cor:intid_rewr_-_SD}
Let $(M,\go,u)$ be a solution to problem~\eqref{eq:prob_SD} and let $N\subseteq M\setminus{\rm MAX}(u)$ be an inner region with virtual mass $m=\mu(N,\go,u)$.
Then it holds
\begin{equation*}
\int_{\pa N}\!\! A^3 
\,\rmd\sigma\,\leq\, 
\left(\int_{\pa N} 
\!\!A\, \frac{\RRR^{\pa N}}{(n-1)(n-2)} \,\rmd\sigma \right)  \, r_-^{2}(m)\,
\, - \, \frac{n[n-(\a+3)]}{(n-1)(n-2)} \left(\int_{\pa N} \!\!A\,\left( 1- A^2 \right)  
\,\rmd\sigma \right)  \, r_-^{2}(m)\,.
\end{equation*}
where $\RRR^{\pa N}$ is the scalar curvature of the metric induced by $\go$ on $\pa N$, $\a\geq 1$ is the solution of~\eqref{eq:alpha_SD} and $A$ is the step function defined in~\eqref{eq:step_function_SD}.
Moreover, if the equality holds, then the solution $(M,\go,u)$ is isometric to a generalized Schwarzschild--de Sitter triple~\eqref{eq:gen_SD} with mass $m$.
\end{theorem}

\begin{proof}
As in the proof of Corollary~\ref{cor:intid_rewr_+_SD}, it is enough to translate formula~\eqref{eq:int_id_g_2_SD} in terms of $u$ and $\go$, using the relations developed in Subsection~\ref{sub:conformal reformulation_SD}. Again one starts by noticing that
$$
|\na\ffi|_g^2\,=\,\frac{\dot\psi^2}{u^2}\,|\D u|^2\,,\qquad\hbox{and}\qquad \max_{\pa N}|\na \ffi|_g\,=\,1\,,
$$
where the second identity follows from Lemma~\ref{le:bound_psi_SD}. In particular we have
$$
\Big(\frac{\dot{\psi}^2}{u^2}\Big)_{|_{\pa N}}\max_{\pa N}|\D u|^2\,=\,1\,,
$$
which in turn implies $|\na\ffi|^2_g=|\D u|^2/\max_{\pa N}|\D u|^2$.
Translating also $\Ricg$ in terms of $\Ric$, with similar computations to the ones done in the proof of Corollary~\ref{cor:intid_rewr_+_SD}, from formula~\eqref{eq:int_id_g_SD} we obtain
\begin{multline*}
\int\limits_{\pa N}|\D u|\bigg[-\Ric(\nu,\nu)
-(n-1)\frac{u}{\psi\dot{\psi}}
+
\left(\frac{\a+2}{2}n+(n-1)\frac{u}{\psi\dot{\psi}}\right)
\left(1-\frac{|\D u|^2}{\max_{\pa N}|\D u|^2}\right)\bigg]\,\rmd\sigma\,\geq\, 0\,,
\end{multline*}
where we recall that the equality holds if and only if the solution is isometric to the Schwarzschild--de Sitter solution). 
We remark that the above formula is not completely rigorous, as the quantity $u/(\psi\dot\psi)$ is not defined on $\pa N$, because $\dot\psi\to 0$ as $u\to 0$. However, from formula~\eqref{eq:dpsideu_SD} that quantity can be explicitated as
$$
\frac{u}{\psi\dot\psi}\,=\,-\left[1-(n-2)m\psi^{-n}\right]\,,
$$
which has a finite value on the boundary, as $\psi=r_-(m)$ on $\pa N$ (because $N$ is an inner region and $\psi$ is defined as specified in~\eqref{eq:pr_function_-}).
Moreover, using the Gauss-Codazzi equation we have $2\Ric(\nu,\nu)=\RRR-\RRR^{\pa N}=n(n-1)-\RRR^{\pa N}$, hence we can rewrite the above inequality as
\begin{multline*}
\int_{\pa N}|\D u|\bigg[\RRR^{\pa N}-(n-1)(n-2)\big(1+2mr_-^{-n}(m)\big)\bigg]\,\rmd\sigma\,\geq
\\
\geq\,-\int_{\pa N}\bigg[\Big(\a n+2+2(n-1)(n-2)mr_-^{-n}(m)\Big)\left(1-\frac{|\D u|^2}{\max_{\pa N}|\D u|^2}\right)\bigg]\,\rmd\sigma\,.
\end{multline*}
We also recall that $r_-(m)$ satisfies $1-r_-^2(m)-2mr_-^{2-n}(m)=0$, hence we easily compute
$$
1\,+\,2\,m\, r_-^{-n}(m)\,\,=\,\,r_-^{-2}(m)\,.
$$
Substituting in the integral inequality above we obtain
\begin{multline*}
\int_{\pa N}A\bigg[\RRR^{\pa N}-(n-1)(n-2)r_-^{-2}(m)\bigg]\,\rmd\sigma\,\geq
\\
\geq\,\int_{\pa N}\bigg[\Big((\a n+2)+2(n-1)(n-2)mr_-^{-n}(m)\Big)\left(1-A^2\right)\bigg]\,\rmd\sigma\,\,,
\end{multline*}
where $A$ is the step function defined in~\eqref{eq:step_function_SD}.
The thesis follows with some easy algebra.
\end{proof}

\noindent 
Since $A\leq 1$ on $\pa N$ by definition, for $n\geq 4$ Theorem~\ref{cor:intid_rewr_-_SD} gives the following formula
\begin{equation*}
\int_{\pa N}\!\! A^3 
\,\rmd\sigma\,\leq\, 
\left(\int_{\pa N} 
\!\!A\, \frac{\RRR^{\pa N}}{(n-1)(n-2)} \,\rmd\sigma \right)  \, r_-^{2}(m)\,.
\end{equation*}

Concerning dimension $n=3$, the above result can be made more explicit using the Gauss-Bonnet formula.

\begin{theorem}
\label{thm:3dareabounds_-_SD}
Let $(M,\go,u)$ be a $3$-dimensional solution to problem~\eqref{eq:prob_SD} and let $N\subseteq M\setminus{\rm MAX}(u)$ be an inner region with virtual mass $m=\mu(N,\go,u)$. Then
$$
\frac{\sum_{i=0}^{p}\bigg[\Big(\frac{\kappa_i}{\kappa_0}\Big)^2-\frac{3}{2}\a r_-^2(m)\Big(1-\Big(\frac{\kappa_i}{\kappa_0}\Big)^2\Big)\bigg]\kappa_i |S_i|}{\sum_{i=0}^{p} \kappa_i}\,\leq 4\pi r_-^2(m)
$$
where $\a\geq 1$ is the solution to equation~\eqref{eq:alpha_SD}, $\pa N=S_0\sqcup\cdots\sqcup S_p$ and $\kappa_0\geq \cdots\geq \kappa_p$ are the surface gravities of $S_0,\dots, S_p$.
Moreover, if the equality holds then $\pa N$ is connected and $(M,\go,u)$ is isometric to the Schwarzschild--de Sitter solution with mass $m$.
\end{theorem}

\begin{proof}
	For $n=3$, the formula in Corollary~\ref{cor:intid_rewr_-_SD} rewrites as
	\begin{equation*}
	\sum_{i=0}^p\int_{S_i}\kappa_i\left[\RRR^{S_i}-2r_-^{-2}(m)+\left(3\a +2+4mr_-^{-3}(m)\right)\left(1-\frac{\kappa_i^2}{\kappa_0^2}\right)\right]\,\rmd\sigma\,\geq \,0\,.
	\end{equation*}
	Since $1-r_-^2(m)-2mr_-^{-1}(m)=0$ by definition, we compute $3\a+2+4mr_-^{-3}(m)=3\a+2r_-^{-2}(m)$. Moreover, from the Gauss-Bonnet formula, we have $\int_{S_i}\RRR^{S_i}\rmd\sigma=4\pi\chi(S_i)$ for all $i=1,\dots,p$. From~\cite[Theorem~B]{Ambrozio}, we also know that each $S_i$ is diffeomorphic to a sphere, hence $\chi(S_i)=2$. Substituting this information inside the above formula, with some manipulations we arrive to the thesis.
\end{proof}

\noindent
As in the outer case, when $\pa N$ is connected, the constancy of the quantity $|\D u|$ on the whole boundary allows to obtain stronger results.

\begin{corollary}
\label{cor:Gauss-Codazzi_-_SD}
Let $(M,\go,u)$ be a solution to problem~\eqref{eq:prob_SD} and let $N\subseteq M\setminus{\rm MAX}(u)$ be an inner region with virtual mass $m$.
If $\pa N$ is connected, then it holds
\begin{equation*}
|\pa N|\,\leq\,\left(\int_{\pa N}
\frac{\RRR^{\pa N}}{(n-1)(n-2)}\,\rmd\sigma\right)\, r_-^{2}(m)\,.
\end{equation*}
Moreover, if the equality is fulfilled, then the solution $(M,\go,u)$ is isometric to a generalized Schwarzschild--de Sitter triple~\eqref{eq:gen_SD} with mass $m$.
\end{corollary}

\begin{proof}
The result is an immediate consequence of Corollary~\ref{cor:intid_rewr_-_SD} and the fact that $|\D u|$ is constant on $\pa N$.
\end{proof}

\begin{theorem}
\label{thm:Gauss-Bonnet_-_SD}
Let $(M,\go,u)$ be a $3$-dimensional solution to problem~\eqref{eq:prob_SD} and let $N\subseteq M\setminus{\rm MAX}(u)$ be an inner region with virtual mass $m$. If $\pa N$ is connected, then $\pa N$ is diffeomorphic to $\Sph^2$ and it holds
\begin{equation*}
|\pa N|\,\leq\,4\pi r_-^{2}(m)\,.
\end{equation*}
Moreover, if the equality holds, then the solution $(M,\go,u)$ is isometric to the Schwarzschild--de Sitter triple~\eqref{eq:SD} with mass $m$.
\end{theorem}

\begin{proof}
Substituting $n=3$ in Corollary~\ref{cor:Gauss-Codazzi_-_SD} and using the Gauss-Bonnet formula, we immediately obtain
\begin{equation*}
4\pi\chi(\pa N)\,\geq\, 2r_-^{-2}(m)
\,|\pa N|\,.
\end{equation*}
In particular, $\chi(\pa N)$ has to be positive, hence $\pa N$ is necessarily a sphere and we obtain the thesis.
\end{proof}

\section{Black Hole Uniqueness Theorem}
\label{sec:ass_SD}




In this section we will prove the Black Hole Uniqueness Theorem~\ref{thm:BHU3D} stated in Subsection~\ref{sub:BHU_intro}, in the case where $m_+<\mmax$. The case $m_+=\mmax$ requires a different analysis, as the model solution will be the Nariai triple~\eqref{eq:cylsol_D}, and it will be studied in Section~\ref{sec:nariai_SD}. 
The hypothesis $m_+<\mmax$ allows us to use the metric $g$ and the functions $\Psi,\ffi$ defined  in the previous sections by formul\ae~\eqref{eq:g_SD},~\eqref{eq:def_Psi_SD},~\eqref{eq:ffi_SD}. We recall the definitions here, for the reader convenience. The function $\Psi=\psi\circ u$ is defined as
\begin{equation*}
\begin{split}
\Psi:\,M &\,\longrightarrow\, \left[r_-(m),r_+(m)\right]
\\
p&\,\longmapsto\, \Psi(p):=
\begin{dcases}
\psi_+(u(p))   & \hbox{ if } p\in M_+\,,
\\
\psi_-(u(p))   & \hbox{ if } p\in M_-\,,
\\
r_0(m) & \hbox{ if } p\in {\rm MAX}(u)\,,
\end{dcases}
\end{split}
\end{equation*}
where we recall that $\psi_+:M_+\to [r_+(m),r_0(m)]$ and $\psi_-:M_-\to [r_0(m),r_-(m)]$ are defined implicitly as the two determinations of the equation
$$
u^2=1-\psi^2-2m\psi^{2-n}\,.
$$
In turn, the metric $g$ and the function $\ffi$ are defined as
$$
\ffi(p)\,=\,\int_{\Psi(p)}^{r_+(m)} \frac{dt}{t \sqrt{1-t^2-2mt^{2-n}}}
$$
$$
g=\frac{g_0}{\Psi^2}\,.
$$
We start by stating a lemma on the regularity of the pseudo-radial function along the hypersurface separating the two regions $M_+$ and $M_-$.

\begin{lemma}
\label{le:C2_2sided}
Let $(M,\go,u)$ be a $2$-sided solution to problem~\eqref{eq:prob_SD}, let $\Psi$ be the global pseudo-radial function defined by~\eqref{eq:pr_function_global} with respect to a parameter $m\in[0,\mmax)$ and let $g,\ffi$ be defined by~\eqref{eq:g_SD} and~\eqref{eq:ffi_SD}. Then, at each point in the top stratum of $\Sigma$, we have that $g$, $\Psi$ and $\ffi$ are $\mathscr{C}^3$ and that $|\na\ffi|_g=1$.
\end{lemma}

\begin{proof}
Proposition~\ref{pro:C2} tells us that $\Psi$ is $\mathscr{C}^3$ at each point of the top stratum of $\Sigma$. The regularity of $g$ and $\ffi$ follows immediately from their definition. Finally, recalling formula~\eqref{eq:grad_u_nearSigma}, we get
$$
|\na\ffi|_g^2\,=\,\frac{|\D u|^2}{\psi^2\left[1-(n-2)m\psi^{-n}\right]^2}\,=\,1\,.
$$
This concludes the proof.
\end{proof}

Lemma~\ref{le:C2_2sided} will play an important role in the proof of the following theorem. which will be a crucial step in the proof of Theorem~\ref{thm:BHU3D}.


\begin{theorem}
\label{thm:estimates_SD}
Let $(M,\go,u)$ be a $3$-dimensional $2$-sided solution to problem~\eqref{eq:prob_SD}, and let $\Sigma\subseteq {\rm MAX}(u)$ be the stratified hypersurface separating $M_+$ and $M_-$. Let also
$$
m_+\,=\,\mu(M_+,\go,u)\,,\qquad \hbox{and} \qquad m_-\,=\,\mu(M_-,\go,u)
$$
be the virtual masses of $M_+$ and $M_-$, respectively. 
Then
$$
m_-\,\,\leq\,\, m_+\,.
$$
Moreover, if $m_+=m_-$ then $\Sigma$ is a $\mathscr{C}^\infty$ hypersurface and it holds
\begin{align}
\HHH\,\,&=\,2\,\sqrt{m_+^{-2/3}\,-\,3}\,,
\\
\hhh\,\,&=\,\sqrt{m_+^{-2/3}\,-\,3}\,\,\go^\Sigma\,,
\\
\label{eq:estimates_SD}
\RRR^\Sigma\,\,&=\,\,2\,m_+^{-2/3}\,,
\\
\Ric(\nu,\nu)\,\,&=\,\,0\,,
\end{align}
where $\nu$ is the $\go$-unit normal to $\Sigma$ pointing towards $M_+$, $\HHH$ and $\hhh$ are the mean curvature and second fundamental form of $\Sigma$ with respect to $\nu$, $\RRR^\Sigma$ is the scalar curvature of the metric $\go^\Sigma$ induced on $\Sigma$ by $\go$.
\end{theorem}

\begin{proof}
If $m_+=\mmax$ the first part of the statement is trivial, whereas the second part will be proved in Section~\ref{sec:nariai_SD}. Therefore, from now on we focus on the case $m_+<\mmax$.
To prove the first part of the statement let us suppose that $m_+<m_-$ and then deduce from this a contradiction. Fix a value $m\in[m_+,m_-]$ and define the global pseudo-radial function $\Psi=\psi\circ u$ with respect to this parameter.
We recall from Lemma~\ref{le:C2_2sided} that the function $\Psi$ is $\mathscr{C}^3$ in a neighborhood of the points in the top stratum of $\Sigma$. In turn, also the pseudo-affine function $\ffi$ defined by~\eqref{eq:ffi_SD} is $\mathscr{C}^3$ in a neighborhood of $\Sigma$, and so is the metric $g=\go/\Psi^2$. In particular, the scalar curvature $\Rg$ is continuous, and from formula~\eqref{eq:tildeR_SD} we deduce that 
$$
\lim_{x\to p,\ x\in M_+}\dot{\psi}(1-|\na\ffi|_g^2)=\lim_{x\to p,\ x\in M_-}\dot{\psi}(1-|\na\ffi|_g^2)
$$ 
for every $p$ in the top stratum of $\Sigma$.
As observed in Remark~\ref{rem:bound_psi_SD}, following the proof of Lemma~\ref{le:bound_psi_SD} it is easily seen that, since $m\in [m_+,m_-]$, it holds
$$
|\na\ffi|_g\,=\,\left|\frac{\D u}{\psi(1-m\psi^{-3})}\right|\,\leq\,1
$$
on the whole boundary $\pa M=\pa M_+\sqcup\pa M_-$. We can then apply the Minimum Principle to the elliptic inequality~\eqref{eq:Degw_SD} on $M_+$ and $M_-$, as we have done in Proposition~\ref{pro:min_pr_SD}. This proves that $|\na\ffi|_g\leq 1$ on the whole $M\setminus{\rm MAX}(u)$.
Furthermore, we recall that $\dot{\psi}$ has positive sign on $M_-$ and negative sign on $M_+$. Therefore $\dot{\psi}(1-|\na\ffi|_g^2)$ has to change sign when passing through $\Sigma$, hence 
$$
\lim_{x\to p}\dot{\psi}(1-|\na\ffi|_g^2)=0$$
for every $p$ in the top stratum of $\Sigma$.
In particular, $\Deg\ffi=0$ and  $|\na\ffi|_g=1$ on $\Sigma$.
Moreover, $|\na\ffi|_g$ has a maximum on $\Sigma$, hence $\na|\na\ffi|_g^2=0$ on $\Sigma$. In particular, $\nana\ffi(\nu_g,\nu_g)=\langle\na\ffi\,|\,\na|\na\ffi|_g^2\rangle_g/|\na\ffi|_g^2=0$, where $\nu_g=\na\ffi/|\na\ffi|_g=\na\ffi$ is the $g$-unit normal vector field to $\Sigma$.
At the points in the top stratum of $\Sigma$, the mean curvature $\Hg$ of $\Sigma$ with respect to $g$ and the normal $\nu_g$ can be computed using formul\ae~\eqref{eq:formula_fundamentalform_SD} and~\eqref{eq:formula_curvature_SD}.
Using the fact that $\Deg\ffi=\nana\ffi(\nu_g,\nu_g)=0$ on $\Sigma$, from~\eqref{eq:formula_curvature_SD} we deduce 
\begin{equation}
\label{eq:estimates1_SD}
\Hg\,=\,0\,,
\end{equation}
on $\Sigma$. 
Translating~\eqref{eq:estimates1_SD} in terms of $\go$ using~\eqref{eq:formula_H_H_g_SD}, and recalling that $|\na\ffi|_g=|\dot\psi/u|\,|\D u|=1$ on $\Sigma$, we obtain 
\begin{equation}
\label{eq:H_Sigma}
\HHH\,=\,2\,\frac{\umax(m)}{r_0(m)}\,=\,2\,\sqrt{m^{-2/3}-3}\,.
\end{equation} 
Notice that the formula for $\HHH$ depends on the parameter $m$, which can be chosen freely in the range $[m_+,m_-]$. But this is a contradiction, as the value of $\HHH$ cannot vary depending on $m$ but depends only on the geometry of $\Sigma$. This proves that $m_+$ cannot be smaller than $m_-$.

To prove the second part of the statement, let us define $\Psi,\ffi,g$ with respect to the parameter $m_+=m_-$. 
We have already observed that $\Psi,\ffi$ and $g$ are $\mathscr{C}^3$ in a neighborhood of the points of the top stratum of $\Sigma$. 
We also recall that, as computed above, on the top stratum of $\Sigma$ we have $\Deg\ffi=\nana\ffi(\nu_g,\nu_g)=\HHH_g=0$.
Now we choose a point $p$ in the top stratum of $\Sigma$ and we consider an embedding
$$
F_0:\overline{B^{2}}\to M
$$ 
such that $\Sigma_0=F_0(\overline{B^{2}})$ is contained in the top stratum of $\Sigma$.
We know from Lemma~\ref{le:C2_2sided} that $|\na\ffi|_g=1$ at each point in the top stratum of $\Sigma$.
Given a small enough number $\ep>0$ we can use the gradient $\na\ffi$ to extend the map $F_0$ to  a cartesian product $[-\ep,\ep]\times\overline{B^{2}}$, obtaining a new map 
$$
F:[-\ep,\ep]\times \overline{B^{2}}\hookrightarrow M\,, \qquad (s,\theta^1,\theta^2)\mapsto F(s,\theta^1,\theta^{2})\,,
$$
satisfying the initial value problem
$$
\frac{dF}{ds}\,=\,\frac{\na\ffi}{|\na\ffi|_g^2}\circ F\,,\quad F(0,\cdot)\,=\,F_0(\cdot)\,.
$$
It is not hard to check that $\ffi(F(s,\theta^1,\theta^2))=\ffi_0+s$, so that, for all $s\in[-\ep,\ep]$, the image $\Sigma_s=F(s,\overline{B^{2}})$ belongs to the level set $\{\ffi=\ffi_0+s\}$. Let us denote by $C_\ep$ the cylinder $F([-\ep,\ep]\times\overline{B^{2}})$.
Integrating Bochner's formula on $C_\ep$, we obtain
\begin{multline*}
\int_{C_\ep} |\nana\ffi|^2_g\rmd\mu_g\,=\,\int_{C_\ep}\left[\frac{1}{2}\Deg|\na\ffi|_g^2\,-\,\Ricg(\na\ffi,\na\ffi)\,-\,\langle\na\ffi\,|\,\na\Deg\ffi\rangle_g\right]\rmd\mu_g
\\
=\frac{1}{2}\int_{\pa C_\ep}\langle\na|\na\ffi|_g^2\,|\,{\rm n}_g\rangle_g\rmd\sigma_g\,-\,\int_{C_\ep}\Ricg(\na\ffi,\na\ffi)\rmd\mu_g\,+\,\int_{C_\ep}(\Deg\ffi)^2\rmd\mu_g\,-\,\int_{\pa C_\ep}\Deg\ffi\langle\na\ffi\,|\,{\rm n}_g\rangle_g\rmd\sigma_g
\end{multline*}
where ${\rm n}_g$ is the $g$-unit outward normal vector field to $\pa C_\ep$ and in the second equality we have integrated by parts.
We can rewrite the above formula as follows:
\begin{multline}
\label{eq:Bochner_Sigma_SD}
\frac{1}{2\ep}\int_{C_\ep}\left[ |\nana\ffi|^2_g\,-\,(\Deg\ffi)^2+\Ricg(\na\ffi,\na\ffi)\right]\rmd\mu_g\,=
\\
=\frac{1}{4\ep}\int_{\pa C_\ep}\left[\langle\na|\na\ffi|_g^2\,|\,{\rm n}_g\rangle_g-2\Deg\ffi\langle\na\ffi\,|\,{\rm n}_g\rangle_g\right]\rmd\sigma_g\,.
\end{multline}
Let us study in more details the right hand side of this formula. First of all, we notice that the integrand goes to zero as we approach $\Sigma_0$.
We also know that ${\rm n}_g\,=\,\na\ffi/|\na\ffi|_g$ on $\Sigma_\ep=F(\ep,\overline{B^{2}})$ and  ${\rm n}_g\,=\,-\na\ffi/|\na\ffi|_g$ on $\Sigma_{-\ep}=F(-\ep,\overline{B^{2}})$. Moreover, from the second equation in~\eqref{eq:pb_conf_SD} it follows that $\Deg\ffi$ is positive on $M_-$ and negative on $M_+$, so that in particular $\Deg\ffi>0$ on $\Sigma_\ep$ and $\Deg\ffi< 0$ on $\Sigma_{-\ep}$.
Concerning the function $\langle\na|\na\ffi|^2_g\,|\,\na\ffi\rangle_g$, we first notice that it is differentiable since $\ffi,g$ are $\mathscr{C}^3$. We now distinguish two cases: either its gradient $\na\langle\na|\na\ffi|^2_g\,|\,\na\ffi\rangle_g$ is zero in $p$ or it is not. If its gradient is zero, this means that the function $\langle\na|\na\ffi|^2_g\,|\,\na\ffi\rangle_g$ goes to zero at the first order as we approach $p$, which in turn implies that $\langle\na|\na\ffi|^2_g\,|\,\na\ffi\rangle_g=o(\ep)$. If instead $\na\langle\na|\na\ffi|^2_g\,|\,\na\ffi\rangle_g(p)\neq 0$, then up to restricting $C_\ep$ we can assume $\na\langle\na|\na\ffi|^2_g\,|\,\na\ffi\rangle_g\neq 0$ on the whole $C_\ep$. This implies that inside $C_\ep$ the level sets of $\langle\na|\na\ffi|^2_g\,|\,\na\ffi\rangle_g$ form a foliation of regular hypersurfaces, the zero level set corresponding to $\Sigma_0$. Since $|\na\ffi|_g$ assumes its maximum value $1$ on $\Sigma_0$, it is easily seen that $\langle\na|\na\ffi|^2_g\,|\,\na\ffi\rangle_g\leq 0$ on $M_-$ (in particular on $\Sigma_\ep$) and $\langle\na|\na\ffi|^2_g\,|\,\na\ffi\rangle_g\geq 0$ on $M_+$ (in particular on $\Sigma_{-\ep}$). In other words, we have $\langle\na|\na\ffi|^2_g\,|\,{\rm n}_g\rangle_g\leq 0$ on $\Sigma_\ep$ and on $\Sigma_{-\ep}$.
In particular, in both cases we have shown that the integrand on the right hand side over $\Sigma_\ep\cup\Sigma_{-\ep}$ is bounded from above by a function going to zero faster than $\ep$ as we approach $p$.
We also notice that the $\mathscr{H}^{n-1}$-measure of $\pa C_\delta\setminus (\Sigma_{-\ep}\cup \Sigma_\ep)$ is of the order of $\ep$. Therefore, using the coarea formula on the left hand side, equation~\eqref{eq:Bochner_Sigma_SD} gives
\begin{equation}
\label{eq:Bochner_Sigma_2_SD}
\frac{1}{2\ep}\int_{-\ep}^{\ep}\left[\int_{\Sigma_s}\frac{1}{|\na\ffi|_g}\left(|\nana\ffi|^2_g\,-\,(\Deg\ffi)^2+\Ricg(\na\ffi,\na\ffi)\right)\rmd\sigma_g\right]ds\,<\,\delta\,,
\end{equation}
where $\delta\to 0$ as $\ep\to 0$. 
Since $\ffi$ is $\mathscr{C}^3$, we can use the mean value property on the function
$$
s\mapsto\int_{\Sigma_s}\frac{1}{|\na\ffi|_g}\left(|\nana\ffi|^2_g\,-\,(\Deg\ffi)^2+\Ricg(\na\ffi,\na\ffi)\right)\rmd\sigma_g
$$
to deduce that there exists a value $\chi\in(-\ep,\ep)$ such that
\begin{multline*}
\int_{\Sigma_\chi}\frac{1}{|\na\ffi|_g}\left(|\nana\ffi|^2_g\,-\,(\Deg\ffi)^2+\Ricg(\na\ffi,\na\ffi)\right)\rmd\sigma_g\,=
\\
=\,\frac{1}{2\ep}\int_{-\ep}^{\ep}\left[\int_{\Sigma_s}\frac{1}{|\na\ffi|_g}\left(|\nana\ffi|^2_g\,-\,(\Deg\ffi)^2+\Ricg(\na\ffi,\na\ffi)\right)\rmd\sigma_g\right]ds\,<\,\delta\,.
\end{multline*}
Recalling that $\Deg\ffi=\nana\ffi(\nu_g,\nu_g)=0$ on $\Sigma$, from the first equation in~\eqref{eq:pb_conf_SD} we find $\Ricg(\nu_g,\nu_g)=0$ on $\Sigma$.
Therefore, taking the limit as $\ep\to 0$ of the above inequality, we get
$$
\int_{\Sigma_0}|\nana\ffi|_g^2\,\rmd\sigma_g\,\leq\,0\,.
$$ 
It follows that $\nana\ffi\equiv 0$ on $\Sigma_0$, which we recall is a neighborhood of $p$ in $\Sigma$. Therefore, from formula~\eqref{eq:formula_h_h_g_SD} we also deduce that $\hg\equiv 0$ at each point in the top stratum of $\Sigma$. 
Since the points in the top stratum of $\Sigma$ are dense in $\Sigma$, it follows that $\hg\equiv 0$ at each point where $\hg$ is well defined, that is, at each point where $\Sigma$ is a $\mathscr{C}^2$ hypersurface. Since later we will show that $\Sigma$ is $\mathscr{C}^\infty$, {\em a posteriori} we will have that $\hg\equiv 0$ on the whole $\Sigma$.

Substituting formula~\eqref{eq:H_Sigma} in~\eqref{eq:formula_h_h_g_SD}, we also find $0=|\hg|_g^2=m_+^{2/3}|\mathring{\hhh}|^2$. Therefore, $\mathring{\hhh}=0$ and it follows
$$
\hhh\,=\,\frac{\HHH}{2}\,\go^\Sigma\,=\,\sqrt{m^{-2/3}-3}\,\go^\Sigma\,.
$$
Now we pass to compute the scalar curvature of $\Sigma$. We have proven above that $\Ricg(\nu_g,\nu_g)=|\hg|_g=\Hg=0$ on the top stratum of $\Sigma$. Moreover, from~\eqref{eq:tildeR_SD} we have $\Rg=2$ on $\Sigma$. Therefore, from the Gauss-Codazzi equation we find
\begin{align}
\label{eq:estimates2_SD}
\Rg^\Sigma\,&=\,\Rg-2\Ricg(\nu_g,\nu_g)-|\hg|_g^2+\Hg^2\,=\,2\,.
\end{align}
Noticing that $\Rg^\Sigma=m_+^{2/3}\,\RRR^\Sigma$, where $\RRR^{\Sigma}$ is the scalar curvature of the metric induced by $\go$ on $\Sigma$, from identity~\eqref{eq:estimates2_SD} we obtain
\begin{equation*}
\RRR^\Sigma\,=\,m_+^{-2/3}\,\Rg^\Sigma\,=\,2\,m_+^{-2/3}\,.
\end{equation*}
Finally, recalling that $\DD u(\nu,\nu)=-3\umax$ on $\Sigma$, we obtain
$$
\Ric(\nu,\nu)\,=\,\frac{\DD u(\nu,\nu)}{u}\,+\,3\langle\nu\,|\,\nu\rangle\,=\,0\,.
$$
This concludes the proof of the formul\ae\ stated in Theorem~\ref{thm:estimates_SD}.

It remains to show that, under the hypothesis $m_+=m_-$, the hypersurface $\Sigma$ is necessarily $\mathscr{C}^\infty$. We start by recalling from Proposition~\ref{pro:SigmaC1} that $\Sigma$ is a $\mathscr{C}^1$ hypersurface, so that in particular its unit normal vector is defined everywhere. Let us start by computing the derivative of the normal vector $\nu$ at a point $p$ of the top stratum of $\Sigma$. Let $\nu,X_2,X_3$ be an orthonormal basis of $T_p M$. Differentiating the identity $|\nu|^2=1$ we deduce that 
\begin{align*}
0\,&=\,\langle\D_{\nu} \nu\,|\,\nu\rangle\,,
\\
0\,&=\,\langle\D_{X_i} \nu\,|\,\nu\rangle\,=\,\langle\D_\nu {X_i}\,|\,\nu\rangle\,=\,-\langle {X_i}\,|\,\D_\nu \nu\rangle\,,\quad\hbox{ for } i=2,3\,,
\end{align*}
at each point in the top stratum of $\Sigma$. This shows that $\D_\nu \nu=0$. Moreover, from our previous computations we get
$$
\langle\D_{X_i}\nu\,|\,X_j\rangle\,=\,\hhh(X_i,X_j)\,=\,\sqrt{m^{-2/3}-3}\,\delta_{ij}\,.
$$ 
Now that we know the components of $\D\nu$ on the top stratum, since the points in the top stratum are dense in $\Sigma$, it is clear that the limit of $\D\nu$ exists when we approach every point of $\Sigma$. It follows that the normal vector is differentiable, that is, $\Sigma$ is $\mathscr{C}^2$. Differentiating again the formul\ae
$$
\langle\D_{X_i}\nu\,|\,\nu\rangle\,=\,0\,,\quad\langle\D_{X_i}\nu\,|\,X_j\rangle\,=\,\sqrt{m^{-2/3}-3}\,\delta_{ij}\,,\quad\hbox{ and }\quad \D_{\nu}\nu\,=\,0\,,
$$
we easily get $\DD\nu\equiv 0$ on the top stratum. From this it follows that $\D^k\nu\equiv 0$ on the top stratum for every $k\geq 2$, hence the limit of all the derivatives of $\nu$ exist when we approach every point of $\Sigma$. This proves that the normal vector is smooth, which in turn implies that $\Sigma$ is $\mathscr{C}^{\infty}$.
\end{proof}

\noindent 
The next result follows combining Theorem~\ref{thm:estimates_SD} with Corollary~\ref{le:in_mon_Phi1_SD}, in order to obtain lower bound on $|\pa M_+|$. 

\begin{proposition}
\label{pro:BHG_n>3_SD}
Let $(M,\go,u)$ be a $2$-sided solution to problem~\eqref{eq:prob_SD}, and let $\Sigma\subseteq {\rm MAX}(u)$ be the stratified hypersurface separating $M_+$ and $M_-$. Let also
$$
m_+\,=\,\mu(M_+,\go,u)\,,\qquad m_-\,=\,\mu(M_-,\go,u)
$$
be the virtual masses of $M_+$ and $M_-$. 
Suppose $m_+=m_-<\mmax$.
Then it holds
$$
\int_{\Sigma}\frac{\RRR^\Sigma}{2}\,\rmd\sigma\,=\,
\frac{|\Sigma|}{m_+^{2/3}}\,\leq\,
\frac{|\pa M_+|}{r_+^2(m_+)}\,.
$$
Moreover, if the equality holds in the latter inequality, then $(M,\go,u)$ is isometric to a generalized Schwarzschild--de Sitter solution~\eqref{eq:gen_SD} with mass $m_+=m_-$.
\end{proposition}

\begin{proof}
The proof is just a collection of the previous results. From formula~\eqref{eq:estimates_SD}, we get
$$
m_+^{2/3}\int_{\Sigma}\frac{\RRR^\Sigma}{2}\rmd\sigma\,=\,|\Sigma|\,.
$$
We also recall that $|\na\ffi|_g\to 1$ as we approach $\Sigma$, as proven in Theorem~\ref{thm:estimates_SD} above.
Therefore, from Corollary~\ref{le:in_mon_Phi1_SD} we deduce
$$
|\Sigma|_g\,\,\leq\,\,|\pa M_+|_g\,,
$$
where we recall that the metric $g$ is defined by $g=\go/\Psi^2$. In particular, it holds
$$
|\pa M_+|_g\,=\,\frac{|\pa M_+|}{r_+^{2}(m_+)}\,,\qquad |\Sigma|_g\,=\,\frac{|\Sigma|}{m_+^{2/3}}\,.
$$
Putting together these formul\ae\ we easily obtain the thesis.
\end{proof}

\noindent 
If we also assume the hypothesis that $\pa M_+$ is connected, we can use Corollary~\ref{cor:Gauss-Codazzi_+_SD} to obtain a bound from above on $\pa M_+$. Combining this bound with the bound from below given by Proposition~\ref{pro:BHG_n>3_SD}, we obtain the chain of inequalities
\begin{equation}
\label{eq:BHG_n>3_SD}
\int_{\Sigma}\frac{\RRR^\Sigma}{2}\,\rmd\sigma\,\leq\,
\frac{|\pa M_+|}{r_+^{2}(m_+)}\,\leq\,
\int_{\pa M_+}\frac{\RRR^{\pa M_+}}{2}\,\rmd\sigma\,.
\end{equation}
Combining this inequality with the Gauss-Bonnet formula, we obtain the following uniqueness theorem.
\begin{theorem}
\label{thm:mon_glob_3_SD}
Let $(M,\go,u)$ be a $3$-dimensional $2$-sided solution to problem~\eqref{eq:prob_SD}, and let $\Sigma\subseteq {\rm MAX}(u)$ be the stratified hypersurface separating $M_+$ and $M_-$. Let also
$$
m_+\,=\,\mu(M_+,\go,u)\,,\qquad m_-\,=\,\mu(M_-,\go,u)
$$
be the virtual masses of $M_+$ and $M_-$. If the following conditions are satisfied
\begin{itemize}
\smallskip
\item \underline{mass compatibility}\qquad\qquad\qquad\quad $\,\,m_+= m_-<\mmax$,
\smallskip
\item \underline{connected cosmological horizon}\qquad $\pa M_+$ is connected,
\end{itemize}
then $(M,\go,u)$ is isometric to 
the Schwarzschild--de Sitter triple~\eqref{eq:SD} with mass $m_+=m_-$.
\end{theorem}

\begin{proof}
The chain of inequalities~\eqref{eq:BHG_n>3_SD} tells us that
$$
\int_\Sigma\RRR^\Sigma\,\rmd\sigma\,\leq\,\int_{\pa M_+}\RRR^{\pa M_+}\,\rmd\sigma\,,
$$
and the equality holds if and only if $(M,\go,u)$ is isometric to the Schwarzschild--de Sitter solution~\eqref{eq:SD}.
Applying the Gauss-Bonnet formula to both sides of the above inequality, we obtain
\begin{equation*}
4\pi\sum_{i=1}^k\chi(\Sigma_i)\,\leq\,4\pi\chi(\pa M_+)\,.
\end{equation*}
We recall from Theorem~\ref{thm:Gauss-Bonnet_+_SD} that if $\pa M_+$ is connected then $\pa M_+$ is diffeomorphic to a sphere, hence we get
\begin{equation}
\label{eq:bonnet_rigidity}
\sum_{i=1}^k\chi(\Sigma_i)\,\leq\,2\,,
\end{equation} 
where $\Sigma_1,\dots,\Sigma_k$ are the connected components of $\Sigma$. Moreover, the equality holds in~\eqref{eq:bonnet_rigidity} if and only if $(M,\go,u)$ is equivalent to the Schwarzschild--de Sitter solution.

On the other hand, from formula~\eqref{eq:estimates_SD} we get
$$
\RRR^\Sigma\,=\,{2}\,{m_+^{-2/3}}>0\,,
$$
hence 
$$
\sum_{i=1}^k\chi(\Sigma_i)\,=\,\int_{\Sigma}\RRR^\Sigma\rmd\sigma\,>\,0\,.
$$
Since $\Sigma$ is a separating surface, $\Sigma$ is necessarily orientable, therefore its Euler characteristic is necessarily an even number. Since $\chi(\Sigma)>0$, it follows $\chi(\Sigma)\geq 2$. Therefore, the equality must hold in~\eqref{eq:bonnet_rigidity} and this triggers the rigidity statement.
\end{proof}

\section{The cylindrical case}
\label{sec:nariai_SD}

In this section we deal with the case where the virtual mass of a region $N$ is equal to $\mmax$. We notice that the metric and the static potential of the Schwarzschild--de Sitter solution~\eqref{eq:SD} collapse as the mass $m$ approaches $\mmax$. Nevertheless, it is well known (see for instance~\cite{Gin_Per,Bousso,
Bou_Haw,Car_Dia_Lem}) that, if one rescales the static potential and the coordinates during the limit process in order to avoid singularities, then the limit of the Schwarzschild--de Sitter solution as the mass $m$ approaches $\mmax$ can be seen to be the Nariai triple~\eqref{eq:cylsol_D}. Therefore, in this section, the Nariai triple will play the role of the reference model. While the following computations are different from the ones shown in the preceding sections, the ideas and the conclusions will be analogue.

\begin{normalization}
According to the Nariai solution~\eqref{eq:cylsol_D}, throughout all this section, the static potential $u$ is normalized in such a way that $\umax:=\max_{M} (u)=1$.
\end{normalization}

\subsection{Conformal reformulation.}
\label{sub:conformal_reformulation_N_SD}

Let $(M,g_0,u)$ be a solution to system~\eqref{eq:prob_SD}, and let $N$ be a connected component of $M\setminus{\rm MAX}(u)$ such that $\max_{\pa N}|\D u|=\sqrt{n}$. On $N$, consider the metric 
\begin{equation}
\label{eq:g_N_SD}
g=\frac{n}{n-2}\,g_0\,.
\end{equation}
We want to reformulate problem \eqref{eq:prob_SD} in terms of $g$. 
\begin{remark}
We notice that this conformal change is analogue to the conformal change~\eqref{eq:g_SD} (in fact, the value of the pseudo-radial function $\Psi$ defined in Section~\ref{sub:pseudo_radial} goes to $\sqrt{(n-2)/n}$ as $m\to\mmax$).
In this case, the conformal change~\eqref{eq:g_N_SD} is just a rescaling of the metric, hence it is not really necessary for the following analysis. However, we have preferred to introduce it, since it allows for an easier comparison between the following computations and the ones shown in the previous sections for $m\neq\mmax$.
\end{remark}

We fix local coordinates in $M$ and we denote by $\Gamma_{\alpha\beta}^{\gamma},G_{\alpha\beta}^{\gamma}$ the Christoffel symbols of $g,g_0$. It is clear that
$
\Gamma_{\alpha\beta}^{\gamma}=G_{\alpha\beta}^{\gamma}
$. 
Denote by $\nabla,\Delta_g$ the Levi-Civita connection and the Laplace-Beltrami operator of $g$. For every $z\in\mathscr{C}^{\infty}$, we compute
\begin{equation}
\label{eq:hessian_N_SD}
\nabla_{\alpha\beta}^2 z=\D_{\alpha\beta}^2 z
\end{equation}
\begin{equation}\label{eq:laplacian_N}
\Delta_g z=\frac{n-2}{n}\,\Delta z
\end{equation}
Moreover, since the Ricci tensor is invariant under rescaling, we have $\Ricg=\Ric$.
Consider now the function
\begin{equation}
\label{eq:ffi_N_SD}
\ffi\,=\,\frac{\arcsin (u)}{\sqrt{n-2}}\,.
\end{equation}
Since $u$ is normalized in such a way that $\umax=1$, the function $\ffi$ is well defined, is zero on $\pa N$ and goes to $\pi/(2\sqrt{n-2})$ when we approach ${\rm MAX}(u)$.
Moreover, the gradient and hessian of $\ffi$ satisfy the following identities
\begin{align}
\label{eq:|naffi|_N_SD}
|\na\ffi|_g^2\,&=\,\frac{1}{n}\,\frac{|\D u|^2}{1-u^2}\,,
\\
\label{eq:nanaffi_N_SD}
\nana\ffi\,&=\,\frac{1}{\sqrt{n-2}\sqrt{1-u^2}}\,\left[\DD u\,+\,\frac{u}{1-u^2}\,du\otimes du\right]\,.
\end{align}
Some more calculations show that, with respect to $(\ffi,g)$, the equations in \eqref{eq:prob_SD} rewrites in $N$ as

\begin{equation}
\label{eq:pb_conf_N_SD}
\begin{dcases}
\Ricg=\frac{\sqrt{n-2}}{\tan(\sqrt{n-2}\,\ffi)}\,\nana\ffi-(n-2)d\ffi\otimes d\ffi+(n-2)\,g, & \mbox{in } N\\
\Deg \ffi=-\sqrt{n-2}\,\tan(\sqrt{n-2}\ffi)\left(1-|\na\ffi|_g^2\right), & \mbox{in } N\\
\ffi=0, & \mbox{on } \partial N\\
\ffi=\ffi_0:=\frac{\pi}{2\sqrt{n-2}} & \mbox{on } \overline{N}\cap {\rm MAX}(u).
\end{dcases}
\end{equation}
We observe that, since $g$ is just a rescaling of $\go$, we have $\Ricg=\Ric$. In particular the scalar curvature of $g$ is constant and more precisely
\begin{equation}
\label{eq:tildeR_N_SD}
\Rg=(n-1)(n-2)\,.
\end{equation}
We can also prove the analogue of Proposition~\ref{pro:rigidity}.

\begin{proposition}
\label{pro:rigidity_N}
Let $(M,\go,u)$ be a solution to problem~\eqref{eq:prob_SD}, and let $N$ be a cylindrical region.
Let $g=[n/(n-2)]\go$ and let $\ffi$ be the pseudo-affine function defined by~\eqref{eq:ffi_N_SD}.

If $\nana\ffi\equiv 0$ and $|\na\ffi|_g\equiv 1$ on $N$, then $(M,\go,u)$ is covered by a generalized Nariai solution~\eqref{eq:gen_cylsol_D}.
\end{proposition}

\begin{proof}
Proceeding as in the proof of Proposition~\ref{pro:rigidity} one shows that $(N,\go,u)$ is isometric to a region $(M_+^n,\go^n,u^n)$ of a Nariai solution~\eqref{eq:cylsol_D}, that we denote in this proof as $(M^n,\go^n,u^n)$. We then distinguish two cases, depending on whether the hypersurface $\Sigma_N=\overline{N}\cap {\rm MAX}(u)$ is two-sided or one-sided.

\begin{itemize}
\item If $\Sigma_N$ is two sided, then one can proceed exactly as in Proposition~\ref{pro:rigidity} to prove that the isometry extends beyond $\Sigma_N$. Therefore, $(M,\go,u)$ is isometric to the Nariai solution~\eqref{eq:cylsol_D}.

\smallskip

\item If $\Sigma$ is one sided then, reasoning as in Proposition~\ref{pro:int_id_g_SD}, we have that $(M,\go,u)=(\overline{N},\go,u)$ is isometric to $(\overline{M}_+^n,\go^n,u^n)/\sim$, where $\sim$ is a relation on the points of 
$$
{\rm MAX}^n(u)\,=\,\{p\in M^n\,:\,u^n(p)=\umax\}\,\subset\,\pa\overline{M}_+^n\,.
$$
induced by an involution of ${\rm MAX}^n(u)$. Notice that ${\rm MAX}^n(u)$, with the metric induced by $\go^n$, is isometric to an Einstein manifold $(E,g_E)$, hence the relation $\sim$ gives rise to an isometric involution $\iota_\sim:E\to E$. But then one can check that
$$
(\overline{M}_+^n,\go^n,u^n)/\sim\,=\,(M^n,\go^n,u^n)/\iota
$$
where $\iota:M^n\to M^n$ is the involution defined, for any $(r,x)\in [0,\pi]\times E=M^n$, by
$$
\iota(r,x)\,=\,(\pi- r, \iota_{\sim}(x))\,.
$$
In particular, $(\overline{M}_+^n,\go^n,u^n)/\sim$ is covered by the Nariai solution~\eqref{eq:cylsol_D} with fiber $E$, and so the same holds for our initial manifold $(M,\go,u)=(\overline{N},\go,u)$.
\end{itemize} 
This concludes the proof.
\end{proof}

Proceeding as in Subsection~\ref{sub:geom_levelsets_SD}, from identity~\eqref{eq:nanaffi_N_SD} one can also prove the following formul\ae\ for the second fundamental form and mean curvature of a level set $\{\ffi=s\}$
\begin{equation}
\label{eq:formula_curvature_N_SD}
\hg\,|\na\ffi|_g\,=\,\frac{1}{\sqrt{n-2}}\,\frac{|\D u|}{\sqrt{1-u^2}}\,\hhh\,,\qquad \Hg\,|\na\ffi|_g\,=\,\frac{\sqrt{n-2}}{n}\,\frac{|\D u|}{\sqrt{1-u^2}}\,\HHH\,.
\end{equation}
Furthermore, starting from the Bochner formula and using the equations in~\eqref{eq:pb_conf_N_SD}, we find

\begin{multline}
\label{eq:bochner_N_SD}
\Deg|\na\ffi|_g^2\,-\,\sqrt{n-2}\,\bigg[\frac{1+2\,\tan^2(\sqrt{n-2}\,\ffi)}{\tan(\sqrt{n-2}\,\ffi)}\bigg]
\,\langle\na|\na\ffi|_g^2\,|\,\na\ffi\rangle_g
\,=
\\
=\,2|\nana\ffi|_g^2-\,2(n-2)\tan^2(\sqrt{n-2}\,\ffi)\,|\na\ffi|_g^2\,(1-|\na\ffi|_g^2)\,.
\end{multline}

Let $w=\beta(1-|\na\ffi|_g^2)$, where $\beta=\cos(\sqrt{n-2}\,\ffi)$.
With computations analogous to the ones shown in Subsection~\ref{sub:bochner_minpr_SD}, we arrive to the following equation
\begin{multline}
\label{eq:Degw_N_SD}
\Deg w\,-\,\,\frac{\sqrt{n-2}}{\tan(\sqrt{n-2}\,\ffi)}\langle \na\ffi\,|\,\na w\rangle\,
-\,(n-2)\,\tan^2(\sqrt{n-2}\,\ffi)\,\left[(n+2)|\na\ffi|_g^2+(n-2)\right]w
\,=\,
\\
=\,-2\,\cos(\sqrt{n-2}\,\ffi)\left[|\nana\ffi|_g^2-\frac{(\Deg\ffi)^2}{n}\right]\,\leq\,0\,.
\end{multline}
In particular, we can apply a Minimum Principle to find the following analogue of Proposition~\ref{pro:min_pr_SD}.

\begin{proposition}
\label{pro:min_pr_N_SD}
Let $(M,\go,u)$ be a solution to problem~\eqref{eq:prob_SD}, let $N$ be  a connected component of $M\setminus{\rm MAX}(u)$ with virtual mass $m=\mmax$, and let $g,\ffi$ be defined by,~\eqref{eq:g_N_SD},~\eqref{eq:ffi_N_SD}. 
Then 
$$
|\na\ffi|_g\leq 1
$$ 
on the whole $N$.

Moreover, if $|\na\ffi|_g=1$ at a point in the interior of $N$, then $|\na\ffi|_g\equiv 1$ on the whole $N$ and $(M,\go,u)$ is isometric to a generalized Nariai solution~\eqref{eq:gen_cylsol_D}.
\end{proposition}

\begin{proof}
The proof is completely analogue to the proof of Proposition~\ref{pro:min_pr_SD} for the case $m\neq\mmax$, so we will not give all the details.

Since $\max_{\pa N}|\D u|=\sqrt{n}$, from~\eqref{eq:|naffi|_N_SD} we deduce $w\geq 0$ on $\pa N$. 
Moreover, again from~\eqref{eq:|naffi|_N_SD}, and Lemma~\ref{le:estimate_upsi_SD}, we have that $w$ goes to zero as we approach ${\rm MAX}(u)$. In particular, since $\cos(\sqrt{n-2}\ffi)\to 0$ as $\ffi\to\ffi_0$, we have $w\to 0$ as we approach ${\rm MAX}(u)$.
In particular, for any $\ep>0$ we can find a small neighborhood $\Omega_\ep$ of ${\rm MAX}(u)$ such that $w\geq -\ep$ on $\pa(N\setminus \Omega_\ep)$. The thesis follows applying the Minimum Principle in $N\setminus \Omega_\ep$, and then letting $\ep$ and the volume of $\Omega_\ep$ go to zero.

Now we pass to the proof of the second part of the statement. Let $x$ be a point in the interior of $N$ such that $|\na\ffi|_g(x)=1$. In particular it holds $w(x)=0$ and we have proved above that $w\geq 0$ on the whole $N$. Applying the Strong Minimum Principle on an open set $\Omega$ containing $x$, we obtain $w\equiv 0$, or equivalently $|\na\ffi|_g\equiv 1$, on $\Omega$. From the arbitrariness of $\Omega$ we deduce $|\na\ffi|_g\equiv 1$ on $N$, and plugging this information inside the Bochner formula~\eqref{eq:bochner_N_SD}, we obtain $|\nana\ffi|_g\equiv 0$. We can now invoke Proposition~\ref{pro:rigidity_N} to conclude.
\end{proof}

Now we consider the function
\begin{equation}
\label{eq:Phi_N_SD}
\Phi(s)=\int_{\{\ffi=s\}}|\na\ffi|_g\,\rmd\sigma_g\,.
\end{equation}
which is defined on $s\in[0,\ffi_0]$, where we recall that we have set $\ffi_0=\pi/(2\sqrt{n-2})$.
Proceeding as in the proof of Proposition~\ref{pro:mon_Phi1_SD}, as an application of Proposition~\ref{pro:min_pr_N_SD} one can prove the following monotonicity result for $\Phi$.

\begin{proposition}
\label{pro:mon_Phi1_N_SD}
Let $(M,\go,u)$ be a solution to problem~\eqref{eq:prob_SD}, let $N\subseteq M\setminus{\rm MAX}(u)$ be a cylindrical region, and let $\Phi(s)$ be the function defined by~\eqref{eq:Phi_N_SD}, with respect to the metric $g$ and the pseudo-affine function $\ffi$ defined by~\eqref{eq:g_N_SD},~\eqref{eq:ffi_N_SD}.
Then the function $\Phi(s)$ is monotonically nonincreasing. 
Moreover, if $\Phi(s_1)=\Phi(s_2)$ for two different values $0\leq s_1<s_2<\ffi_0$, then the solution $(M,\go,u)$ is isometric to 
a generalized Nariai triple~\eqref{eq:gen_cylsol_D}.
\end{proposition}

%

From Proposition~\ref{pro:C2} and formula~\eqref{eq:grad_u_nearSigma_N} we also know that $|\na\ffi|_g$ goes to $1$ as we approach the points where ${\rm MAX}(u)$ is an analytic hypersurface. Proceeding as in the proof of Corollary~\ref{le:in_mon_Phi1_SD} we obtain the following.

\begin{corollary}
	\label{le:in_mon_Phi1_N_SD}
Let $(M,\go,u)$ be a solution to problem~\eqref{eq:prob_SD}, let $N\subseteq M\setminus{\rm MAX}(u)$ be a cylindrical region, and let $g$ and $\ffi$ be defined by~\eqref{eq:g_N_SD},~\eqref{eq:ffi_N_SD}. Let also $\Sigma_N=\overline{N}\cap\overline{M\setminus\overline{N}}$ be the hypersurface separating $N$ from the rest of the manifold $M$. Then
\begin{equation*}
|\pa N|_g\,\geq\,|\Sigma_N|_g \,.
\end{equation*}
Moreover, if the equality holds, 
then $(M,\go,u)$ is isometric to a generalized Nariai triple~\eqref{eq:gen_cylsol_D}.
\end{corollary}

In order to make use of Corollary~\ref{le:in_mon_Phi1_N_SD}, we need some information on the set ${\rm MAX}(u)$ and on the behavior of $\na\ffi$ at the limit $\ffi\to\ffi_0$. In Section~\ref{sub:mon_BHU_N_SD} we will see how to recover some more explicit information from Corollary~\ref{le:in_mon_Phi1_N_SD} in the case where our solution is $2$-sided according to Definition~\ref{def:2-sided}.

\subsection{Integral identities.}
\label{sub:integral_identities_N_SD}

Consider the vector field $Y=\na|\na\ffi|_g^2+\Deg\ffi\na\ffi$.
Starting from the Bochner formula~\eqref{eq:bochner_N_SD}, we easily compute
\begin{equation*}
{\rm div}_g (Y)\,-\,\sqrt{n-2}\,\left[\frac{1+3\tan^2(\sqrt{n-2}\ffi)}{\tan(\sqrt{n-2}\ffi)}\right]\langle\na\ffi\,|\,Y\rangle_g\,=\,2|\nana\ffi|_g^2+(\Deg\ffi)^2\,\geq\,0\,.
\end{equation*}
If we introduce the function
\begin{equation}
\label{eq:gamma_N_SD}
\gamma\,=\,\frac{\cos^3(\sqrt{n-2}\,\ffi)}{\sin(\sqrt{n-2}\,\ffi)}\,,
\end{equation}
the identity above can be rewritten as
\begin{equation}
\label{eq:divY_N_SD}
{\rm div}_g (\gamma \,Y)\,=\,\gamma\left[2\,|\nana\ffi|_g^2+(\Deg\ffi)^2\right]\,\geq\,0\,.
\end{equation}
As an application of the Divergence Theorem, we obtain the following result, which is the analogue of Propositions~\ref{pro:int_id_g_SD} and~\ref{pro:int_id_g_2_SD}.

\begin{proposition}
\label{pro:int_id_g_N_SD}
Let $(M,\go,u)$ be a solution to problem~\eqref{eq:prob_SD}, let $N\subseteq M\setminus{\rm MAX}(u)$ be a cylindrical region, and let $g$ and $\ffi$ be defined by~\eqref{eq:g_N_SD} and~\eqref{eq:ffi_N_SD}.
For any $0\leq s<\ffi_0$ it holds
\begin{multline}
\label{eq:int_id_g_N_SD}
	\int_{\pa N}|\na\ffi|_g\left[\Ricg(\nu_g,\nu_g)-\frac{3}{2}(n-2)(1-|\na\ffi|_g^2)\right]\rmd\sigma_g\,=
\\
=\,-\,\sqrt{n-2}\int_{N}\gamma\left[|\nana\ffi|_g^2+\frac{1}{2}(\Deg\ffi)^2\right]
\,\leq\,0\,.
\end{multline}
where $\gamma$ is the function defined by~\eqref{eq:gamma_N_SD}.
Moreover, if the equality
\begin{equation}
\label{eq:int_id_rig_g_N_SD}
\int_{\pa N}|\na\ffi|_g\left[\Ricg(\nu_g,\nu_g)-\frac{3}{2}(n-2)(1-|\na\ffi|_g^2)\right]\rmd\sigma_g\,=\,0\,,
\end{equation}
holds, then the solution $(M,\go,u)$ is covered by a generalized Nariai triple~\eqref{eq:gen_cylsol_D}.
\end{proposition}

\begin{proof}
Let us recall from Subsection~\ref{sub:prelim}
 that $u$ is an analytic function. In particular, also $\ffi$ is analytic in the interior of $N$, hence its critical level sets are discrete. It follows that we can choose $0<s<S<\ffi_0$, with $s$ arbitrarily close to $0$ and $S$ arbitrarily close to $\ffi_{0}$ such that both $s$ and $S$ are regular values for $\ffi$. 
Integrating ${\rm div}_g(\gamma Y)$ on $\{s\leq\ffi\leq S\}$ and using the Divergence Theorem we obtain
\begin{equation}
\label{eq:int_in_N_aux_SD}
\int\limits_{\{S\leq\ffi\leq s\}}\!\!\!{\rm div}_g(\gamma Y)\rmd\sigma_g
\,=\,
\int\limits_{\{\ffi= S\}}\!\!\!\gamma(S)\Big\langle Y\,\Big|\,\frac{\na\ffi}{|\na\ffi|_g}\Big\rangle_{\!g}\rmd\sigma_g
-\int\limits_{\{\ffi= s\}}\!\!\!\gamma(s)\Big\langle Y\,\Big|\,\frac{\na\ffi}{|\na\ffi|_g}\Big\rangle_{\!g}\rmd\sigma_g\,,
\end{equation}
First of all, we notice that it holds
\begin{equation}
\label{eq:limit_int_N}
\lim_{S\to\ffi_0}\gamma(S)\int_{\{\ffi= S\}}\Big\langle Y\,\Big|\,\frac{\na\ffi}{|\na\ffi|_g}\Big\rangle_{\!g}\rmd\sigma_g=0\,.
\end{equation}
In fact, using formul\ae~\eqref{eq:|naffi|_N_SD},~\eqref{eq:nanaffi_N_SD} and~\eqref{eq:pb_conf_N_SD} to translate the integrand in terms of $u,\go$, we find
\begin{align*}
\gamma\,\Big\langle Y\,\Big|\,\frac{\na\ffi}{|\na\ffi|_g}\Big\rangle_{\!g}\,&=\gamma\,\left(\frac{\langle\na|\na\ffi|_g^2\,|\,\na\ffi\rangle_g}{|\na\ffi|_g}\,+\,\Deg\ffi\,|\na\ffi|_g\right)
\\
&=\,\gamma\,|\na\ffi|_g\left[2\,\nana\ffi(\nu_g,\nu_g)+\Deg\ffi\right]
\\
&=\,\sqrt{\frac{n-2}{n}\,}\frac{\sqrt{1-u^2}}{u}|\D u|\bigg[\frac{2}{n}\left(\DD u(\nu,\nu)+\frac{u}{1-u^2}|\D u|^2\right)-\left(1-\frac{1}{n}\,\frac{|\D u|^2}{1-u^2}\right)\,\bigg]\,,
\end{align*}
where $\nu=\D u/|\D u|,\nu_g=\na\ffi/|\na\ffi|_g=\sqrt{(n-2)/n\,}\,\nu$ are the unit normals to $\{\ffi=S\}$ which exist everywhere because $\{\ffi=S\}$ is a regular level set.
Since $|\na\ffi|_g^2=(1/n)|\D u|^2/(1-u^2)\leq 1$ by Proposition~\ref{pro:min_pr_N_SD}, we deduce that the limit of the term in square bracket as $S\to\ffi_0$ (or equivalently $u\to 1$) is bounded from above. Therefore, in order to prove~\eqref{eq:limit_int_N}, it is enough to show that
$$
\lim_{u\to 1}\int_{\{u=t\}}(1-u^2)\,|\D u|\,\rmd\sigma\,=\,0\,.
$$
But this can be done proceeding exactly as in the proof of~\cite[Theorem~4.4]{Bor_Maz_2-I}, via a simple argument using the coarea formula and the fact that $u\to 1$ and $|\D u|\to 0$ (more precisely $|\D u|^2/(1-u^2)$ is bounded) as $u\to 1$. 
Therefore, taking the limit as $S\to\ffi_0$ of~\eqref{eq:int_in_N_aux_SD}, we deduce
\begin{equation}
\label{eq:int_in_aux2_N_SD}
\int_{\{\ffi= s\}}\gamma(s)\Big\langle Y\,\Big|\,\frac{\na\ffi}{|\na\ffi|_g}\Big\rangle_{\!g}\rmd\sigma_g
\,=\,
-\int_{\{s\leq\ffi<\ffi_0\}}{\rm div}_g(\gamma Y)\rmd\sigma_g\,\leq\,0\,,
\end{equation}
where in the last inequality we have used~\eqref{eq:divY_N_SD}.
Now we compute the integral on the left hand side. Using the equations in~\eqref{eq:pb_conf_N_SD}, we obtain
\begin{align}
\notag
\frac{1}{\tan(\sqrt{n-2}\,\ffi)}\Big\langle Y\,\Big|\,\frac{\na\ffi}{|\na\ffi|_g}\Big\rangle_{\!g}\,&=\,\frac{1}{\tan(\sqrt{n-2}\,\ffi)}\left[2\,\frac{\nana\ffi(\na\ffi,\na\ffi)}{|\na\ffi|_g}+\Deg\ffi|\na\ffi|_g\right]
\\
\label{eq:integrand_aux_N_SD}
&=\,|\na\ffi|_g\,\bigg[\frac{2}{\sqrt{n-2}}\,\Ricg(\nu_g,\nu_g)\,-\,3\,\sqrt{n-2}\,\left(1-|\na\ffi|_g^2\right)\bigg]\,.
\end{align}
Moreover, recalling the definition~\eqref{eq:gamma_N_SD} of $\gamma$, we find
\begin{equation}
\label{eq:limit_aux_N_SD}
\lim_{s\to 0}\,\big[\tan(\sqrt{n-2}\,\ffi)\,\gamma\big]_{|_{\{\ffi=s\}}}\,=\,\lim_{s\to 0}\left[\cos^2(\sqrt{n-2}\,\ffi)\right]_{|_{\{\ffi=s\}}}\,=\,1\,.
\end{equation}
Taking the limit of~\eqref{eq:int_in_aux2_N_SD} as $s\to 0$ and using~\eqref{eq:integrand_aux_N_SD} and~\eqref{eq:limit_aux_N_SD}, we obtain the desired inequality~\eqref{eq:int_id_g_N_SD}.

Concerning the rigidity statement, if the equality in~\eqref{eq:int_id_rig_g_N_SD} holds, then necessarily the right hand side of~\eqref{eq:int_id_g_N_SD} is null. In particular, $|\na\ffi|_g\equiv 1$ on $N$.
Substituting this information in the Bochner formula~\eqref{eq:bochner_N_SD} we obtain $|\nana\ffi|_g\equiv 0$, hence we can apply Proposition~\ref{pro:rigidity_N} to conclude.
\end{proof}

\subsection{Proof of the area bounds.}
\label{sub:consequences_N}

The area bounds for cylindrical regions is proven in the exact same way as in the outer and inner case discussed in Subsection~\ref{sub:pointwise_bounds}. Namely, one compares formula~\eqref{eq:Ambrozio_nearboundary} with the gradient estimate proven in Proposition~\ref{pro:min_pr_N_SD}, obtaining that the scalar curvature of $\pa N$ is necessarily greater that or equal to the one of the sections of the Nariai solution.

\begin{theorem}
Let $(M,\go,u)$ be a solution to problem~\eqref{eq:prob_SD} of dimension $n\geq 3$, and let $N\subseteq M\setminus{\rm MAX}(u)$ be a cylindrical region. Then
	\begin{equation}
	\RRR^{\pa N}\,\geq\, n(n-1)\,.
	\end{equation}
\end{theorem}

We pass now to discuss the consequences of Proposition~\ref{pro:int_id_g_N_SD} proved above. First of all, translating it in terms of $u$ and $g_0$, we obtain the following result.

\begin{corollary}
\label{cor:intid_rewr_N_SD}
Let $(M,\go,u)$ be a solution to problem~\eqref{eq:prob_SD} and let $N\subseteq M\setminus{\rm MAX}(u)$ be a cylindrical region.
Then it holds
\begin{equation*}
\int_{\pa N}A\left[\RRR^{\pa N}-n(n-1)+3n\left(1-A^2\right)\right]\,\rmd\sigma\,\geq \,0\,,
\end{equation*}
where $\RRR^{\pa N}$ is the scalar curvature of the metric induced by $\go$ on $\pa N$.
Moreover, if the equality holds, then the solution $(M,\go,u)$ is covered by
a generalized Nariai triple~\eqref{eq:gen_cylsol_D}.
\end{corollary}

\begin{proof}
It is enough to translate formula~\eqref{eq:int_id_g_N_SD} in terms of $u$ and $\go$, using the relations developed in Subsection~\ref{sub:conformal_reformulation_N_SD}. In particular, let us notice that
$$
|\na\ffi|_g^2\,=\,\frac{1}{n}\,\frac{|\D u|^2}{1-u^2}\,,\qquad\hbox{and}\qquad \max_{\pa N}|\na \ffi|_g\,=\,1\,,
$$
where the second identity follows from Lemma~\ref{le:bound_psi_SD}. Therefore
$$
\frac{1}{n}\max_{\pa N}|\D u|^2\,=\,1\,,
$$
which in turn implies $|\na\ffi|^2_g=|\D u|^2/\max_{\pa N}|\D u|^2$.
Since we have already observed that $\Ricg=\Ric$ and $\nu_g=\sqrt{(n-2)/n\,}\,\nu$, from formula~\eqref{eq:int_id_g_N_SD} we obtain
$$
\int_{\pa N}|\D u|\left[-\frac{1}{n}\Ric(\nu,\nu)
+\frac{3}{2}
\left(1-\frac{|\D u|^2}{\max_{\pa N}|\D u|^2}\right)\right]\,\rmd\sigma\,\geq\, 0\,,
$$
where we remark that the equality holds if and only if the solution is covered by the Nariai triple.

Using the Gauss-Codazzi equation we have $2\Ric(\nu,\nu)=\RRR-\RRR^{\pa N}=n(n-1)-\RRR^{\pa N}$. Substituting in the inequality above we easily obtain the thesis.
\end{proof}

In dimension $n=3$, the above formula can be made more explicit using the Gauss-Bonnet formula.
\begin{theorem}
\label{thm:3dareabounds_N_SD}
Let $(M,\go,u)$ be a $3$-dimensional solution to problem~\eqref{eq:prob_SD} and let $N\subseteq M\setminus{\rm MAX}(u)$ be a cylindrical region. Then
$$
\frac{\sum_{i=0}^{p}\bigg[\Big(\frac{\kappa_i}{\kappa_0}\Big)^2-\frac{1}{2}\Big(1-\Big(\frac{\kappa_i}{\kappa_0}\Big)^2\Big)\bigg]\kappa_i |S_i|}{\sum_{i=0}^{p} \kappa_i}\,\leq \frac{4\pi}{3} 
$$
where $\pa N=S_0\sqcup\cdots\sqcup S_p$ and $\kappa_0\geq \cdots\geq \kappa_p$ are the surface gravities of $S_0,\dots, S_p$.
Moreover, if the equality holds then $\pa N$ is connected and $(M,\go,u)$ is covered by the Nariai triple~\eqref{eq:cylsol_D}.
\end{theorem}

\begin{proof}
For $n=3$, the formula in Corollary~\ref{cor:intid_rewr_N_SD} rewrites as
\begin{equation*}
\sum_{i=0}^p\int_{S_i}\kappa_i\left[\RRR^{S_i}-6+9\left(1-\frac{\kappa_i^2}{\kappa_0^2}\right)\right]\,\rmd\sigma\,\geq \,0\,.
\end{equation*}
From the Gauss-Bonnet formula, we have $\int_{S_i}\RRR^{S_i}\rmd\sigma=4\pi\chi(S_i)$ for all $i=0,\dots,p$. From~\cite[Theorem~B]{Ambrozio}, we also know that each $S_i$ is diffeomorphic to a sphere, hence $\chi(S_i)=2$. Substituting these pieces of information inside formula above, with some manipulations we arrive to the thesis.
\end{proof}

In the case when $\pa N$ is connected, the constancy of the quantity $|\D u|$ on the whole boundary allows to obtain the following stronger results.

\begin{corollary}
\label{cor:Gauss-Codazzi_N_SD}
Let $(M,\go,u)$ be a solution to problem~\eqref{eq:prob_SD} and let $N\subseteq M\setminus{\rm MAX}(u)$ be a cylindrical region.
If $\pa N$ is connected, then it holds
\begin{equation*}
\int_{\pa N}
\RRR^{\pa N}\,\rmd\sigma\,\geq\, n(n-1)|\pa N|\,.
\end{equation*}
Moreover, if the equality holds, then the solution $(M,\go,u)$ is covered by
a generalized Nariai triple~\eqref{eq:gen_cylsol_D}.
\end{corollary}

\begin{proof}
The result is an immediate consequence of Corollary~\ref{cor:intid_rewr_N_SD} and the fact that $|\D u|$ is constant on $\pa N$.
\end{proof}	

\begin{theorem}
\label{thm:Gauss-Bonnet_N_SD}
Let $(M,\go,u)$ be a $3$-dimensional solution to problem~\eqref{eq:prob_SD} and let $N\subseteq M\setminus{\rm MAX}(u)$ be a cylindrical region.
If $\pa N$ is connected, then $\pa N$ is diffeomorphic to $\Sph^2$ and it holds
\begin{equation*}
|\pa N|\,\leq\,\frac{4\pi}{3}\,.
\end{equation*}
Moreover, if the equality holds, then the solution $(M,\go,u)$ is covered by 
the Nariai triple~\eqref{eq:cylsol_D}.
\end{theorem}

\begin{proof}
Substituting $n=3$ in Corollary~\ref{cor:Gauss-Codazzi_N_SD} and using the Gauss-Bonnet formula, we immediately obtain
\begin{equation*}
4\pi\chi(\pa N)\,\geq\, 6
\,|\pa N|\,.
\end{equation*}
In particular, $\chi(\pa N)$ has to be positive, hence $\pa N$ is necessarily a sphere and we obtain the thesis.
\end{proof}

We now pass to investigate the hypersurface $\Sigma_N$ that separates $N$ from the rest of the manifold. Combining the results of this section with Corollary~\ref{le:in_mon_Phi1_N_SD}, it is straightforward to obtain the following area bound.

\begin{theorem}
Let $(M,\go,u)$ be a solution to problem~\eqref{eq:prob_SD}, let $N$ be  a cylindrical region with smooth compact boundary $\pa N$. Let $\Sigma_N=\overline{N}\cap\overline{M\setminus\overline{N}}$ be the possibly stratified hypersurface separating $N$ from the rest of the manifold $M$. Then
\begin{equation}
\label{eq:areaboundsephy_N}
|\Sigma_N|\,\leq\,\int_{\pa N}\frac{\RRR^{\pa N}}{n(n-1)}\,\rmd\sigma\,,
\end{equation}
and, if the equality holds, 
then $(M,\go,u)$ is covered by a generalized Nariai triple~\eqref{eq:gen_cylsol_D}.
\end{theorem}

\begin{proof}
Let us study the case where $N$ is outer, the inner case being completely analogous.
From Corollary~\ref{le:in_mon_Phi1_N_SD}, recalling the definitions of $g,\ffi$, we get
$$
\left(\frac{n}{n-2}	\right)^{\!\frac{n-1}{2}}|\Sigma_N|\,=\,|\Sigma_N|_g\,\leq\, |\pa N|_g\,=\,\left(\frac{n}{n-2}\right)^{\!\frac{n-1}{2}}|\pa N|\,.
$$
Now we conclude using Corollary~\ref{cor:Gauss-Codazzi_N_SD}.
\end{proof}

In particular, in dimension $n=3$, applying Gauss Bonnet Theorem to the right hand side of formula~\eqref{eq:areaboundsephy_N}, we obtain the cylindrical case of Corollary~\ref{cor:lower_bound}, which we recall here for the reader's convenience.

\begin{corollary}
Let $(M,\go,u)$ be a $3$-dimensional solution to problem~\eqref{eq:prob_SD}, let $N\subseteq M\setminus{\rm MAX}(u)$ be a cylindrical region with connected boundary $\pa N$. Let $\Sigma_N=\overline{N}\cap\overline{M\setminus\overline{N}}$ be the possibly stratified hypersurface separating $N$ from the rest of the manifold $M$. 
Then
\begin{equation}
|\Sigma_N|\,\leq\,\frac{4\,\pi}{3}\,.
\end{equation}
Moreover, if the equality holds, 
then $(M,\go,u)$ is covered by a generalized Nariai triple~\eqref{eq:gen_SD}.
\end{corollary}

\subsection{Black Hole Uniqueness.}
\label{sub:mon_BHU_N_SD}
In this section we will complete the proof of Theorem~\ref{thm:BHU3D}, started in Theorem~\ref{thm:mon_glob_3_SD}, by discussing the missing case $m_+=\mmax$.
To this end, on $M=\overline M_+\cup \overline M_-$ we define the metric $g$ as in~\eqref{eq:g_N_SD}, and the function $\ffi$ as follows
\begin{equation}
\label{eq:ffi_glob_N_SD}
\ffi=
\begin{dcases}
\frac{\arcsin(u)}{\sqrt{n-2}}\,, & \mbox{ on }  M_+\,,
\\
\frac{\pi-\arcsin(u)}{\sqrt{n-2}}\,, & \mbox{ on }  M_-\,.
\end{dcases}
\end{equation}
The function $\ffi$ defined here is equal to $0$ on $\pa M_+$, it is equal to $\ffi_0=\pi/(2\sqrt{n-2})$ on $\Sigma=\overline{M}_+\cap\overline{M}_-$ and is equal to $\ffi_{\rm max}=\pi/\sqrt{n-2}$ on $\pa M_-$.
Moreover, it is easily checked that $\ffi,g$ satisfy the equations in~\eqref{eq:pb_conf_N_SD} on $M_+$ and $M_-$. In particular, the elliptic inequality~\eqref{eq:Degw_N_SD} holds on every connected component of $M_+$ and $M_-$, and this leads to the following global estimate for the gradient of $\ffi$ (which is defined a priori only on $M_-\cup M_+$ and not on $\Sigma$).

\begin{proposition}
\label{pro:min_pr_ass_N_SD}
Let $(M,\go,u)$ be a $2$-sided solution to problem~\eqref{eq:prob_SD} such that the virtual masses $m_+=\mu(M_+,\go,u)$, $m_-=\mu(M_-,\go,u)$ satisfy $m_+=m_-=\mmax$,
and let $g,\ffi$ be defined by~\eqref{eq:g_N_SD},~\eqref{eq:ffi_glob_N_SD}. Then $|\na\ffi|_g\leq 1$ on the whole $M\setminus{\rm MAX}(u)$.
\end{proposition}

\begin{proof}
The proof is an easy adjustment of the proof of Proposition~\ref{pro:min_pr_SD}. First of all, we notice that our function $\ffi$ satisfies formula~\eqref{eq:|naffi|_N_SD} hence, thanks to the assumption, we have
	$$
	|\na\ffi|_g\,=\,\frac{1}{n}|\D u|^2\,\leq\,1
	$$
	on the whole boundary $\pa M=\pa M_+\sqcup\pa M_-$. The thesis follows applying the Minimum Principle to the elliptic inequality~\eqref{eq:Degw_N_SD} on each connected component of $M_+$ and $M_-$.
\end{proof}

\noindent
A second important remark is that the regularity of $\sqrt{\umax-u}$ implies the regularity of $\ffi$.

\begin{proposition}
\label{pro:C2_N}
Let $(M,\go,u)$ be a $2$-sided solution to problem~\eqref{eq:prob_SD}, and let $\ffi$ be defined by~\eqref{eq:ffi_glob_N_SD}.
Then the function $\ffi$ is $\mathscr{C}^3$ in a neighborhood of each point in the top stratum of $\Sigma$.
\end{proposition}

\begin{proof}
	From the definition of $\ffi$, it is clear that it is enough to show that $\arcsin(u)$ is $\mathscr{C}^3$. This is an easy exercise of analysis starting from the expansion~\eqref{eq:expansion_u_final} for $u$ proven before.
\end{proof}

\noindent
As an easy consequence of the above results, we obtain the following analogue of Theorem~\ref{thm:estimates_SD}.

\begin{theorem}
\label{thm:estimates_N_SD}
Let $(M,\go,u)$ be a $3$-dimensional $2$-sided solution to problem~\eqref{eq:prob_SD}, and let $\Sigma\subseteq {\rm MAX}(u)$ be the stratified hypersurface separating $M_+$ and $M_-$. 
Suppose that the virtual masses of $M_+$ and $M_-$ satisfy
$$
\mu(M_+,\go,u)\,\,=\,\, \mu(M_-,\go,u)\,\,=\,\,\mmax\,.
$$
Then $\Sigma$ is a $\mathscr{C}^\infty$ hypersurface and it holds
\begin{align}
\HHH\,\,&=\,\,0\,,
\\
\hhh\,\,&=\,0\,,
\\
\label{eq:estimates_N}
\RRR^\Sigma\,\,&=\,\,6\,,
\\
\Ric(\nu,\nu)\,\,&=\,\,0\,,
\end{align}
where $\nu$ is the $\go$-unit normal to $\Sigma$ pointing towards $M_+$, $\HHH$ and $\hhh$ are the mean curvature and second fundamental form of $\Sigma$ with respect to $\nu$, $\RRR^\Sigma$ is the scalar curvature of the metric $\go^\Sigma$ induced on $\Sigma$ by $\go$.
\end{theorem}

\begin{proof}
This proof follows the scheme of the proof of Theorem~\ref{thm:estimates_SD}. Define $\ffi$ and $g$ as in~\eqref{eq:ffi_glob_N_SD} and~\eqref{eq:g_N_SD}, consider a point $p\in\Sigma$ and consider a neighborhood $\Omega\ni p$ such that $\Sigma\cap\Omega$ is contained in the top stratum of $\Sigma$. From Proposition~\ref{pro:C2_N} we know that $\ffi$ is $\mathscr{C}^3$ in $\Omega$. Therefore $\Deg\ffi$ is continuous in $\Omega$, thus from the second formula in~\eqref{eq:pb_conf_N_SD} we deduce that also $\tan(\ffi)(1-|\na\ffi|_g^2)$ can be extended to a continuous function along $\Sigma\cap\Omega$. We also notice that $|\na\ffi|_g\leq 1$ everywhere by Proposition~\ref{pro:min_pr_ass_N_SD}, whereas $\tan(\ffi)$ has positive sign on $M_+$ and negative sign on $M_-$. Therefore, $\tan(\ffi)(1-|\na\ffi|_g^2)$ has to change sign when passing through $\Sigma$, hence $\tan(\ffi)(1-|\na\ffi|_g^2)=0$ on $\Sigma\cap\Omega$.
In particular, $\Deg\ffi=0$ and  $|\na\ffi|_g=1$ on $\Sigma\cap\Omega$.
Furthermore, $|\na\ffi|_g$ has a maximum on $\Sigma$, hence $\na|\na\ffi|_g^2=0$ on $\Sigma\cap\Omega$. In particular, $\nana\ffi(\nu_g,\nu_g)=\langle\na\ffi\,|\,\na|\na\ffi|_g^2\rangle_g/|\na\ffi|_g^2=0$, where $\nu_g=\na\ffi/|\na\ffi|_g=\na\ffi$ is the $g$-unit normal vector field to $\Sigma$, and substituting in the first formula in~\eqref{eq:pb_conf_N_SD}, we obtain $\Ricg(\nu_g,\nu_g)=0$ on $\Sigma\cap\Omega$.
The second fundamental form $\hg$ and the mean curvature $\Hg$ of $\Sigma$ can be computed using formul\ae~\eqref{eq:formula_curvature_N_SD}. Since $\Deg\ffi=\nana\ffi(\nu_g,\nu_g)=0$ on $\Sigma\cap\Omega$, from~\eqref{eq:formula_curvature_N_SD} we deduce 
\begin{equation}
\label{eq:estimates1_N}
\Hg\,\,=\,\,0\,,
\end{equation}
on $\Sigma\cap\Omega$.
Proceeding exactly as in Proposition~\ref{thm:estimates_SD}, one also shows that $|\hg|\equiv 0$ and that $\Sigma$ is $\mathscr{C}^{\infty}$.
  
Moreover, from the Gauss-Codazzi equation we find
\begin{align}
\notag
\Rg^\Sigma\,&=\,\Rg-2\Ricg(\nu_g,\nu_g)-|\hg|_g^2+\Hg^2
\\
\notag
&=\,\Rg
\\
\label{eq:estimates2_N}
&=\, 2\,,
\end{align}
where in the last equality we have used from~\eqref{eq:tildeR_N_SD}.

Translating~\eqref{eq:estimates1_N} in terms of $\go$ recalling~\eqref{eq:formula_curvature_N_SD}, and using the fact that $|\na\ffi|^2_g=(1/3)\,|\D u|^2/(1-u^2)=1$ on $\Sigma$, we obtain 
$$
\HHH\,=\,0\,,\qquad |\mathring{\hhh}|^2\,=\,|\hhh|^2\,=\,\frac{1}{3}\,|\hg|_g^2\,=\,0\,.
$$
Finally, noticing that $\Rg^\Sigma=\RRR^\Sigma/3$, where $\RRR^{\Sigma}$ is the scalar curvature of the metric induced by $\go$ on $\Sigma$, from identity~\eqref{eq:estimates2_N} we obtain
\begin{equation*}
\mmax^{2/3}\left(\RRR^\Sigma+|\mathring{\hhh}|^2\right)\,=\,\frac{\RRR^\Sigma+|\mathring{\hhh}|^2}{3}\,
=\,\Rg^\Sigma+|\hg|_g^2\,=\,2\,.
\end{equation*}
This concludes the proof.
\end{proof}


\noindent
The next result follows combining Propositions~\ref{pro:mon_Phi1_N_SD},~\ref{thm:estimates_N_SD} and the results in Subsection~\ref{sub:consequences_N}. 

\begin{proposition}
\label{pro:BHG_n>3_N_SD}
Let $(M,\go,u)$ be a $3$-dimensional $2$-sided solution to problem~\eqref{eq:prob_SD}, and let $\Sigma\subseteq {\rm MAX}(u)$ be the stratified hypersurface separating $M_+$ and $M_-$. Suppose that the virtual masses of $M_+$ and $M_-$ satisfy
$$
\mu(M_+,\go,u)\,\,=\,\mu(M_-,\go,u)\,\,=\,\,\mmax\,.
$$
Then it holds
$$
\int_{\Sigma}\frac{\RRR^\Sigma}{6}\,\rmd\sigma\,=\,
|\Sigma|\,\leq\,
|\pa M_+|\,.
$$
Moreover, if the equality holds, then $(M,\go,u)$ is isometric to the Nariai solution~\eqref{eq:cylsol_D}.
\end{proposition}

\begin{proof}
The proof is just a collection of the previous results. From~\eqref{eq:estimates_N}, we immediately get
$$
\int_{\Sigma}\frac{\RRR^\Sigma}{6}\rmd\sigma\,=\,|\Sigma|\,.
$$ 
Since $\mu(M_+,\go,u)=\mmax$, we have $|\na\ffi|_g^2=({1}/{3})|\D u|^2\leq 1$ on $\pa M_+$.
Moreover, we recall from the proof of Proposition~\ref{thm:estimates_N_SD} that $|\na\ffi|_g$, where $g$ and $\ffi$ are defined by~\eqref{eq:g_N_SD} and~\eqref{eq:ffi_N_SD} as usual, goes to $1$ as we approach $\Sigma$. 
Therefore, from Corollary~\ref{le:in_mon_Phi1_N_SD} we obtain
$$
3\,|\Sigma|\,=\,|\Sigma|_g\,\leq\,|\pa M_+|_g\,=\,3\,|\pa M_+|
$$
This concludes the proof of the inequality. The rigidity statement follows from the corresponding rigidity statements in Proposition~\ref{pro:mon_Phi1_N_SD}.
\end{proof}
If we also assume that $\pa M_+$ is connected, then we can combine Proposition~\ref{pro:BHG_n>3_N_SD} with Corollary~\ref{cor:Gauss-Codazzi_N_SD} and we obtain the following inequality
\begin{equation}
\label{eq:BHG_n>3_N_SD}
\int_{\Sigma}\RRR^\Sigma\,\rmd\sigma\,\leq\,
\int_{\pa M_+}\RRR^{\pa M_+}\,\rmd\sigma\,.
\end{equation}
Combining this inequality with the Gauss-Bonnet formula we obtain the following result, which concludes the proof of the Black Hole Uniqueness Theorem~\ref{thm:BHU3D} by addressing the cylindrical case.

\begin{theorem}
\label{thm:mon_glob_3_N_SD}
Let $(M,\go,u)$ be a $3$-dimensional $2$-sided solution to problem~\eqref{eq:prob_SD}, and let $\Sigma\subseteq {\rm MAX}(u)$ be the stratified hypersurface separating $M_+$ and $M_-$. Let also
$$
m_+\,=\,\mu(M_+,\go,u)\,,\qquad m_-\,=\,\mu(M_-,\go,u)
$$
be the virtual masses of $M_+$ and $M_-$. If the following conditions are satisfied
\begin{itemize}
\smallskip
\item \underline{mass compatibility}\qquad\qquad\qquad\ \ \ $\;m_+= m_-=\mmax$,
\smallskip
\item \underline{connected cylindrical horizon}\quad\quad$\ \ \; \pa M_+$ is connected,
\smallskip
\end{itemize}
then $(M,\go,u)$ is isometric to the Nariai triple~\eqref{eq:cylsol_D}. 
\end{theorem}

\begin{proof}
Inequality~\eqref{eq:BHG_n>3_N_SD} tells us that
$$
\int_\Sigma\RRR^\Sigma\,\rmd\sigma\,\leq\,\int_{\pa M_+}\RRR^{\pa M_+}\,\rmd\sigma\,,
$$
and the equality holds if and only if $(M,\go,u)$ is isometric to the Nariai solution~\eqref{eq:cylsol_D}.
Recalling that $\Sigma$ has no conical singularities as proved in Theorem~\ref{thm:estimates_SD}, applying the Gauss-Bonnet formula to both sides of the above inequality, we obtain
\begin{equation*}
4\pi\sum_{i=1}^k\chi(\Sigma_i)\,\leq\,4\pi\chi(\pa M_+)\,.
\end{equation*}
We recall from Theorem~\ref{thm:Gauss-Bonnet_N_SD} that if $\pa M_+$ is connected then $\pa M_+$ is diffeomorphic to a sphere, hence we obtain
\begin{equation}
\label{eq:auxauxaux}
\sum_{i=1}^k\chi(\Sigma_i)\,\leq\,2\,,
\end{equation}
where the equality holds if and only if the solution is isometric to the Nariai solution.

On the other hand, in dimension $n=3$, formula~\eqref{eq:estimates_N} gives
$$
\RRR^\Sigma\,=\,6\,.
$$
In particular, again from the Gauss-Bonnet formula, it follows
$$
\sum_{i=1}^k\chi(\Sigma_i)\,=\,\frac{1}{4\pi}\int_{\Sigma}\RRR^{\Sigma}\rmd\sigma\,>\,0\,,
$$
but $\sum_{i=1}^k\chi(\Sigma_i)$ can only assume even integer values, hence $\sum_{i=1}^k\chi(\Sigma_i)\geq 2$. Therefore the equality holds in~\eqref{eq:auxauxaux}, as wished.
\end{proof}

\subsection*{Acknowledgements}
{\em The authors would like to thank L. Ambrozio, C. Arezzo, A. Carlotto, C. Cederbaum and P. T. Chru\'sciel for their interest in our work and for stimulating discussions during the preparation of the manuscript. The authors are members of the Gruppo Nazionale per l'Analisi Matematica, la Probabilit\`a e le loro Applicazioni (GNAMPA) of the Istituto Nazionale di Alta Matematica (INdAM) and are partially founded by the GNAMPA Project ``Principi di fattorizzazione, formule di monotonia e disuguaglianze geometriche''. The paper was partially completed during the authors' attendance to the program ``Geometry and relativity'' organized by the Erwin Schr\"{o}dinger International Institute for Mathematics and Physics (ESI).
}

\bibliographystyle{plain}

\bibliography{biblio}

\end{document}